\documentclass[a4paper,oneside,11pt]{amsart}
\usepackage{listings}
\usepackage{amsmath}
\usepackage{verbatim}
\usepackage{csquotes}
\usepackage{mathrsfs}
\usepackage[bookmarks=false]{hyperref}
\usepackage{nameref}
\usepackage{color}
\usepackage{xpatch}
\usepackage{chngcntr}
\usepackage{apptools}
\usepackage{accents}
\usepackage{enumitem}
\usepackage{cleveref}
\renewcommand\labelenumi{(\roman{enumi})}
\renewcommand\theenumi\labelenumi

\newtheorem{lem}{Lemma}
\newtheorem{prop}{Proposition}
\newtheorem{thm}{Theorem}

\newtheorem{cor}{Corollary}
\theoremstyle{definition}
\newtheorem{defn}{Definition}
\newtheorem{rem}{Remark}

\AtAppendix{\counterwithin{defn}{section}}
\AtAppendix{\counterwithin{prop}{section}}
\AtAppendix{\counterwithin{thm}{section}}
\AtAppendix{\counterwithin{lem}{section}}
\AtAppendix{\counterwithin{cor}{section}}

\newcounter{numl}
\newcommand{\labelnuml}{\textup{(\roman{numl})}}
\newenvironment{numlist}{\begin{list}{\labelnuml}%
{\usecounter{numl}\setlength{\leftmargin}{0pt}%
\setlength{\itemindent}{2\parindent}%
\setlength{\itemsep}{\smallskipamount}\def
\makelabel ##1{\hss \llap {\upshape ##1}}}}{\end{list}}

\DeclareSymbolFont{script}{U}{eus}{m}{n}
\DeclareSymbolFontAlphabet{\mathscr}{script}
\DeclareMathSymbol{\Wedge}{0}{script}{"5E}
\DeclareMathAlphabet{\mathrmsl}{OT1}{cmr}{m}{sl}
\newcommand{\R}{{\mathbb R}}

\newcommand{\Z}{{\mathbb Z}}

\newcommand{\Sph}{{\mathbb S}}

\newcommand{\T}{{\mathbb T}}

\newcommand{\cL}{{\mathcal L}}
\newcommand{\cO}{{\mathcal O}}

\newcommand{\grad}{\mathrm{grad}}

\newcommand{\tor}{{\mathfrak t}}

\newcommand{\Lab}{{\mathbf L}}
\renewcommand{\tt}{{\boldsymbol t}}
\newcommand{\pprim}{p^{\prime}}
\newcommand{\tprim}{t^{\prime}}

\newcommand{\Hess}{\mathop{\mathrm{Hess}}}

\newcommand{\bdelta}{\boldsymbol \delta}

\newcommand{\G}{\mathrm G}

\newcommand{\q}{{\mathrm Q}}
\newcommand{\xxi}{\boldsymbol \xi}
\newcommand{\ba}{\boldsymbol a}

\newcommand{\bH}{\boldsymbol H}
\newcommand{\bG}{\boldsymbol G}

\newcommand{\bA}{{\boldsymbol A}}
\newcommand{\vv}{\mathrm w}
\newcommand{\uu}{\mathrm v}
\newcommand{\PP}{\mathrm P}

\newcommand{\w}{{\mathrm u}}

\newcommand{\Hilb}{{\rm Hilb}}
\newcommand{\Autred}{{\rm Aut}_{\rm red}(X)}
\newcommand{\Tr}{{\rm Tr}}
\newcommand{\vol}{\mathrm{vol}}

\newcommand{\Ric}{{\rm Ric}}

\newcommand{\Scal}{{\rm Scal}}
\newcommand{\FS}{\omega_{\rm{FS}}}

\newcommand{\fs}{{\rm FS}}
\newcommand{\Ent}{{\rm Ent}}

\newcommand{\DF}{{\rm DF}}

\usepackage[left=3cm,top=3cm,right=3cm,bottom=3cm,bindingoffset=0.5cm]{geometry}
\makeatletter
\let\@@citation@@=\citation

\renewcommand{\citation}[1]{\@@citation@@{#1}%
\@for\@tempa:=#1\do{\@ifundefined{cit@\@tempa}%
  {\global\@namedef{cit@\@tempa}{}}{}}%
}
\makeatother

\usepackage{hyperref}

\makeatletter
\def\@lbibitem[#1]#2#3\par{%
  \@ifundefined{cit@#2}{}{\@skiphyperreftrue
  \H@item[%
    \ifx\Hy@raisedlink\@empty
      \hyper@anchorstart{cite.#2\@extra@b@citeb}%
        \@BIBLABEL{#1}%
      \hyper@anchorend
    \else
      \Hy@raisedlink{%
        \hyper@anchorstart{cite.#2\@extra@b@citeb}\hyper@anchorend
      }%
      \@BIBLABEL{#1}%
    \fi
    \hfill
  ]%
  \@skiphyperreffalse}%
  \if@filesw
    \begingroup
      \let\protect\noexpand
      \immediate\write\@auxout{%
        \string\bibcite{#2}{#1}%
      }%
    \endgroup
  \fi
  \ignorespaces
  \@ifundefined{cit@#2}{}{#3}}

\def\@bibitem#1#2\par{%
  \@ifundefined{cit@#1}{}{\@skiphyperreftrue\H@item\@skiphyperreffalse
  \Hy@raisedlink{%
    \hyper@anchorstart{cite.#1\@extra@b@citeb}\relax\hyper@anchorend
    }}%
  \if@filesw
    \begingroup
      \let\protect\noexpand
      \immediate\write\@auxout{%
        \string\bibcite{#1}{\the\value{\@listctr}}%
      }%
    \endgroup
  \fi
  \ignorespaces
  \@ifundefined{cit@#1}{}{#2}}
\makeatother

\begin{document}

\author[A. Lahdili]{Abdellah Lahdili}
\address{Abdellah Lahdili\\ D{\'e}partement de Math{\'e}matiques\\
UQAM\\ C.P. 8888 \\ Succursale Centre-ville \\ Montr{\'e}al (Qu{\'e}bec) \\
H3C 3P8 \\ Canada}
\email{lahdili.abdellah@gmail.com}

\title[Weighted cscK metrics and weighted K-stability]{K\"ahler metrics with constant weighted scalar curvature and weighted K-stability}
\date{\today}

\begin{abstract} We introduce a notion of a K\"ahler metric with constant weighted scalar curvature on a compact K\"ahler manifold $X$, depending on a fixed real torus $\T$ in the reduced group of automorphisms of $X$, and two smooth (weight) functions $\uu>0$ and $\vv$, defined on the momentum image (with respect to a given K\"ahler class $\alpha$ on $X$) of $X$ in the dual Lie algebra of $\T$. A number of natural problems in K\"ahler geometry, such as the existence of extremal K\"ahler metrics and conformally K\"ahler, Einstein--Maxwell metrics, or prescribing the scalar curvature on a compact toric manifold reduce to the search of K\"ahler metrics with constant weighted scalar curvature in a given K\"ahler class $\alpha$, for special choices of the weight functions $\uu$ and $\vv$.

We show that a number of known results obstructing the existence of constant scalar curvature K\"ahler (cscK) metrics can be extended to the weighted setting. In particular, we introduce a functional $\mathcal M_{\uu, \vv}$ on the space of $\T$-invariant K\"ahler metrics in $\alpha$, extending the Mabuchi energy in the cscK case, and show (following  the arguments in \cite{li,ST} in the cscK and extremal cases and \cite{lahdili2} in the case of conformally K\"ahler Einstein-Maxwell) that if $\alpha$ is Hodge, then constant weighted scalar curvature metrics in $\alpha$ are minima of $\mathcal M_{\uu,\vv}$. Motivated by the recent work \cite{Dervan, Dervan-Ross, Dyr1, Dyr2} in the cscK and extremal cases, we define a $(\uu,\vv)$-weighted Futaki invariant of a $\T$-compatible smooth K\"ahler test configuration associated to $(X, \alpha, \T)$, and show that the boundedness from below of the $(\uu,\vv)$-weighted Mabuchi functional $\mathcal M_{\uu, \vv}$ implies a suitable notion of a $(\uu,\vv)$-weighted K-semistability.

We illustrate our theory with specific computations on smooth toric varieties and on the toric fibre bundles introduced in \cite{ACGT}. As an application, we obtain a Yau--Tian--Donaldson type correspondence for $(\uu,\vv)$-extremal K\"ahler classes on  ${\mathbb P}^1$-bundles over products of compact Hodge cscK manifolds, thus extending  some of the results in \cite{ACGT3, AMT} to the $(\uu, \vv)$-weighted setting.
\end{abstract}
\maketitle

\section{Introduction}\label{sec-1}
In this paper, we define a notion of a (weighted) $\uu$-scalar curvature ${\rm Scal}_{\uu}(\omega)$,  associated to a K\"ahler metric  $\omega$ on a smooth compact complex manifold $X$, a real torus $\T$ in the reduced  group  ${\rm Aut}_{\rm red}(X)$ of automorphisms of $X$,  and  a positive smooth function  $\uu(p)$  defined over the image $\PP \subset \tor^*$ of $X$ under the moment map $m_{\omega}: X \to \tor^*$ of $\T$ with respect to $\omega$. Here $\tor^*$ stands for the dual vector space of the Lie algebra $\tor$ of $\T$ and $p$ for a point of $\tor^*$. Originally, we have identified in \cite{lahdili2} $\Scal_{\uu}(\omega)$ with a coefficient in the asymptotic expansion of a certain weighted Bergman kernel associated with a function $\uu$ (see \Cref{TYZ} for a precise statement) reminiscent to the celebrated results by Catlin \cite{catlin}, Ruan \cite{Ruan}, Tian \cite{Tian-Berg} and Zelditch \cite{Zeld} in the case $\uu=1$. Our main motivation in this paper for introducing and systematically studying the $\uu$-scalar curvature is the observation that the problem of finding a $\T$-invariant K\"ahler metric $\omega$ in a given K\"ahler class $\alpha$ on $X$, for which 
\begin{equation}\label{main}
{\rm Scal}_{\uu}(\omega) = c_{\uu,\vv}(\alpha) \vv(m_{\omega}), 
\end{equation}
where $\vv(p)$ is another given smooth function on $\PP$ and $c_{\uu,\vv}(\alpha)$ is a suitable real constant (depending only on $\alpha$, $\PP$, $\uu$ and $\vv$) englobes a number of problems in K\"ahler geometry of current interest, including the following well-studied cases:
\begin{enumerate}
\item\label{1i} Letting $\uu=\vv\equiv 1$, we obtain the Calabi problem of finding a K\"ahler metric of constant scalar curvature (cscK for short) in $\alpha$;
\item\label{ii} Letting  $\T$ be a maximal torus in ${\rm Aut}_{\rm red}(X)$, $\uu\equiv 1$ and $\vv(p)= \vv_{\rm ext}(p)$ be  a suitable affine-linear function on $\tor^*$, the solutions of \eqref{main} are the extremal K\"ahler metrics in the sense of Calabi~\cite{calabi} in $\alpha$;
\item If $X$ is a Fano manifold equipped with the K\"ahler class $\alpha=2\pi c_1(X)$, $\T$ a maximal torus in ${\rm Aut}_{\rm red}(X)$, $\uu(p)=e^{\langle\xi,p\rangle}$ for $\xi\in\tor$ and $\vv=\uu \vv_{\rm ext}$ for a suitable affine-linear function $\vv_{\rm ext}$, then solutions of \eqref{main} are gradiant K\"ahler-Ricci solitons on $X$ (see \cite{Eiji1});
\item\label{iii} Letting $\uu(p)= \big(\langle \xi, p\rangle + a\big)^{-2m+1}$ and $\vv(p) = \big(\langle \xi, p\rangle + a\big)^{-2m-1}$ for $\xi\in \tor$ and $a\in \R$ such that $\langle \xi, p\rangle + a>0$ over $\PP$, \eqref{main} describes the K\"ahler metrics in $\alpha$, which are conformal to Einstein--Maxwell metrics, see \cite{AM,LeB08,LeB0};
\item\label{iv} If $\alpha = 2\pi c_1(L)$ for an ample holomorphic line bundle $L$ over $X$, $\uu(p)= \big(\langle \xi, p\rangle + a\big)^{-m-1}$ and $\vv(p) = \big(\langle \xi, p\rangle + a\big)^{-m-3}$ for $\xi\in \tor$ and $a\in \R$ such that $\langle \xi, p\rangle + a>0$ over $\PP$, then \eqref{main} describes K\"ahler metrics on $X$ giving rise to extremal Sasaki metrics on the unit circle bundle associated to $L^*$, see \cite{AC};
\item\label{v} The search for extremal K\"ahler metrics, or more generally, prescribing the scalar curvature of a class of K\"ahler metrics on toric fibre-bundles given by the generalized Calabi anstaz \cite{ACGT} or on manifolds with free multiplicity \cite{donaldson-survey, R} reduces to finding solutions of \eqref{main} on the (toric) fibre. In this toric setting \eqref{main} is known as the {\it generalized  Abreu equation}, see \cite{LLS1, LSZ}.
\end{enumerate}

\bigskip
We shall refer to the solutions of \eqref{main} as {\it constant $(\uu,\vv)$-scalar curvature K\"ahler metrics} (or $(\uu,\vv)$-cscK for short) and our main thesis in this paper is that most of the known obstructions to the existence of cscK metrics extend naturally to the $(\uu,\vv)$-cscK case. Indeed, as we show in the Appendices A and B to this paper, some of our previous results in \cite{lahdili1, lahdili2} regarding \ref{iii} and \ref{iv} are just special cases of more general statements concerning $(\uu,\vv)$-cscK metrics, thus providing a more conceptual explanation for the arguments therein. In particular, there is a natural generalization ${\mathcal M}_{\uu,\vv}$ of the Mabuchi functional (see e.g.~\cite{Gauduchon, Sz-book}) on the space of $\T$-invariant K\"ahler metrics in $\alpha$ (which we call the $(\uu,\vv)$-{\it Mabuchi energy} and define in \Cref{sec-5} below) and we show that the arguments of \cite[Thm.~1]{lahdili2} (which in turn build on \cite{li,ST}) yield the following

\begin{thm}\label{thm:mabuchi-bounded} Let $(X, L)$ be a compact smooth polarized projective variety, $\T \subset {\rm Aut}_{\rm red}(X)$ a real torus, and suppose that $X$ admits a $(\uu,\vv)$-cscK metric $\omega$ in $\alpha=2\pi c_1(L)$ for some smooth functions $\uu>0$ and $\vv$ on the momentum image $\PP\subset \tor^*$ associated to $(\T, \alpha)$. Then, $\omega$ is a global minima of the $(\uu,\vv)$-Mabuchi energy ${\mathcal M}_{\uu,\vv}$ of $(X, \T, \alpha, \PP, \uu, \vv)$.
\end{thm}

Instead of \eqref{main}, one can more generally consider the condition
\begin{equation}\label{ext}
{\rm Scal}_{\uu}(\omega) = \vv(m_{\omega})(m_{\omega}^{\xi} + c)
\end{equation}
for a $\T$-invariant K\"ahler metric $\omega$ in $\alpha$, where $\xi \in \tor$, $c\in \R$ and $m_{\omega}^{\xi}:= \langle m_{\omega}, \xi \rangle$ is the Killing potential associated to $\xi$. A $\T$-invariant K\"ahler metric satisfying \eqref{ext} generalizes the notion of an extremal K\"ahler metric (see \ref{ii} above), and will be referred to as a {\it $(\uu,\vv)$-extremal} K\"ahler metric. As it is apparent from the example \ref{ii} above, and as we establish more generally in \Cref{sec-3}, when $\vv>0$ the smooth function $(m_{\omega}^{\xi} + c)$ in the RHS of \eqref{ext} must be of the form $\vv_{\rm ext} (m_{\omega})$ for an affine-linear function $\vv_{\rm ext}(p)= \langle \xi, p \rangle + c$ on $\tor^*$ defined in terms of  $(\T, \alpha, \PP, \uu, \vv)$. Thus, the problem \eqref{ext} of finding $(\uu,\vv)$-extremal K\"ahler metrics in $\alpha$ reduces to the problem \eqref{main} of finding $(\uu, \vv \vv_{\rm ext})$-cscK metrics. Furthermore, as we show in \Cref{rem-Mab-rel}, the corresponding Mabuchi energy ${\mathcal M}_{\uu, \vv \vv_{\rm ext}}$ coincides with the $(\uu,\vv)$-weighted {\it relative} Mabuchi energy (see e.g. \cite{Gauduchon,Sz-book} for the definition of the relative Mabuchi energy in the case $\uu=\vv=1$ and \Cref{def-Mab-rel} below for its definition for general values of $\uu,\vv$), so we obtain as a consequence of Theorem~\ref{thm:mabuchi-bounded} the following

\begin{cor}\label{c:extremal} Let $(X,L)$ be a compact smooth polarized projective variety, $\T \subset {\rm Aut}_{\rm red}(X)$ a real torus, and suppose that $X$ admits a $(\uu,\vv)$-extremal metric in $\alpha=2\pi c_1(L)$ for some positive smooth functions $\uu, \vv$ defined on the momentum image $\PP\subset \tor^*$ associated to $(\T,\alpha)$. Then, the relative $(\uu,\vv)$-Mabuchi energy ${\mathcal M}_{\uu,\vv}^{\rm rel}$ of $(X, \T, \alpha, \PP, \uu, \vv)$ is bounded from below.
\end{cor}

The above corollary provides a scope of extending \cite[Thm.~2]{AMT} to an obstruction to the existence of $(z+a,p)$-extremal metrics in the sense of \cite{AMT}, in rational admissible K\"ahler classes on admissible projective bundles. We explore this ramification in \Cref{p:calabi-type} below.

\bigskip With the above in mind, the main goal of this paper is to introduce a suitable notion of {\it $(\uu, \vv)$-K-stability} associated to $(X, \alpha, \T, \PP, \uu, \vv)$ as above, extending the corresponding notion in the cscK and the extremal cases, introduced by Tian~\cite{Tian-book, tian}, Donaldson~\cite{Do-02} and Sz{\'e}kelyhidi~\cite{Sz}, and extensively studied in recent times. Our inspiration comes mainly from the recent works \cite{Dervan-Ross, Dervan, Dyr1} which, in turn, build on the seminal work of Tian \cite{tian, tian-old} and a key observation by Wang~\cite{wang} and Odaka~\cite{odaka1, odaka2} that the Donaldson--Futaki invariant of a (suitably compactified) test configuration can be realized as an intersection number defined on the total space of the latter. In the cscK case, it is shown by Dervan--Ross \cite[Prop. 2.23]{Dervan-Ross} and \cite[Prop. 3.12]{Dyr1} that in order to test K-stability (or K-semi-stability) of a compact K\"ahler manifold, it is sufficient to control the sign of the Donaldson--Futaki invariant of test configurations which are smooth and whose central fibre is reduced. This allows one to rewrite the Donaldson--Futaki invariant as a global differential geometric quantity of the test configuration. This is precisely the setting in which we introduce the notion of a $(\uu,\vv)$-Futaki invariant of a smooth K\"ahler test configuration with reduced central fibre, compatible with $(X, \alpha, \T)$, and show in \Cref{Mabuchi-Slope} that it must be non-negative should the $(\uu,\vv)$-Mabuchi energy associated to $(X, \T, \alpha, \PP, \uu, \vv)$ is bounded from below. This, combined with Theorem~\ref{thm:mabuchi-bounded} and \Cref{fut-obs} in \Cref{sec-6} yields our main result, which establishes one direction of a Yau--Tian--Donaldson type correspondence for the existence of $(\uu,\vv)$-cscK metrics.

\begin{thm}\label{thm:main} 
Let $(X, L)$ be a compact smooth polarized projective variety, $\T \subset {\rm Aut}_{\rm red}(X)$ a real torus, and suppose that $X$ admits a $(\uu,\vv)$-cscK metric in $\alpha=2\pi c_1(L)$. Then $X$ is $(\uu,\vv)$-K-semistable on smooth, $\T$-compatible K\"ahler test configuration with reduced central fibre associated to $(X, \alpha)$, i.e. the $(\uu, \vv)$-Futaki invariant of any such test configuration is non-negative.
\end{thm}

Several remarks are in order.

\smallskip
As we have already mentioned in the example \ref{1i} above, in the case  $\uu=\vv \equiv 1$ \Cref{thm:main} follows from the results in \cite{Dervan-Ross} and \cite{Dyr1} concerning the existence of cscK metrics in $\alpha$. Furthermore, as we show in \Cref{Rem-Dervan} below, taking $\uu=1$ and $\vv= \vv_{\rm ext}$ as in the example \ref{ii}, our definition of the Futaki invariant $\mathcal{F}_{1, \vv_{\rm ext}}$ reduces to the relative Futaki invariant of a smooth test configuration introduced in \cite{Dervan, Sz}, so in this special case \Cref{thm:main} above is established in \cite{Dervan}.

\smallskip
A natural question that arises in the case when the test-configuration is a polarized projective variety is the interpretation of the $(\uu,\vv)$-Futaki invariant in terms of a purely algebraic invariant defined on the central fibre $X_0$. This  was in fact the initial approach of Tian~\cite{tian} and Donaldson~\cite{Do-02} in the cscK case for defining an invariant of a test configuration, and a similar definition of a $(\uu,\vv)$-Donaldson--Futaki invariant  on $X_0$ has been proposed in \cite{AM, AMT} (regarding the cases \ref{iii} and \ref{iv}). At this point, it is not clear whether such an algebraic definition of a $(\uu,\vv)$-Donaldson--Futaki invariant can be given for any central fibre $X_0$, nor that it would agree with our differential geometric definition on the total space of a smooth test configuration. In fact, when $\uu$ and $\vv$ are not polynomials, the proposed algebraic definition of a $(\uu,\vv)$-Donaldson--Futaki invariant of $X_0$ involves transcendental quantities leading to difficulties reminiscent (but somewhat more complex) to the ones involved in the definition of the $L^p$-norm of a test configuration for positive real values of $p$, see the discussion at the end of \cite{Do-06}. Nevertheless, we prove that the two approaches give the same invariant in two special cases: when the test configuration is a smooth submersion (\Cref{DF=F}) and when $(X, \alpha, \T)$ is a smooth toric variety, and the test configuration is a toric test configuration in the sense of \cite{Do-02} (\Cref{DF=F-toric}).

\smallskip
Using a blowup technique and the glueing theorem of Arezzo--Pacard--Singer~\cite{APS}, Stoppa~\cite{stoppa} and Stoppa--Sz{\'e}kelyhidi~\cite{SSz} have shown that the existence of a cscK or extremal K\"ahler metric in $2\pi c_1(L)$ does actually imply $(1, \vv_{\rm ex})$-$K$-stability relative to $\T$, i.e. that the corresponding Futaki invariant of a non-product polarized normal test configuration is strictly positive. Similar results hold true for K\"ahler test configurations by \cite{Dervan, Dyr2}. At this point, it is not clear to us whether or not these techniques can be extended to the $(\uu,\vv)$-case.

\smallskip
Finally, one might hope to extend Theorem~\ref{thm:main} beyond the polarized case. Indeed, in the cscK and extremal cases such extensions have been found in \cite{Dervan, Dervan-Ross, Dyr2} by using a deep result of Berman--Berndtsson~\cite{BB} on the convexity and boundedness of the Mabuchi functional. We expect that along the method of \cite{BB} (and using \Cref{Chen-Tian} below) similar properties can possibly be established for the $(\uu,\vv)$-Mabuchi functional, but the details go beyond the scope of the present article. We however notice that the arguments in~\cite{BB} hold true in the case when $\uu\equiv1$ and $\vv$ is arbitrary, see \Cref{Mab-bound-1-w}. We thus have (by virtue of \Cref{Mabuchi-Slope})

\begin{thm}\cite{BB}\label{thm:BB} Let $X$ be a smooth compact K\"ahler manifold, $\T \subset {\rm Aut}_{\rm red}(X)$ a real torus, and suppose that $X$ admits a $(1,\vv)$-cscK metric $\omega$ in the K\"ahler class $\alpha$ for some smooth function $\vv$ on the momentum image $\PP\subset \tor^*$ associated to $(\T, \alpha)$. Then, the $(1,\vv)$-Mabuchi energy ${\mathcal M}_{1,\vv}$ of $(X, \T, \alpha, \PP, \vv)$ is bounded from below, and $X$ is $(1,\vv)$-K-semistable on smooth, $\T$-compatible K\"ahler test configuration with reduced central fibre associated to $(X, \alpha)$.
\end{thm}

\subsection{Outline of the paper} In \Cref{sec-2} we introduce the weighted $\uu$-scalar curvature of a $\T$-invariant K\"ahler metric and the constant $c_{\uu,\vv}(\alpha)$ in \eqref{main}, in terms of the data $(\alpha, \PP, \uu, \vv)$ on $(X, \T)$. As our definitions are new, in \Cref{sec-3} we describe in some more detail the examples listed in \ref{1i}--\ref{v} above. In \Cref{sec-4}, we generalize the arguments of Donaldson~\cite{donaldson} and Fujiki~\cite{fujiki} in the cscK case, and of Apostolov--Maschler~\cite{AM} in the conformally K\"ahler, Einstein--Maxwell case, thus providing a formal GIT interpretation of the problem of finding solutions of \eqref{main} within a given K\"ahler class $\alpha$ on $X$. In Section 5.1 we introduce the $(\uu,\vv)$-Mabuchi energy on $(X, \alpha, \T)$  associated to \eqref{main}. Our main result here is \Cref{Chen-Tian} which extends the Chen--Tian formula for the Mabuchi functional to the general $(\uu,\vv)$-case. In Section~{5.2}, assuming $\vv>0$, we define the relative $(\uu,\vv)$-Mabuchi energy $\mathcal{M}_{\uu,\vv}^{\rm rel}$ associated to the problem \eqref{ext} and show that it is given by the $(\uu,\vv\vv_{\rm ext})$-Mabuchi energy for a suitable affine linear function $\vv_{\rm ext}$ on $\tor^*$. In Section 5.3 we show the boundedness of the $(1,\vv)$-Mabuchi energy. In \Cref{sec-6} we define the differential-geometric $(\uu,\vv)$-Futaki invariant on $(X,\alpha,\T)$ and show in \Cref{fut-obs} that it provides a first obstruction of the existence of a solution of \eqref{main}. In the next \Cref{sec-7} we introduce a global invariant, which we call the {\it $(\uu,\vv)$-Futaki invariant}, on a $\T$-compatible smooth K\"ahler test configuration associated to $(X, \T, \alpha)$. We observe in \Cref{submersion-case}, that by an adaptation of the original arguments of \cite{ding-tian}, when the test configuration is a smooth submersion, the corresponding $(\uu,\vv)$-Futaki invariant agrees with the differential-geometric $(\uu,\vv)$-Futaki invariant of $X_0$. Our main result here is \Cref{Mabuchi-Slope} which shows how in the case when the central fibre is reduced the $(\uu,\vv)$-Futaki invariant of a $\T$-compatible smooth K\"ahler test configuration associated to $(X, \T, \alpha)$ is related to the $(\uu,\vv)$-Mabuchi functional of $(X, \T, \alpha)$. It yields \Cref{thm:main} from the introduction, modulo Theorem 1 which we establish in Appendix A. The arguments in the proof of \Cref{Mabuchi-Slope} go back to the foundational works \cite{ding-tian,tian} and are very close to the ones in \cite{Dervan-Ross, Dyr1}. In \Cref{sec-8}, we discuss the alternative approach to defining a Futaki invariant of a $\T$-compatible polarized test configuration in terms of algebraic constructions on the central fibre $X_0$, as suggested in \cite{AM,AMT}. In the case when the test configuration is a smooth submersion, by using the asymptotic expension of the $(\uu,\vv)$-equivariant Bergman kernel, we show in \Cref{DF=F} that two approaches produce the same invariant. Similar result is established in \Cref{DF=F-toric} in the case when $(X,\alpha,\T)$ is a smooth toric variety and we consider toric test configurations in the sense of \cite{Do-02}.
In \Cref{sec-10}, we consider the case when $(X,\alpha,\T)$ is a toric fiber-bundle over the product of cscK smooth projective manifolds, given by the generalized Calabi construction of \cite{ACGT}. We compute the $(\uu,\vv)$-Futaki invariant of certain test configurations of $(X,\alpha,\T)$, defined in terms of the toric geometry of the fiber. As an application of our theory, in the case when $X$ is a $\mathbb{P}^1$-bundle over a product of cscK smooth projective manifolds we derive a Yau-Tian-Donaldson type correspondence for $(\uu,\vv)$-extremal K\"ahler classes in terms of the positivity of a single function of one variable over the interval $(-1,1)$. In the Appendices A and B, we extend some of our previous results obtained for special values of $\uu$ and $\vv$ to the general case, including the proof of Theorem 1 from the introduction, a structure result for the automorphism group of a $(\uu,\vv)$-extremal metric with $\vv>0$, as well as a stability under deformation of $\uu,\vv$ and $\alpha$ of the solution of \eqref{ext} (again assuming $\vv>0$). 
\section*{Acknowledgement} 
I would like to thank my supervisor V. Apostolov for his guidance and constant help. I am grateful to E. Legendre for bringing to my attention the work \cite{Mcduff-Tolman}, and to G. Tian and X. Zhu for their interest and comments. Finally, I want to thank the referee for his/her careful reading of the manuscript and suggesting improvements, especially regarding the material in \Cref{sec-8}.      
\section{The $\uu$-scalar curvature}\label{sec-2}
Let $X$ be a compact K\"ahler manifold of complex dimension $n\geq 2$. We denote by $\Autred$ the reduced automorphism group of $X$ whose Lie algebra $\mathfrak{h}_{\rm red}$ is given by real holomorphic vector fields with zeros (see \cite{Gauduchon}). Let $\T$ be an $\ell$-dimentional real torus in ${\rm Aut}_{\rm red}(X)$ with Lie algebra $\mathfrak{t}$, and $\omega$ a $\T$-invariant K\"ahler form on $X$. We denote by $\mathcal{K}^{\T}_{\omega}$ the space of $\mathbb{T}$-invariant K\"ahler potentials with respect to $\omega$, and for any $\phi\in\mathcal{K}_{\omega}^{\T}$, by $\omega_\phi=\omega+dd^{c}\phi$ the corresponding K\"ahler form in the K\"ahler class $\alpha$. It is well-known that the $\T$-action on $X$ is $\omega_\phi$-Hamiltonian (see \cite{Gauduchon}) and we choose $m_{\phi}:X\rightarrow\mathfrak{t}^{*}$ to be a $\omega_\phi$-momentum map of $\T$. It is also known \cite{Atiyah, Guil-Stern} that $\PP_{\phi}:=m_{\phi}(X)$ is a convex polytope in $\mathfrak{t}^{*}$. Furthermore, the following is true

\begin{lem}\label{lem-mom-norm}
The following facts are equivalent:
\begin{enumerate}
\item\label{m1} For any $\phi\in\mathcal{K}^{\T}_\omega$ we have $\PP_\phi=\PP_\omega$.
\item\label{m2} For any $\phi\in\mathcal{K}^{\T}_\omega$ we have $\int_X m_\phi\omega^{[n]}_\phi=\int_X m_\omega\omega^{[n]}$, where $\omega^{[n]}_\phi:=\frac{\omega^{n}_\phi}{n!}$ is the volume form.
\item\label{m3} For any $\xi\in\mathfrak{t}$ and $\phi\in\mathcal{K}^{\T}_\omega$ we have $m^{\xi}_\phi=m^{\xi}_\omega+d^{c}\phi(\xi)$, where $m_{\phi}^\xi:=\langle m_\phi,\xi\rangle$.
\end{enumerate}
\end{lem}

\begin{proof}
Presumably, \Cref{lem-mom-norm} is well-known, see e.g. \cite[Section 4]{APS} and  \cite[Section 3.1]{Sz-blow} for the case of a single hamiltonian. We include an argument for the sake of completeness. We start by proving that \ref{m2} is equivalent with \ref{m3}. By the very definition of the momentum map, Cartan's formula and the fact that $\xi$ is a real holomorphic vector field we have
\begin{equation}\label{eq-m-phi-omega}
d(m^{\xi}_\omega-m^{\xi}_\phi)=-d(d^{c}\phi(\xi)).
\end{equation}
Thus, there exist a $\alpha_\phi\in \mathfrak{t}^{*}$ such that
\begin{equation}\label{eq-a}
m^{\xi}_\phi=m^{\xi}_\omega+d^{c}\phi(\xi) +\alpha_\phi(\xi).
\end{equation}
Suppose that \ref{m2} holds. Then $\alpha_\phi$ is given by
\begin{equation*}
\alpha_\phi(\xi)=\frac{1}{{\rm Vol}(X,\alpha)}\left(\int_X m^{\xi}_\omega\omega^{[n]}-\int_X(m^{\xi}_\omega+d^{c}\phi(\xi))\omega^{[n]}_\phi\right).
\end{equation*}
For a variation $\dot{\phi}$ of $\phi$ in $\mathcal{K}^{\T}_\omega$, the corresponding variation of $\alpha_\phi$ is given by
\begin{align*}
-{\rm Vol}(X,\alpha)\dot{\alpha}_\phi(\xi)=&\int_X m^{\xi}_\omega dd^{c}\dot{\phi}\wedge\omega^{[n-1]}_\phi+\int_Xd^{c}\dot{\phi}(\xi)\omega^{[n]}_\phi+\int_X d^{c}\phi(\xi)dd^{c}\dot{\phi}\wedge\omega^{[n-1]}_\phi\\
=&\int_X d^{c}\phi(\xi)dd^{c}\dot{\phi}\wedge\omega^{[n-1]}_\phi+\int_X dm^{\xi}_\phi\wedge d^{c}\dot{\phi}\wedge\omega^{[n-1]}_\phi\\&
+\int_X (-dm^{\xi}_\phi+d(d^{c}\phi(\xi))) \wedge d^{c}\dot{\phi}\wedge\omega^{[n-1]}_\phi\\
=&\int_X d(d^{c}\phi(\xi)) \wedge d^{c}\dot{\phi}\wedge\omega^{[n-1]}_\phi+\int_X d^{c}\phi(\xi)dd^{c}\dot{\phi}\wedge\omega^{[n-1]}_\phi=0,
\end{align*}
where we have used \eqref{eq-m-phi-omega}, the fact that $d^{c}\dot{\phi}(\xi)\omega^{[n]}_\phi=dm^{\xi}_\phi\wedge d^{c}\dot{\phi}\wedge\omega^{[n-1]}_\phi$, and integration by parts. It follows that $\alpha_\phi=\alpha_\omega=0$ which gives the implication \enquote{\ref{m2}$\Rightarrow$\ref{m3}}. Conversely if we suppose that \ref{m3} holds, then for any variation $\dot{\omega}=dd^{c}\dot{\phi}$ of $\omega_\phi$ in $\mathcal{K}^{\T}_\omega$, we get
\begin{equation*}
\frac{d}{dt}\int_X m^{\xi}_{\phi_t}\omega^{[n]}_{\phi_t}=\int_X -m^{\xi}_{\phi_t} dd^{c}\dot{\phi}\wedge\omega^{[n-1]}_{\phi_t}+d^{c}\dot{\phi}(\xi)\omega^{[n]}_{\phi_t}=0.
\end{equation*} 
It follows that $
\int_X m^{\xi}_{\phi}\omega^{[n]}_{\phi}=\int_X m^{\xi}_\omega\omega^{[n]}$ for any $\xi\in \mathfrak{t}$, which yields \ref{m2} .\\
Now we prove the equivalence between \ref{m1} and \ref{m3}. Suppose that \ref{m1} is true and let $x\in X$ be a fixed point for the $\T$-action on $X$. Then we have
\begin{equation}\label{eq-m-m}
m_\phi(x)-m_\omega(x)=(d^{c}\phi)_x+\alpha_\phi=\alpha_\phi.
\end{equation}
By a result of Atiyah and Guillemin--Sternberg (see \cite{Atiyah, Guil-Stern}) $\PP_\phi$ (resp. $\PP_\omega$) is the convex hull of the image by $m_\phi$ (resp. $m_\omega$) of the fixed points for the $\T$-action. It then follows from \eqref{eq-m-m} that $\PP_\phi=\PP_\omega+\alpha_\phi$. Using $\PP_\omega=\PP_\phi$, we get $\alpha_\phi=0$ which proves \ref{m3}.
For the inverse implication if $m_\phi(x)-m_\omega(x)=(d^{c}\phi)_{x}$ for any $x\in X$, then $m_\phi(x)=m_\omega(x)$ for any point $x\in X$ fixed by the $\T$-action and we have $\PP_\phi=\PP_\omega$ by \cite{Atiyah, Guil-Stern}. 
\end{proof}

It follows from \Cref{lem-mom-norm} that for each $\phi\in\mathcal{K}^{\T}_\omega$ we can normalize $m_\phi$ such that the momentum polytope $\PP=m_\phi(X)\subset \mathfrak{t}^{*}$ is $\phi$-independent.

\begin{defn}
For $\uu\in C^{\infty}(\PP,\mathbb{R}_{>0})$ we define the $\uu$-{\it scalar curvature} of the K\"ahler metric $g_\phi=\omega_\phi(\cdot,J\cdot)$ for $\phi\in\mathcal{K}^{\T}_{\omega}$ to be
\begin{equation}\label{Scal-v}
{\rm Scal}_{\uu}(\phi):=\uu(m_{\phi}){\rm Scal}(g_\phi)+2\Delta_\phi(\uu(m_{\phi}))+\Tr(\G_\phi\circ \left(\Hess(\uu)\circ m_{\phi}\right)),
\end{equation}
where $m_\phi$ is the momentum map of $\omega_\phi$ normalized as in \Cref{lem-mom-norm}, ${\rm Scal}(g_\phi)$ is the scalar curvature, $\Delta_\phi$ is the Riemannian Laplacian on functions of the K\"ahler metric $\omega_\phi$ and $\Hess(\uu)$ is the hessian of $\uu$, viewed as bilinear form on $\mathfrak{t}^{*}$ whereas $\G_\phi$ is the bilinear form with smooth coefficients on $\mathfrak{t}$, given by the restriction of the Riemannian metric $g_\phi$ on fundamental vector fields. 

In a basis $\xxi=(\xi_i)_{i=1,\cdots,\ell}$ of $\tor$ we have
\begin{equation*}
\Tr(\G_\phi\circ \left(\Hess(\uu)\circ m_{\phi}\right)):=\sum_{1\leq i,j\leq\ell}\uu_{,ij}(m_\phi)g_\phi(\xi_i,\xi_j),
\end{equation*}
where $\uu_{,ij}$ stands for the partial derivatives of $\uu$ with respect the dual basis of $\xxi$. 
\end{defn}

\begin{defn}\label{theta-mom}
Let $\theta$ be a $\T$-invariant closed $(1,1)$-form on $X$. A {\it $\theta$-momentum map} for the action of $\T$ on $X$ is a smooth $\T$-invariant function $m_\theta:X\rightarrow \mathfrak{t}^{*}$ with the property $\theta(\xi,\cdot)=-dm^{\xi}_\theta$ for all $\xi\in\mathfrak{t}$. 
\end{defn}

\begin{lem}\label{top-c}
Let $\theta$ be a fixed $\T$-invariant closed $(1,1)$-form and $m_\theta$ a momentum map for $\theta$. Then with the normalization for $m_\phi$ given by \Cref{lem-mom-norm}, the following integrals are independent of the choice of $\phi\in\mathcal{K}^{\T}_\omega$,
\begin{align*}
A_{\uu}(\phi):=&\int_X \uu(m_\phi)\omega^{[n]}_\phi,\\
B_{\uu}^{\theta}(\phi):=&\int_X \uu(m_\phi)\theta\wedge\omega^{[n-1]}_\phi+\langle (d\uu)(m_\phi),m_\theta\rangle\omega^{[n]}_\phi,\\
C_{\uu}(\phi):=&\int_X \Scal_{\uu}(\phi)\omega^{[n]}_\phi.
\end{align*} 
\end{lem}
\begin{proof}
The fact that $A_{\uu}(\phi)$ is constant is well known, see e.g. \cite[Theorem 3.14]{Dervan}. The constancy of $B_{\uu}^{\theta}(\phi)$ can be easily established by a direct computation, but it also follows from the arguments in the proof of \Cref{Lem-E-theta} below. Indeed, we note that $B_{\uu}^{\theta}(\phi)=(\mathcal{B}_{\uu}^{\theta})_\phi(1)$ where $\mathcal{B}_{\uu}^{\theta}$ is the $1$-form on $\mathcal{K}^{\T}_\omega$ given by \eqref{B-theta}. By taking $\dot{\phi}=1$ in \eqref{dA} we get $(\bdelta B_{\uu}^\theta)_\phi(\dot{\psi})=0$ where $\dot{\psi}$ is a $\T$-invariant function on $X$ defining a $\T$-invariant variation $\dot{\omega}=dd^{c}\dot{\psi}$ of $\omega_\phi$. From this we infer that $B_{\uu}^{\theta}(\phi)$ is constant. For the last function $C_{\uu}(\phi)$, we will calculate its variation $(\bdelta  C_{\uu})_\phi(\dot{\phi})$ with respect to a $\T$-invariant variation $\dot{\omega}=dd^{c}\dot{\phi}$ of $\omega_\phi$. For this, we use that the variation of $\Scal_{\uu}(\phi)$ is given by 
\begin{equation}\label{s-var}
(\bdelta\Scal_{\uu})_\phi(\dot{\phi})=-2(D^{-}d)^{*}\uu(m_\phi)(D^{-}d)\dot{\phi}+(d\Scal_{\uu}(\phi),d\dot{\phi})_\phi,
\end{equation}
where $D$ is the Levi-Civita connection of $\omega_\phi$, $D^-d$ denotes the $(2,0)+(0,2)$ part of $Dd$ with respect to $J$ and $(D^-d)^*$ is the formal adjoint operator of $(D^-d)$ (see \cite[Section 1.23]{Gauduchon}). Formula \eqref{s-var} will be established in the Appendix B, see \eqref{var-Scal-u-v}. By \eqref{s-var}, we calculate 
\begin{align*}
(\bdelta C_{\uu})_\phi(\dot{\phi})=&\int_X-2(D^{-}d)^{*}\uu(m_\phi)(D^{-}d)(\dot{\phi})\omega^{[n]}_\phi+\int_Xd\Scal_{\uu}(\phi)\wedge d^{c}\dot{\phi}\wedge\omega^{[n-1]}_\phi\\
&-\int_X\Scal_{\uu}(\phi)dd^{c}\dot{\phi}\wedge\omega^{[n-1]}_\phi.
\end{align*}
Integration by parts yields $(\bdelta C_{\uu})_\phi=0$. Thus $C_{\uu}$ does not depend on the choice of $\phi\in\mathcal{K}^{\T}_\omega$.
\end{proof}

\begin{defn}\label{Def-Top-const}
Let $(X,\alpha)$ be a compact K\"ahler manifold, $\T\subset {\rm Aut}_{\rm red}(X)$ a real torus with momentum image $\PP\subset\tor^*$ associated to $\alpha$ as in \Cref{lem-mom-norm}, and $\uu\in C^{\infty}(\PP,\mathbb{R}_{>0})$, $\vv\in C^{\infty}(\PP,\mathbb{R})$. The $(\uu,\vv)$-{\it slope} of $(X,\alpha)$ is the constant given by
\begin{equation}\label{Top-const}
c_{(\uu,\vv)}(\alpha):=\begin{cases} \frac{\int_X \Scal_{\uu}(\omega)\omega^{[n]}}{\int_X \vv(m_\omega)\omega^{[n]}},&\text{if }\int_X \vv(m_\omega)\omega^{[n]}\neq 0\\ 
1,&\text{if } \int_X \vv(m_\omega)\omega^{[n]}= 0,
\end{cases}
\end{equation}
which is independent from the choice of $\omega\in \alpha$ by virtue of \Cref{top-c}.
\end{defn}  
\begin{rem}\label{rem-1}
If $\phi\in\mathcal{K}_{\omega}^{\T}$ defines a K\"ahler metric which satisfies $
\Scal_{\uu}(\phi)=c\vv(m_{\phi})$ for some real constant $c$ and $\int_X \vv(m_\omega)\omega^{[n]}\neq0$, then we must have $c=c_{(\uu,\vv)}(\alpha)$ with $c_{\uu,\vv}(\alpha)$ given by \eqref{Top-const}.
\end{rem}
Because of \Cref{rem-1} above, and to simplify the notation in the case when $ \int_X \vv(m_\omega)\omega^{[n]}= 0$, we adopt the following definition
\begin{defn}
Let $(X,\alpha)$ be a compact K\"ahler manifold, $\T\subset {\rm Aut}_{\rm red}(X)$ a real torus with momentum image $\PP\subset\tor^*$ associated to $\alpha$ as in \Cref{lem-mom-norm}, and $\uu\in C^{\infty}(\PP,\mathbb{R}_{>0})$, $\vv\in C^{\infty}(\PP,\mathbb{R})$. A $(\uu,\vv)$-{\it cscK} metric $\omega\in\alpha$ is a $\T$-invariant K\"ahler metric satisfying \eqref{main}, where $c_{\uu,\vv}(\alpha)$ is given by \eqref{Top-const}
\end{defn} 

\section{Examples}\label{sec-3}
We list below some geometrically significant examples of $(\uu,\vv)$-cscK metrics,  obtained for special values of the weight functions $\uu, \vv$.

\subsection{Constant scalar curvature and extremal K\"ahler metrics}\label{s:extremal} When $\uu\equiv 1$, ${\rm Scal}_{\uu}(\phi)= {\rm Scal}(\phi)$ is the usual scalar curvature of the K\"ahler metric $\omega_{\phi} \in {\mathcal K}_{\omega}^{\T}$, so letting $\vv\equiv 1$ the problem \eqref{main} reduces to the Calabi problem of finding a cscK metric in the K\"ahler class $\alpha=[\omega]$. In this case, we can take $\T\subset {\rm Aut}_{\rm red}(X)$ to be a maximal torus by a result of Calabi~\cite{calabi}. More generally, for a fixed maximal torus $\T \subset {\rm Aut}_{\rm red}(X)$ we can consider the more general problem of the existence of an {\it extremal K\"ahler metric} in ${\mathcal K}_{\omega}^{\T}$, i.e. a K\"ahler metric $\omega_{\phi}$ such that ${\rm Scal}(\phi)$ is a Killing potential for $\omega_{\phi}$. As the Killing vector field $\xi_{\rm ext}$ generated by ${\rm Scal}(\phi)$  is $\T$-invariant, it belongs to the Lie algebra $\tor$ of $\T$ (by the maximality of $\T$). More generally, Futaki--Mabuchi~\cite{FM} observed that for any $\phi\in {\mathcal K}_{\omega}^{\T}$, the $L^2$ projection ${\check {\rm Scal}}(\phi)$ (with respect to the global inner product on smooth functions defined by $\omega_{\phi}$) of ${\rm Scal}(\phi)$ to the sub-space 
$\{m^{\xi}_{\phi} + c,  \ c\in \R\}$ of Killing potentials for $\xi \in \tor$  defines a $\phi$-independent element $\xi_{\rm ext}\in \tor$, i.e. ${\check {\rm Scal}}(\phi)= m_{\phi}^{\xi_{\rm ext}} + c_{\phi}$. The vector field $\xi_{\rm ext}$ is  called the {\it extremal vector field} of $(X, \alpha, \T)$. Furthermore, using the normalization for the moment map $m_{\phi}$ in \Cref{lem-mom-norm}, we see that 
\begin{equation*}
\begin{split}
4\pi c_1(X) \cup\alpha^{[n-1]} &=\int_X {\rm Scal}(\phi) \omega_{\phi}^{[n]} =  \int_X {\check{\rm Scal}}(\phi) \omega_{\phi}^{[n]}  \\
                                               &= \int_X m_{\phi}^{\xi_{\rm ext}} \omega_{\phi}^{[n]} + c_{\phi}{\rm Vol}(X, \alpha),
                                                \end{split}
\end{equation*}
showing that the real constant $c_{\rm ext}=c_{\phi}$ is independent of $\omega_{\phi}$ too. Thus, there exists an affine-linear function ${\vv}_{\rm ext}(p)=\langle \xi_{\rm ext}, p \rangle + c_{\rm ext}$ on $\tor^*$, such that $\omega_{\phi} \in {\mathcal K}^{\T}_{\omega}$ is extremal if and only if $\Scal_{\uu}(\phi)=\vv_{\rm ext}(m_\phi)$ i.e. if and only if $\omega_{\phi}$ is $(1, \vv_{\rm ext})$-cscK (as 
$c_{1,\vv_{\rm ext}}(\alpha)=1$ by definition of $\vv_{\rm ext}$).

\subsection{$(\uu,\vv)$-extremal K\"ahler metrics}\label{s:(u,v)-extremal}   
As mentioned in the Introduction (and motivated by the previous example) one can consider instead of \eqref{main} the more general problem \eqref{ext} of finding a $(\uu,\vv)$-extremal K\"ahler metric $\omega_\phi$ for $\phi \in {\mathcal K}^{\T}_{\omega}$. It turns out that if $\vv(p)>0$ on $\PP$, similarly to the previous example, one can reduce the problem \eqref{ext} to the problem \eqref{main} with the same $\uu$ but a different $\vv$. This essentially follows from \Cref{moment-map} below, which implies that for any $\T$-invariant, $\omega$-compatible K\"ahler metric $g$, the orthogonal projection of ${\rm Scal}_{\uu}(g)/\vv(m_{\omega})$ to the space of affine-linear functions in momenta with respect to the $\vv$-weighted global inner product \eqref{v-prod} is independent of $g$. Using the $\T$-equivariant Moser lemma for a K\"ahler metric $\omega_\phi \in {\mathcal K}_{\omega}^{\T}$ and the normalization for $m_{\phi}$ given by \Cref{lem-mom-norm}, one can conclude as in the proof \cite[Cor.~2]{AM} that there exist a $\phi$-independent affine-linear function $\vv_{\rm ext}(p)$ such that  $m_{\phi}^{\xi} + c=\vv_{\rm ext}(m_{\phi})$ for any metric in $\mathcal K_{\omega}^{\T}$ satisfying \eqref{ext}. In other words, if $\phi \in {\mathcal K}^{\T}_{\omega}$ is $(\uu,\vv)$-extremal then $\omega$ is $(\uu, \vv \vv_{\rm ext})$-cscK. Conversely if $\omega_\phi$ is $(\uu, \vv \vv_{\rm ext})$-cscK, then $\Scal_{\uu}(\omega_\phi)=c_{\uu,\vv\vv_{\rm ext}}(\alpha)\vv(m_\phi)\vv_{\rm ext}(m_\phi)$ where $c_{\uu,\vv\vv_{\rm ext}}(\alpha)$ is given by \eqref{Top-const}. We claim that $c_{\vv,\vv\vv_{\rm ext}}(\alpha)=1$, which in turn implies that $\omega_\phi$ is $(\uu,\vv)$-extremal. Indeed, if $\int_X\vv(m_\phi)\vv_{\rm ext}(m_\phi)\omega^{[n]}_\phi= 0$, then $c_{\uu,\vv\vv_{\rm ext}}(\alpha)=1$ by \Cref{Def-Top-const}. Otherwise, if $\int_X\vv(m_\phi)\vv_{\rm ext}(m_\phi)\omega^{[n]}_\phi\neq 0$, we get
\begin{align*}
c_{\uu,\vv\vv_{\rm ext}}(\alpha)\int_X\vv(m_\phi)\vv_{\rm ext}(m_\phi)\omega^{[n]}_\phi&=\int_X(\Scal_{\uu}(\phi)/\vv(m_\phi))\vv(m_\phi)\omega^{[n]}_\phi\\
&=\int_X\vv_{\rm ext}(m_\phi)\vv(m_\phi)\omega^{[n]}_\phi,
\end{align*}
showing again that $c_{\uu,\vv\vv_{\rm ext}}(\alpha)=1$.
\subsection{The K\"ahler-Ricci solitons}
This is the case when $X$ is a smooth Fano manifold, $\alpha = 2\pi c_1(X)$ corresponds to the anti-canonical polarization, $\T \subset {\rm Aut}_{\rm red}(X)$ is a maximal torus with momentum image $\PP$, and $\uu(p)=\vv(p)=e^{\langle \xi, p \rangle}$  for some $\xi\in\tor$. It was shown recently in \cite{Eiji1} that a solutions of \eqref{main}
with $\vv_{\rm ext}(p) = 2(\langle \xi, p\rangle + c)$ (for some real constant $c$) corresponds to a K\"ahler metric $\omega \in \alpha$ which is a {\it gradient K\"ahler--Ricci soliton} with respect to $\xi$, i.e. satisfies
\begin{equation}\label{KRS}
{\rm Ric}({\omega}) - \omega = -\frac{1}{2} \mathcal{L}_{J\xi} \omega,
\end{equation}
where ${\rm Ric}(\omega)$ is the Ricci form of $\omega$. Thus, the theory of gradient K\"ahler-Ricci solitons (see e.g. \cite{B-N,CTZ, TZ1,TZ2}) fits in to our setting too. Further ramifications of this setting appear in \cite{Eiji2}.

\subsection{K\"ahler metrics conformal to Einstein--Maxwell metrics} This class of K\"ahler metrics was first introduced in \cite{LeB08} and more recently studied in \cite{AM, AMT, FO, FO1, KoTo16,  lahdili1, lahdili2, LeB0, LeB1}. These are $(\uu,\vv)$-cscK metrics with

\begin{equation*}
\uu(p)=(\langle \xi, p \rangle + a)^{-2m+1} \ \textrm{and} \ \  \vv(p)=(\langle \xi, p \rangle + a)^{-2m-1},
\end{equation*}  
where $\langle \xi, p \rangle + a$  is positive affine-linear function on $\PP$. In this case, ${\rm Scal}_{\uu}(\phi)/\vv(m_{\phi})$ equals to the usual scalar curvature of the Hermitian metric ${\tilde g}_{\phi}= \frac{1}{(m_{\phi}^{\xi} +a)^2} g_{\phi}$. Thus, a $(\uu,\vv)$-cscK metric $\omega_{\phi}$ gives rise to a conformally K\"ahler, Hermitian metric $\tilde{g}_{\phi}$ which has Hermitian Ricci tensor and constant scalar curvature. The latter include the conformally K\"ahler, Einstein metrics classified in \cite{CLW, DM}. In real dimension 4, conformally K\"ahler, Einstein--Maxwell metrics give rise to analogues, in riemannian signature, of the Einstein--Maxwell field equations with a cosmological constant in general relativity.

\subsection{Extremal Sasaki metrics}
Following \cite{AC}, let $(X,L)$ be a smooth compact polarized variety and $\alpha = {2\pi} c_1(L)$ the corresponding K\"ahler class. Recall that for any K\"ahler metric $\omega\in \alpha$, there exits a unique Hermitian metric $h$ on $L$, whose curvature is $\omega$. We denote by $h^*$ the induced Hermitian metric on the dual line bundle $L^*$. It is well-known (see e.g. \cite{BG-book}) that the principal circle bundle $\pi : S \to X$  of  vectors of unit norm of $(L^*, h^*)$ has the structure of a Sasaki manifold, i.e. there exists a contact $1$-form $\theta$ on $S$ with $d\theta = \pi^* \omega$, defining a contact distribution $D \subset TS$ and a Reeb vector field $\chi$ given by the generator of the $\mathbb{S}^1$-action on the fibres of $S$, and a CR-structure $J$ on $D$ induced from the complex structure of $L^*$. The Sasaki structure $(\theta, \chi, D,  J)$ on $S$ in turn defines a transversal K\"ahler structure $(g_{\chi}, \omega_{\chi})$ on $D$ by letting $\omega_{\chi} = (d\theta)_{D}$ and $g_{\chi} = -(d\theta)_{D}\circ J,$ where the subscript $D$ denotes restriction to $D\subset TS$; it is a well-known fact that $(g_{\chi}, \omega_{\chi})$ coincides with the restriction to $D$ of the pull-back of the K\"ahler structure $(g,\omega)$ on $X$ or, equivalently, that $(g_{\chi}, \omega_{\chi})$ induces the initial K\"ahler structures $(g,\omega)$ on the orbit space $X= S/\mathbb{S}^{1}_{\chi}$ for the $\mathbb{S}^1$-action $\mathbb{S}^{1}_{\chi}$ generated by $\chi$.

Let $\T \subset {\rm Aut}_{\rm red}(X)$ be a maximal torus, with a fixed momentum polytope $\PP \subset \tor^*$ associated to the K\"ahler class $\alpha$ as in Lemma 1.  We suppose that $\omega$ is a $\T$-invariant K\"ahler metric in $\alpha$. For any positive affine-linear function $\langle \xi, p \rangle + a$ on $\PP$, we consider  the corresponding Killing potential $f= m_{\omega}^\xi + a$ of $\omega$ and define the lift $\xi_f$ of the Killing vector field $\xi\in\tor$ on $X$ to $S$ by
\begin{equation*}
\xi_f = \xi^{D} + (\pi^* f) \chi,
\end{equation*}
where the super-scrip $D$ stands for the horizontal lift. It is  easily checked that $\xi_f$  preserves the contact distribution  $D$ and the CR-structure $J$, and defines a new Sasaki structure $((\pi^*f )^{-1}\theta, \xi_f, D, J)$ on $S$. In general, the flow of $\xi_f$ is not periodic, and the orbit space of $\xi_f$ is not Hausdorff, but when it is, $X_f:=S/\mathbb{S}^{1}_{\xi_f}$ is a compact complex orbifold endowed with a K\"ahler structure $(g_f, \omega_f)$. In \cite{AC}, the triple $(X_f, g_f, \omega_f)$ is referred to as a {\it CR $f$-twist} of $(X, \omega, g)$ and it is shown there that $(X_f, g_f, \omega_f)$ is an extremal K\"ahler manifold or orbifold in the sense of Sect.~\ref{s:extremal} iff $(X, \omega, g)$ is $(\uu,\vv)$-extremal in the sense of Sect.~\ref{s:(u,v)-extremal} with  
\begin{equation}\label{(m+2)-extremal}
\uu(p)=(\langle \xi, p \rangle + a)^{-m-1} \ \ \textrm{and} \ \ \vv(p)=(\langle \xi, p \rangle + a)^{-m-3}.
\end{equation}
\begin{comment}
 In general, by using the $\T$-equivariant Moser lemma, a $(\uu,\vv)$-extremal metric $\omega_{\phi}\in {\mathcal K}_{\omega}^{\T}$  (with $\uu, \vv$ given by \eqref{(m+2)-extremal}) gives rise to an extremal Sasaki structure $(\xi_f, D, J_{\phi})$ on $S$ within a class of $(\xi_f, D)$-compatible CR structures parametrized by $\phi \in {\mathcal K}_{\omega}^{\T}$. By the discussion in Sect.~\ref{s:(u,v)-extremal}, \Cref{thm:main} provides an obstruction to the existence of such Sasaki structures. We note that a similar obstruction theory has been developed in \cite{CSz} and \cite{BC} for a somewhat different class of $(\xi_f, D)$-compatible CR structures. It will be interesting to understand how these theories are inter-related.
\end{comment}
\subsection{The generalized Calabi construction and manifolds without multiplicities} 

In \cite{ACGT} the authors consider smooth compact manifolds $X$, which are fibre-bundles over the product of cscK Hodge manifolds $(B, \omega_B)= (B_1, \omega_1) \times \cdots \times (B_N, \omega_N)$ with fibre a smooth $\ell$-dimensional compact toric  K\"ahler manifold $(V, \omega_V, \T)$. More precisely, $X$ is a $V$-fibre bundle associated to a certain principle $\T$-bundle over $B$. They introduce a class of $\T$-invariant K\"ahler metrics on $X$, compatible with the bundle structure, which are parametrized by $\omega_{V}$-compatible toric K\"ahler metrics on $V$, and refer to them as K\"ahler metrics given by the {\it generalized Calabi construction}. The condition for the metric $\omega$ on $X$ to be extremal is computed in \cite{ACGT} and can be re-written in our formulation as (see \eqref{u-scalar} below)
\begin{equation}\label{generalized-Calabi-extremal}
{\rm Scal}_{\uu} (g_V) = \vv(m),
\end{equation}
where $g_V$ is the corresponding toric K\"ahler metric on $(V, \omega_V)$, with 
\begin{equation}\label{(u,v)-generalized-Calabi}
\begin{split}
\uu(p) &= \prod_{j=1}^N (\langle \xi_j, p \rangle + c_j)^{d_j}, \\
\vv(p) & = (\langle \xi_0, p \rangle + c_0)\prod_{j=1}^N (\langle \xi_j, p \rangle + c_j)^{d_j} - \sum_{j=1}^N {\rm Scal}_j \Big(\frac{\prod_{k=1}^N (\langle \xi_k, p \rangle + c_k)^{d_j}}{(\langle \xi_j, p \rangle + c_j)}\Big),
         \end{split}
         \end{equation}
in the above expressions,  $m :  V \to \tor^*$ stands for the momentum map of $(V, \omega_V, \T)$,  $d_j$ and ${\rm Scal}_j$ denote the complex dimension  and (constant) scalar curvature of $(B_j, \omega_j)$, respectively, whereas the affine-linear functions $(\langle \xi_k, p \rangle + c_k), k=1,\cdots, N$  on $\tor^*$ are determined by the topology and the K\"ahler class $\alpha=[\omega]$ of $X$, and satisfy $(\langle \xi_j, p \rangle + c_j)>0$  for $ j=1, \cdots, N$ on the Delzant polytope $\PP=m(V)$. Thus, a K\"ahler metric $\omega$ on $X$ given by the generalized Calabi ansatz is extremal if and only if the corresponding toric K\"ahler metric $g_V$ on $V$ is $(\uu, \vv)$-extremal for the values of $\uu,\vv$ given in \eqref{(u,v)-generalized-Calabi}. More generally, considering an arbitrary weight function $\vv$ in \eqref{generalized-Calabi-extremal} allows one to prescribe the scalar curvature of the K\"ahler metrics given by the geberalized Calabi construction on $X$. We note that a very similar equation for a toric K\"ahler metric on $V$ appears in the construction of K\"ahler manifolds without multiplicities, see \cite{donaldson-survey,R}. We refer the Reader to \cite{LLS1,LLS2,LSZ} for a comprehensive study of the equation \eqref{generalized-Calabi-extremal} on a toric variety, for arbitrary weight functions $\uu(p)>0$ and $\vv(p)$, which is referred to as the {\it generalized Abreu equation}.

\section{A formal momentum map picture}\label{sec-4}

In this section we extend the momentum map interpretation, originally introduced Donaldson \cite{donaldson} and Fujiki \cite{fujiki} in the cscK case and generalized by Apostolov--Maschler \cite{AM} to the case \ref{iii} from the Introduction, to arbitrary positive weights $\uu,\vv$ on $\PP$.

\smallskip

In the notation of \Cref{sec-2}, let $\mathcal{AC}^{\T}_\omega$ be the space of all $\omega$-compatible, $\T$-invariant almost complex structures on $(X,\omega)$ and $\mathcal{C}^{\T}_\omega\subset\mathcal{AC}^{\T}_\omega$ the subspace of $\T$-invariant K\"ahler structures. We consider the natural action on $\mathcal{AC}^{\T}_\omega$ of the infinite dimensional group ${\rm Ham}^{\T}(X,\omega)$ of $\T$-equivariant Hamiltonian transformations of $(X,\omega)$, which preserves $\mathcal{C}^{\T}_\omega$. We identify ${\rm Lie}\left({\rm Ham}^{\T}(X,\omega)\right)\cong C^{\infty}(X,\mathbb{R})^{\T}/\mathbb{R}$ where $C^{\infty}(X,\mathbb{R})^{\T}/\mathbb{R}$ is endowed with the Poisson bracket.\\

\smallskip
For any $\uu\in C^{\infty}(\PP,\mathbb{R}_{>0})$, the space $\mathcal{AC}^{\T}_\omega$ carries a weighted formal K{\"a}hler structure $(\bold{J},\bold{\Omega}^{\uu})$ given by (\cite{AM, donaldson, fujiki})
\begin{align*}
\bold{\Omega}^{\uu}_J(\dot{J}_1,\dot{J}_2):=&\frac{1}{2}\int_X {\rm Tr}(J\dot{J}_1\dot{J}_2)\uu(m_{\omega})\omega^{[n]},\\
\bold{J}_J(\dot{J}):=&J\dot{J},
\end{align*}
in which the tangent space of $\mathcal{AC}^{\T}_\omega$ at $J$ is identified with the space of smooth $\T$-invariant sections $\dot{J}$ of ${\rm End}(TX)$ satisfying
\begin{equation*}
\dot{J}J+J\dot{J}=0,\quad \omega(\dot{J}\cdot,\cdot)+\omega(\cdot,\dot{J}\cdot)=0.
\end{equation*}
In what follows, we denote by $g_J:=\omega(\cdot,J\cdot)$ the almost K{\"a}hler metric corresponding to $J\in\mathcal{AC}^{\T}_\omega$, and index all objects calculated with respect to $J$ similarly. On $C^{\infty}(X,\mathbb{R})^{\T}$, for $\vv\in C^{\infty}(\PP,\mathbb{R}_{>0})$, we consider the scalar product given by,
\begin{equation}\label{v-prod}
\langle\phi,\psi\rangle_{\vv}:=\int_X\phi\psi \vv(m_{\omega})\omega^{[n]},
\end{equation}
 
\begin{thm}\label{moment-map}\cite{AM, donaldson, fujiki}
The action of ${\rm Ham}^{\T}(X,\omega)$ on $\left(\mathcal{AC}^{\T}_\omega,\bold{J},\bold{\Omega}^{\uu}\right)$ is Hamiltonian whose momentum map at $J\in\mathcal{C}^{\T}_\omega$ is the $\langle.,.\rangle_{\vv}$-dual of $\Big(\frac{{\rm Scal}_{\uu}(J)}{\vv(m_{\omega})}-c_{\uu,\vv}([\omega])\Big)$, where ${\rm Scal}_{\uu}(J)$ is the $\uu$-scalar curvature of $g_J$ given by \eqref{Scal-v} and the real constant $c_{\uu,\vv}([\omega])$ is given by \eqref{Top-const}.
\end{thm}

\begin{proof}
The proof follows from the computation of \cite[Theorem 1]{AM} and \cite[Section 9.6]{Gauduchon} and will be left to the reader.
\end{proof}

\section{A variational setting}\label{sec-5}

\subsection{The $(\uu,\vv)$-Mabuchi energy}In this section we suppose that $\uu\in C^{\infty}(\PP,\mathbb{R}_{>0})$ and $\vv\in C^{\infty}(\PP,\mathbb{R})$ is an arbitrary smooth function. We consider $\mathcal{K}^{\T}_\omega$ as a Fr\'echet space with tangent space $T_\phi\mathcal{K}^{\T}_\omega=C^{\infty}(X,\R)^{\T}$ the space of $\T$-invariant smooth functions $\dot{\phi}$ on $X$. 
\begin{defn}\label{Defn-Mabuchi}
The $(\uu,\vv)$-Mabuchi energy $\mathcal{M}_{\uu,\vv}:\mathcal{K}^{\T}_{\omega}\rightarrow\R $ is defined by 
\begin{equation}\label{Mabuchi}
\begin{cases} (d\mathcal{M}_{\uu,\vv})_{\phi}(\dot{\phi})={\displaystyle-\int_X\dot{\phi}\big(\Scal_{\uu}(\phi)-c_{(\uu,\vv)}(\alpha)\vv(m_{\phi})\big)\omega^{[n]}_{\phi}},\\ 
\mathcal{M}_{\uu,\vv}(0)=0,
\end{cases}
\end{equation}
for all $\dot{\phi}\in T_\phi \mathcal{K}^{\T}_{\omega}$, where $c_{(\uu,\vv)}(\alpha)$ is the constant given by \eqref{Top-const}.   
\end{defn}
\begin{rem} The critical points of $\mathcal{M}_{\uu,\vv}$ are precisely the $\T$-invariant K\"ahler potentials $\phi\in\mathcal{K}^{\T}_{\omega}$ such that $\omega_\phi$ is a solution to the equation \eqref{main}.
\end{rem}

We will show that the $(\uu,\vv)$-Mabuchi energy is well-defined by establishing in  \Cref{Chen-Tian} below an analogue of the Chen-Tian formula (see \cite{chen, Tian-book}). We start with few lemmas.

\begin{lem}\label{E-lem-def}
The functional $\mathcal{E}_{\vv}:\mathcal{K}^{\T}_{\omega}\rightarrow\mathbb{R}$ given by 
\begin{equation}\label{E}
\begin{cases} \left(d\mathcal{E}_{\vv}\right)_{\phi}(\dot{\phi})={\displaystyle\int_X\dot{\phi}\vv(m_{\phi})\omega^{[n]}_{\phi}},\\ 
\mathcal{E}_{\vv}(0)=0,
\end{cases}
\end{equation} 
for any $\dot{\phi}\in T_\phi\mathcal{K}^{\T}_\omega$ is well-defined.
\end{lem}
\begin{proof}
See \cite[Lemma 2.14]{B-N}.
\begin{comment}
\begin{equation*}
(\mathcal{A}_{\vv})_\phi(\dot{\phi})=\int_X\dot{\phi}\vv(m_\phi)\omega^{[n]}_\phi,
\end{equation*}
is closed. For $\dot{\phi},\dot{\psi}\in T_\phi\mathcal{K}^{\T}_\omega$ we compute
\begin{align*}
(\bdelta \mathcal{A}_{\vv}(\dot{\phi}))_{\phi}(\dot{\psi})=&\int_X\dot{\phi}d(\vv(m_\phi))\wedge d^{c}\dot{\psi}\wedge\omega_{\phi}^{[n-1]}+\int_X\dot{\phi}\vv(m_\phi)dd^{c}\dot{\psi}\wedge\omega_{\phi}^{[n-1]}\\
=&\int_X(d\dot{\phi},d\dot{\psi})_\phi\vv(m_\phi)\omega_{\phi}^{[n]}.
\end{align*}
It follows that
\begin{equation*}
(d\mathcal{A}_{\vv})_\phi(\dot{\phi},\dot{\psi})=(\bdelta \mathcal{A}_{\vv}(\dot{\phi}))_{\phi}(\dot{\psi})-(\bdelta \mathcal{A}_{\vv}(\dot{\psi}))_{\phi}(\dot{\phi})=0,
\end{equation*}
which shows that $\mathcal{A}_{\vv}$ is closed and therefore $\mathcal{E}_{\vv}:\mathcal{K}^{\T}_{\omega}\rightarrow\mathbb{R}$ is well-defined.
\end{comment}
\end{proof}

\begin{lem}\label{Lem-E-theta}
Let $\theta$ be a fixed $\T$-invariant closed $(1,1)$-form and $m_\theta:X\rightarrow \mathfrak{t}^{*}$ a momentum map with respect to $\theta$, see \Cref{theta-mom}. Then the functional $\mathcal{E}_{\uu}^{\theta}:\mathcal{K}^{\T}_\omega\rightarrow\mathbb{R}$ given by 
\begin{equation}\label{E-theta}
\begin{cases} (d\mathcal{E}_{\uu}^{\theta})_{\phi}(\dot{\phi})={\displaystyle\int_X \dot{\phi} \left[\uu(m_{\phi})\theta\wedge\omega^{[n-1]}_\phi+\langle (d\uu)(m_{\phi}),m_{\theta}\rangle\omega^{[n]}_{\phi}\right]},\\ 
\mathcal{E}_{\uu}^{\theta}(0)=0,
\end{cases}
\end{equation}
for any $\dot{\phi}\in T_\phi\mathcal{K}^{\T}_\omega$ is well-defined.
\end{lem}

\begin{proof}
As the Frech\'et space $\mathcal{K}^{\T}_\omega$ is contractible, we have to show that the $1$-form on $\mathcal{K}^{\T}_\omega$ 
\begin{equation}\label{B-theta}
(\mathcal{B}_{\uu})_\phi(\dot{\phi}):=\int_X \dot{\phi}\left[\uu(m_{\phi})\theta\wedge\omega^{[n-1]}_\phi+\langle (d\uu)(m_{\phi}),m_{\theta}\rangle\omega^{[n]}_{\phi}\right],
\end{equation}
is closed. For $\dot{\phi},\dot{\psi}\in T_\phi\mathcal{K}^{\T}_{\omega}$ we compute
\begin{align*}
(\bdelta\mathcal{B}_{\uu}(\dot{\phi})&)_\phi(\dot\psi)=\int_X\dot{\phi} (d(\uu(m_{\phi})),d\dot{\psi})_\phi\theta\wedge\omega^{[n-1]}_\phi+\int_X\dot{\phi}\uu(m_{\phi})\theta\wedge dd^{c}\dot{\psi}\wedge\omega^{[n-2]}_\phi\\&+\int_X\sum_{j=1}^{\ell} \dot{\phi} m_{\theta}^{\xi_j}d(\uu_{,j}(m_{\phi}))\wedge d^{c}\dot{\psi}\wedge\omega^{[n-1]}_\phi-\int_X\sum_{j=1}^{\ell} \dot{\phi}\uu_{,j}(m_{\phi})m_{\theta}^{\xi_j}dd^{c}\dot{\psi}\wedge\omega_{\phi}^{[n-1]}\\
=&\int_X\dot{\phi} (d(\uu(m_{\phi})),d\dot{\psi})_\phi\theta\wedge\omega^{[n-1]}_\phi+\int_X\dot{\phi}\uu(m_{\phi})\theta\wedge dd^{c}\dot{\psi}\wedge\omega^{[n-2]}_\phi\\&
-\int_X \sum_{j=1}^{\ell}\dot{\phi}\uu_{,j}(m_{\phi})(dm^{\xi_j}_{\theta},d\dot{\psi})_\phi\omega_{\phi}^{[n]}-\int_X(d\dot{\phi},d\dot{\psi})_\phi\langle (d\uu)( m_{\phi}),m_{\theta}\rangle\omega^{[n]}_{\phi},
\end{align*}
where $\xxi:=(\xi_j)_{j=1,\cdots,\ell}$ is a basis of $\tor$. Integrating by parts, we obtain
\begin{align*}
&\int_X\dot{\phi}\uu(m_{\phi})\theta\wedge dd^{c}\dot{\psi}\wedge\omega^{[n-2]}_\phi\\
=&-\int_X \uu(m_{\phi})\theta\wedge d\dot\phi\wedge d^{c}\dot{\psi}\wedge\omega^{[n-2]}_\phi-\int_X\dot{\phi}\theta\wedge d(\uu(m_{\phi}))\wedge d^{c}\dot{\psi}\wedge\omega^{[n-2]}_\phi\\
=&-\int_X(d\dot{\phi},d\dot{\psi})_\phi \uu(m_{\phi})\theta\wedge\omega^{[n-1]}_\phi+\int_X(\theta,d\dot\phi\wedge d^{c}\dot{\psi})_\phi \uu(m_{\phi})\omega^{[n]}_\phi\\&-\int_X \dot{\phi}(d(\uu(m_{\phi})),d\dot{\psi})_\phi\theta\wedge\omega^{[n-1]}_\phi-\int_X\sum_{j=1}^{\ell}\dot{\phi}\uu_{,j}(m_{\phi})(dm_{\theta}^{\xi_j},d\dot{\psi})_\phi \omega^{[n]}_\phi,
\end{align*}
where we used that
\begin{align*}
&\theta\wedge d(\uu(m_\phi)) \wedge d^{c}\dot{\psi}\wedge \omega^{[n-2]}_\phi\\
=&(d(\uu(m_\phi)),d\dot{\psi})_\phi\theta \wedge \omega^{[n-1]}_\phi-(\theta,d(\uu(m_\phi))\wedge d^{c}\dot{\psi})_\phi\omega^{[n]}_\phi\\
=&(d(\uu(m_\phi)),d\dot{\psi})_\phi\theta \wedge \omega^{[n-1]}_\phi-\sum_{j=1}^{\ell}\uu_{,j}(m_\phi)(\theta,dm^{\xi_j}_\phi\wedge d^{c}\dot{\psi})_\phi\omega^{[n]}_\phi\\
=&(d(\uu(m_\phi)),d\dot{\psi})_\phi\theta \wedge \omega^{[n-1]}_\phi-\sum_{j=1}^{\ell}\uu_{,j}(m_{\phi})(dm_{\theta}^{\xi_j},d\dot{\psi})_\phi \omega^{[n]}_\phi.
\end{align*}
showing that
\begin{align}
\begin{split}\label{dA}
(\bdelta\mathcal{B}_{\uu}(\dot{\phi}))_\phi(\dot\psi)=&-\int_X \uu(m_{\phi})(d\dot{\phi},d\dot{\psi})_\phi \theta\wedge\omega^{[n-1]}_\phi\\
&-\int_X(d\dot{\phi},d\dot{\psi})_\phi\langle (d\uu)(m_{\phi}),m_{\theta}\rangle\omega^{[n]}_{\phi}\\&+\int_X(\theta,d\dot\phi\wedge d^{c}\dot{\psi})_\phi \uu(m_{\phi})\omega^{[n]}_\phi,
\end{split}
\end{align}
so that 
\begin{equation*}
(d\mathcal{B}_{\uu})_\phi(\dot{\phi},\dot{\psi})=(\bdelta\mathcal{B}_{\uu}(\dot{\phi}))_\phi(\dot\psi)-(\bdelta\mathcal{B}_{\uu}(\dot{\psi}))_\phi(\dot\phi)=0.
\end{equation*}
Thus, $\mathcal{B}_{\uu}$ is closed and therefore $\mathcal{E}_{\uu}^{\theta}:\mathcal{K}^{\T}_{\omega}\rightarrow\mathbb{R}$ is well-defined.
\end{proof}

\begin{defn}\label{Ent}
We let
\begin{equation*}
\mathcal{H}_{\uu}(\phi):=\int_X\log\left(\frac{\omega^{n}_\phi}{\omega^{n}}\right)\uu(m_{\phi})\omega^{[n]}_\phi
\end{equation*}
be the $\uu$-{\it entropy} functional $\mathcal{H}_{\uu}:\mathcal{K}^{\T}_{\omega}\to \mathbb{R}$.
\end{defn}

\begin{rem}
If $\tilde\mu$ is an absolutely continuous measure with respect to $\mu_\omega:=\omega^{[n]}$, then the entropy of $\tilde\mu$ relatively to $\mu$ is defined by,
\begin{equation*}
\Ent_{\mu_\omega}(\tilde\mu):=\int_X\log\left(\frac{d\tilde\mu}{d\mu_\omega}\right)d\tilde\mu.
\end{equation*}
The entropy is convex on the space of finite measures $\tilde{\mu}$ endowed with its natural affine structure. In the case when $\uu\in C^{\infty}(\PP,\mathbb{R}_{>0})$, the $\uu$-entropy functional in \Cref{Ent} is given by
\begin{equation*}
\mathcal{H}_{\uu}(\phi)=\Ent_{\mu}\left(\uu(m_{\phi})\omega^{[n]}_\phi\right)+c(\alpha,\uu)
\end{equation*}
for all $\phi\in\mathcal{K}^{\T}_\omega$, where $c(\alpha,\uu)=\int_X (\uu\log\circ \uu)(m_{\phi})\omega^{[n]}_\phi$ is a constant depending only on $(\alpha,\uu)$ (see \Cref{top-c}). 
\end{rem}

\begin{lem}\label{lm-m}
\begin{enumerate}
\item\label{lm-m-i} For any $\T$-invariant K\"ahler form $\omega$ on $X$, we have
\begin{equation*}
\Ric(\omega)(\xi,\cdot)=-\frac{1}{2}d\langle\Delta_\omega(m_\omega),\xi\rangle.
\end{equation*}
\item\label{lm-m-ii} For any $\phi\in\mathcal{K}^{\T}_{\omega}$ and $\xi\in \tor$, we have  
\begin{align*}
\Ric(\omega_\phi)=&\Ric(\omega)-\frac{1}{2}dd^{c}\Psi_\phi,\\
m_{\Ric(\omega_\phi)}^{\xi}=&m_{\Ric(\omega)}^{\xi}-\frac{1}{2}\big(d^{c}\Psi_\phi\big)(\xi),
\end{align*}
where $m_{\Ric(\omega)}:=\frac{1}{2}\Delta_\omega(m_\omega)$ is the $\Ric(\omega)$-momentum map of the action of $\T$ on $X$ and $\Psi_\phi=\log\left(\frac{\omega^{n}_\phi}{\omega^{n}}\right)$.
\end{enumerate}
\end{lem}

\begin{proof}
We give a simple argument for the sake of exposition. The statement \ref{lm-m-i} is well known (see e.g. \cite[Remark 8.8.2]{Gauduchon} and \cite[Lemma 28]{Sz-blow}).  For the statement \ref{lm-m-ii}, let $\phi\in\mathcal{K}^{\T}_{\omega}$ and $\xi\in\tor$. Using that $\mathcal{L}_{J\xi}\omega_\phi=-dd^{c}m^{\xi}_\phi$ we obtain
$$\mathcal{L}_{J\xi}\omega^{[n]}_\phi=\Delta_\phi(m^{\xi}_\phi)\omega^{[n]}_\phi.$$
It follows that
\begin{equation*}
-\frac{1}{2}(d^c\Psi_\phi)(\xi)=\frac{1}{2}\mathcal{L}_{J\xi}\Psi_\phi
=\frac{1}{2}\frac{\mathcal{L}_{J\xi}\omega^{[n]}_\phi}{\omega^{[n]}_\phi}-\frac{1}{2}\frac{\mathcal{L}_{J\xi}\omega^{[n]}}{\omega^{[n]}}=m_{\Ric(\omega_\phi)}^{\xi}-m_{\Ric(\omega)}^{\xi}.
\end{equation*}
\end{proof}

We now extend a formula obtained in the case $\uu=\vv=1$ by Chen-Tian (see \cite{chen,Tian-book}) to general values of $\uu$ and $\vv$.

\begin{thm}\label{Chen-Tian}
We have the following expression for the $(\uu,\vv)$-Mabuchi energy,
\begin{equation}\label{C-T}
\mathcal{M}_{\uu,\vv}=\mathcal{H}_{\uu}-2\mathcal{E}^{\Ric(\omega)}_{\uu}+c_{(\uu,\vv)}(\alpha)\mathcal{E}_{\vv}.
\end{equation}
\end{thm}

\begin{proof}
We compute
\begin{align*}
(d\mathcal{H}_{\uu})_{\phi}(\dot{\phi})=&-\int_X\dot\phi\Delta_\phi(\uu(m_{\phi}))\omega^{[n]}_\phi-\int_X(d\Psi_\phi,d\dot\phi)_{\phi}\uu(m_{\phi})\omega^{[n]}_\phi\\
=&-\int_X\sum_{j=1}^{\ell}\dot\phi \uu_{,j}(m_{\phi})\Delta_\phi(m_{\phi}^{\xi_j})\omega^{[n]}+\int_X\sum_{i,j=1}^{\ell}\dot{\phi}\uu_{,ij}(m_{\phi})(\xi_i,\xi_j)_\phi\omega^{[n]}_\phi\\
&+\int_X\sum_{j=1}^{\ell}\dot{\phi}(d\Psi_\phi,dm_{\phi}^{\xi_j})_\phi\uu_{,j}(m_{\phi})\omega^{[n]}_\phi-\int_X\dot{\phi}\Delta_\phi(\Psi_\phi)\uu(m_{\phi})\omega^{[n]}_\phi,
\end{align*}
where $\xxi:=(\xi_j)_{j=1,\cdots,\ell}$ is a basis for $\mathfrak{t}$. Using \Cref{lm-m} and the fact that 
\begin{equation*}
\Delta_\phi\left(\Psi_\phi\right)=-\Lambda_{\omega_\phi}dd^{c}\Psi_\phi=2\Lambda_{\omega_\phi}(\Ric(\omega_\phi)-\Ric(\omega))=\Scal_\phi-2\Lambda_{\omega_\phi}\Ric(\omega),
\end{equation*}
we get
\begin{align*}
(d\mathcal{H}_{\uu})_{\phi}(\dot{\phi})=&-\int_X\dot\phi \sum_{j=1}^{\ell}\uu_{,j}(m_{\phi})\Delta_\phi(m^{\xi_j}_{\phi})\omega^{[n]}+\int_X\sum_{i,j=1}^{\ell}\dot{\phi}\uu_{,ij}(m_{\phi})(\xi_i,\xi_j)_\phi\omega^{[n]}_\phi\\
&+\int_X\sum_{j=1}^{\ell}\dot{\phi}\left(\Delta_\omega(m_{\omega}^{\xi_j})-\Delta_\phi(m_{\phi}^{\xi_j})\right)\uu_{,j}(m_{\phi})\omega^{[n]}_\phi\\&-\int_X\dot{\phi}\left(\Scal_\phi-2\Lambda_{\omega_\phi}\Ric(\omega)\right)\uu(m_{\phi})\omega^{[n]}_\phi.
\end{align*}
It follows that
\begin{align}\label{H+E}
d(\mathcal{H}_{\uu}-2\mathcal{E}_{\uu}^{\Ric(\omega)})_{\phi}(\dot{\phi})=-\int_X\dot{\phi}\Scal_{\uu}(\phi)\omega^{[n]}_\phi,
\end{align}
which yields \eqref{C-T} via \eqref{H+E} and \eqref{E}.  
\end{proof}
By the work of Mabuchi \cite{Mabuchi1, Mabuchi2}, the space of $\T$-invariant K\"ahler potentials $\mathcal{K}^{\T}_\omega$ is an infinite dimensional riemannian manifold with a natural riemannian metric, called the {\it Mabuchi metric}, defined by
\begin{equation*}
\langle\dot{\phi}_1,\dot{\phi}_2\rangle_\phi=\int_X\dot{\phi}_1\dot{\phi}_2\omega^{[n]}_\phi,
\end{equation*}
for any $\dot{\phi}_1,\dot{\phi}_2\in T_\phi\mathcal{K}^{\T}_\omega$. The equation of a geodesic $(\phi_t)_{t\in[0,1]}\in\mathcal{K}^{\T}_\omega$ connecting two points $\phi_0,\phi_1\in\mathcal{K}^{\T}_\omega$ is given by 
\begin{equation*}
\ddot{\phi}_t=|d\dot{\phi}_t|^{2}_{\phi_t}.
\end{equation*}
\begin{prop}[\cite{Donaldson-geod, Guan}]
Let $X$ be a compact K\"ahler manifold with a fixed K\"ahler class $\alpha$, $\T\subset \Autred$ a real torus and suppose that $\omega\in\alpha$ is a $(\uu,\vv)$-cscK metric for smooth functions $\uu\in C^\infty(\PP,\R_{>0})$, $\vv\in C^\infty(\PP,\R)$ on the momentum image $\PP\subset\tor^*$ associated to $(\T,\alpha)$. Then for any $(\uu,\vv)$-cscK metric $\omega_\phi\in\alpha$ connected to $\omega$ by a geodesic segment in $\mathcal{K}^{\T}_\omega$, there exists $\Phi\in\Autred$ commuting with the action of $\T$, such that $\omega_\phi=\Phi^*\omega$.
\end{prop}

\begin{proof}
By a straightforward calculation using the formula \eqref{var-Scal-u-v} in the Appendix B, we obtain the following expression for the second variation of the $(\uu,\vv)$-Mabuchi energy along a $\T$-invariant segment of K\"ahler potentials $(\phi_t)_{t\in[0,1]}\in\mathcal{K}^{\T}_\omega$:
\begin{align}
\begin{split}\label{eq-Mab-pp}
\frac{d^2\mathcal{M}_{\uu,\vv}(\phi_t)}{dt^2}=&2\int_X|D^-d\dot{\phi}_t|_{\phi_t}^{2}\uu(m_{\phi_t})\omega_{\phi_t}^{[n]}\\&-\int_X\left(\ddot{\phi}_t-|d\dot{\phi}_t|_{\phi_t}^{2}\right)\left(\Scal_{\uu}(\phi_t)-\vv(m_{\phi_t})\right)\omega_{\phi_t}^{[n]}.
\end{split}
\end{align}
Suppose now that $\omega_\phi$, $\phi\in\mathcal{K}^{\T}_\omega$ is a $(\uu,\vv)$-cscK metric connected to $\omega$ by a smooth geodesic $(\phi_t)_{t\in[0,1]}$, such that $\phi_0=0$ and $\phi_1=\phi$. Then $\left.\frac{d\mathcal{M}_{\uu,\vv}(\phi_t)}{dt}\right|_{t=0}=\left.\frac{d\mathcal{M}_{\uu,\vv}(\phi_t)}{dt}\right|_{t=1}=0$, and using \eqref{eq-Mab-pp} we obtain
\begin{equation*}
\frac{d^2\mathcal{M}_{\uu,\vv}(\phi_t)}{dt^2}=2\int_X|D^-d\dot{\phi}_t|_{\phi_t}^{2}\uu(m_{\phi_t})\omega_{\phi_t}^{[n]}\geq 0.
\end{equation*} 
It follows that $\frac{d^2\mathcal{M}_{\uu,\vv}(\phi_t)}{dt^2}\equiv0$ and $D^-d\dot{\phi}_t\equiv0$. Thus, we have a family of real holomorphic vector vector fields $V_t:=-\grad_{g_t}\dot{\phi}_t$, $t\in[0,1]$. By \cite[Proposition 4.6.3]{Gauduchon}, $V_t=V_0$ for all $t$, and $\omega_\phi=\big(\Phi_{1}^{V_0}\big)^*\omega$ where $\Phi_{t}^{V_0}\in \Autred$ is the flow of the real holomorphic vector field $V_0$.
\end{proof}
\begin{comment}
\begin{cor}
If $(X,\T)$ is a compact toric K\"ahler manifold, then the $(\uu,\vv)$-cscK metrics with $\uu\in C^\infty(\PP,\R_{>0})$, $\vv\in C^\infty(\PP,\R)$, are unique modulo ${\rm Aut}_{\rm red}^{\T}(X)$ the $\T$-equivariant automorphisms of $X$.
\end{cor}
\end{comment}
\begin{rem}
In general, the space $\mathcal{K}^{\T}_\omega$ is not geodesically connected by smooth geodesics (see \cite[Theorem 1.2]{Darvas}). However, by a result of Chen \cite{Chen0}, the space $\mathcal{K}^{\T}_\omega$ is geodesically connected by $\T$-invariant weak $C^{1,1}$-geodesics, i.e. in the space $\big(\mathcal{K}^{1,1}_\omega\big)^{\T}$ of $\T$-invariant real valued functions $\phi$ such that $\omega+dd^{c}\phi$ is a positive current with bounded coefficients. Using the formula $m_\phi=m_\omega+d^{c}\phi$ and \Cref{Chen-Tian}, one can extend the $(\uu,\vv)$-Mabuchi energy to a functional $\mathcal{M}_{\uu,\vv}:\big(\mathcal{K}^{1,1}_\omega\big)^{\T}\rightarrow\mathbb{R}$. One thus might hope to obtain a uniqueness up to a $\T$-equivariant isometry of $(\uu,\vv)$-cscK metrics along the lines of the proof of \cite[Theorem 1.1]{BB}, but this goes beyond the scope of this article.
\end{rem}

\subsection{The relative $(\uu,\vv)$-Mabuchi energy}
In this section we assume that both $\uu$ and $\vv$ are positive smooth functions on $\PP$.
\begin{defn}\label{def-Mab-rel}
The $(\uu,\vv)$-relative Mabuchi energy $\mathcal{M}^{\rm rel}_{\uu,\vv}:\mathcal{K}^{\T}_\omega\to\mathbb{R}$ is defined by 
\begin{equation}
\begin{cases} (d\mathcal{M}^{\rm rel}_{\uu,\vv})_{\phi}(\dot{\phi})={\displaystyle-\int_X\dot{\phi}\big(\Scal_{\uu}(\phi)/\vv(m_{\phi})-\vv_{\rm ext}(m_{\phi})\big)\vv(m_{\phi})\omega^{[n]}_{\phi}},\\ 
\mathcal{M}^{\rm rel}_{\uu,\vv}(0)=0,
\end{cases}
\end{equation}
for any $\dot{\phi}\in T_\phi\mathcal{K}^{\T}_\omega$, where $\vv_{\rm ext}$ is the affine linear function on $\PP$ defined in Section 3.2.
\end{defn}

\begin{lem}\label{rem-Mab-rel}
We have $\mathcal{M}^{\rm rel}_{\uu,\vv}=\mathcal{M}_{\uu,\vv\vv_{\rm ext}}$.
\end{lem}
\begin{proof}
In Section 3.2, we showed that $c_{\uu,\vv\vv_{\rm ext}}(\alpha)=1$. From the definitions of $\mathcal{M}_{\uu,\vv}$ and $\mathcal{M}_{\uu,\vv}^{\rm rel}$, it then follows that $\mathcal{M}^{\rm rel}_{\uu,\vv}=\mathcal{M}_{\uu,\vv\vv_{\rm ext}}+c$ and using $\mathcal{M}^{\rm rel}_{\uu,\vv}(0)=\mathcal{M}_{\uu,\vv\vv_{\rm ext}}(0)=0$ we get $c=0$.
\end{proof}
\subsection{Boundednes of the $(1,\vv)$-Mabuchi energy}
Now we show how the results of Berman-Berndtsson in \cite{BB} can be extended to the $(1,\vv)$-cscK metrics.
\begin{thm}\cite{BB}\label{Mab-bound-1-w}
Let $X$ be a smooth compact K\"ahler manifold, $\T \subset {\rm Aut}_{\rm red}(X)$ a real torus, and suppose that $X$ admits a $(1,\vv)$-cscK metric $\omega$ in the the K\"ahler class $\alpha$ for some smooth function $\vv$ on the momentum image $\PP\subset \tor^*$ associated to $(\T, \alpha)$. Then, $\omega$ is a global minima of ${\mathcal M}_{1,\vv}$.
\end{thm}
\begin{proof}
We denote by $\mathcal{M}_{\vv}$ the $(1,\vv)$-Mabuchi energy and by $\mathcal{M}$ the $(1,c_{1,\vv}(\alpha))$-Mabuchi energy. From the definition of the Mabuchi energy we have the following relation
\begin{equation*}
\mathcal{M}_{\vv}=\mathcal{M}+\mathcal{E}_{\tilde{\vv}},
\end{equation*}
where $\tilde{\vv}:=c_{1,\vv}(\alpha)(1-\vv)$ and $\mathcal{E}_{\tilde{\vv}}$ is the functional \eqref{E}. Let $\phi_0,\phi_1\in\mathcal{K}^{\T}_\omega$ be two smooth K\"ahler potentials and $\phi_t$ the weak geodesic connecting $\phi_0$ and $\phi_1$ (see \cite{BB, chen-tian} and the references therein for the definition of a weak geodesic). By \cite[Proposition 10.d]{Bern} the function $t\mapsto \mathcal{E}_{\tilde{\vv}}(\phi_t)$ is affine on $[0,1]$, whereas by \cite[Theorem 3.4]{BB}, the function $t\mapsto \mathcal{M}(\phi_t)$ is convex. It follows that $t\mapsto \mathcal{M}_{\vv}(\phi_t)$ is convex. By \cite[Lemma 3.5]{BB} and its proof, we get 
\begin{equation*}
\underset{t\rightarrow 0^+}{\lim}\frac{\mathcal{M}_{\vv}(\phi_1)-\mathcal{M}_{\vv}(\phi_0)}{t}\geq \int_X\big(\Scal(\phi_0)-c_{1,\vv}(\alpha)\vv(m_{\phi_0})\big)\dot{\phi}\omega_{\phi_0}^{[n]}.
\end{equation*}
where $\dot{\phi}:=\frac{d\phi_t}{dt}_{|t=0^+}$. Using the sub-slope inequality for convex functions and the Cauchy--Shwartz inequality we get 
\begin{align*}
\mathcal{M}_{\vv}(\phi_1)-\mathcal{M}_{\vv}(\phi_0)\geq& \underset{t\rightarrow 0^+}{\lim}\frac{\mathcal{M}_{\vv}(\phi_1)-\mathcal{M}_{\vv}(\phi_0)}{t}\\
\geq& \int_X\big(\Scal(\phi_0)-c_{1,\vv}(\alpha)\vv(m_{\phi_0})\big)\dot{\phi}\omega_{\phi_0}^{[n]}\\
\geq&-d(\phi_0,\phi_1)\Big(\int_X\big(\Scal(\phi_0)-c_{(1,\vv)}(\alpha)\vv(m_{\phi_0}))^{2}\omega_{\phi_0}^{[n]}\Big)^{\frac{1}{2}},
\end{align*} 
where $d(\phi_0,\phi_1)^{2}=\int_X\dot{\phi}^{2}\omega_{\phi_0}^{[n]}$ is the Mabuchi distance between $\phi_0$ and $\phi_1$. In particular, if $\omega_{\phi_0}$ is a $(1,\vv)$-cscK metric in the K\"ahler class $\alpha$, then $\mathcal{M}_{\vv}(\phi)\geq\mathcal{M}_{\vv}(\phi_0)$ 
for any $\phi\in\mathcal{K}^{\T}_\omega$.
\end{proof}

\section{The $(\uu,\vv)$-Futaki invariant for a K\"ahler class}\label{sec-6}
Let $(X,\alpha)$ be a compact K\"ahler manifold and $\T\subset {\rm Aut}_{\rm red}(X)$ a real torus with momentum polytope $\PP$ with respect to $\alpha$ as in \Cref{lem-mom-norm}. For any $\phi\in\mathcal{K}_{\omega}^{\T}$ and $V\in\mathfrak{h}_{{\rm red}}^{\T}$ in the Lie algebra of the centralizer of $\T$ in $\Autred$, we denote by  $h_{\phi}^{V}+\sqrt{-1}f_{\phi}^{V}\in C^{\infty}_{0,\phi}(X,\mathbb{C})$ the normalized holomorphy potantial of $\xi$, i.e. $h^{V}_{\phi}$ and $f^{V}_{\phi}$ are smooth functions such that, 
\begin{align*}
&V=\grad_{g_\phi}(h_{\phi}^{V})+J\grad_{g_\phi} (f_{\phi}^{V}),\\
&\int_X f_{\phi}^{V}\omega^{[n]}_\phi=\int_X h_{\phi}^{V}\omega^{[n]}_\phi=0.
\end{align*}
Using that the tangent space in $\phi$ of $\mathcal{K}_{\omega}^{\T}$ is given by $T_{\phi}\big(\mathcal{K}_{\omega}^{\T})\cong C^{\infty}_{0,\phi}(X,\mathbb{R})^{\T}\oplus\R$, the vector field $JV$ defines a vector field $\widehat{JV}$ on $\mathcal{K}_{\omega}^{\T}$, given by:
\begin{equation*}
\phi\mapsto\mathcal{L}_{JV}\omega_\phi=-dd^{c}f_{\phi}^{V},
\end{equation*}
so that $\widehat{JV}_{\phi}=f_{\phi}^{V}$. We consider the 1-form $\sigma$ on $\mathcal{K}_{\omega}^{\T}$, defined by
\begin{equation*}
\sigma_{\phi}(\dot{\phi}):= \left(d\mathcal{M}_{\uu,\vv}\right)_\phi(\dot{\phi})
\end{equation*}
where $\mathcal{M}_{\uu,\vv}$ is the $(\uu,\vv)$-Mabuchi energy associated to the smooth functions $\uu\in C^{\infty}(\PP,\mathbb{R}_{>0})$ and $\vv\in C^{\infty}(\PP,\mathbb{R})$ (see \eqref{Mabuchi}). By the invariance of $\sigma$ under the ${\rm Aut}_{\rm red}^{\T}(X)$-action and Cartan's formula, we get
\begin{equation*}
\mathcal{L}_{\widehat{JV}}\sigma=d\big(\sigma(\widehat{JV})\big)=0.
\end{equation*}
Then $\phi\mapsto\sigma_{\phi}(\widehat{JV})$ is constant on $\mathcal{K}_{\omega}^{\T}$, and we define
\begin{defn}\label{Def-Fut}
We let 
\begin{equation}\label{Fut-uv}
\mathcal{F}_{\uu,\vv}^{\alpha}(V):=\sigma_\omega(\widehat{JV})=\int_X\big(\Scal_{\uu}(\omega)-c_{(\uu,\vv)}(\alpha)\vv(m_{\omega})\big)f_{\omega}^{V}\omega^{[n]},
\end{equation}
be the real constant associated to $V\in\mathfrak{h}_{\rm red}^{\T}$. We thus get a linear map $\mathcal{F}^{\alpha}_{\uu,\vv}:\mathfrak{h}_{\rm red}^{\T}\to \mathbb{R}$ called the $(\uu,\vv)$-Futaki invariant associated to $(\alpha,\PP,\uu,\vv)$.
\end{defn}
By its very definition, we have
\begin{prop}\label{fut-obs}
If $(X,\alpha,\T)$ admits a $(\uu,\vv)$-cscK metric then 
\begin{equation}\label{eq-fut-obs}
\int_X\Scal_{\uu}(\omega)\omega^{[n]}=c_{(\uu,\vv)}(\alpha)\int_X\vv(m_{\omega})\omega^{[n]}\text{ and }\mathcal{F}^{\alpha}_{\uu,\vv}\equiv 0.
\end{equation}
\end{prop}
\begin{rem}
The first condition in \eqref{eq-fut-obs} is satisfied when $\int_X\vv(m_\omega)\omega^{[n]}\neq 0$ by the very definition of $c_{\uu,\vv}(\alpha)$ (see \Cref{Def-Top-const}). Furthermore, in the case of a $(\uu,\vv)$-extremal K\"ahler metric considered in Section 3.1, the both conditions in \eqref{eq-fut-obs} hold true with respect to the weights $\uu$ and $\vv\vv_{\rm ext}$, by the  very definition of $\vv_{\rm ext}$.
\end{rem}
\section{The $(\uu,\vv)$-Futaki invariant of a smooth test configuration}\label{sec-7}
Let $X$ be a compact K\"ahler manifold endowed with an $\ell$-dimensional real torus $\T\subset {\rm Aut}_{\rm red}(X)$ and a K\"ahler class $\alpha\in H^{1,1}(X,\mathbb{R})$. 
Following \cite{Dervan-Ross, Dyr1, Dyr2} we give the following
\begin{defn}
A smooth $\T$-compatible K\"ahler test configuration for $(X,\alpha)$ is a compact smooth $(n+1)$-dimensional K\"ahler manifold $(\mathcal{X},\mathcal{A})$, endowed with a holomorphic action of a real torus $\hat{\T}\subset {\rm Aut}_{\rm red}(\mathcal{X})$ with Lie algebra $\hat{\tor}$ and
\begin{itemize}
\item a surjective holomorphic map $\pi:\mathcal{X}\rightarrow \mathbb{P}^{1}$ such that the torus action $\hat{\T}$ on $\mathcal{X}$ preserves each fiber $X_\tau:=\pi^{-1}(\tau)$ and $(X_1,\mathcal{A}_{|X_1},\hat{\T})\cong (X,\alpha,\T)$,
\item a $\mathbb{C}^{\star}$-action $\rho$ on $\mathcal{X}$ commuting with $\hat{\T}$ and covering the usual $\mathbb{C}^{\star}$-action on $\mathbb{P}^{1}$,
\item a biholomorphism  
\begin{equation}\label{lambda}
\lambda : \mathcal{X}\setminus X_0\simeq X\times\left(\mathbb{P}^{1}\setminus \{0\}\right),
\end{equation}
which is equivariant with respect to the actions of $\hat{\mathbb{G}}:=\hat{\T}\times \mathbb{S}^{1}_{\rho}$ on $\mathcal{X}\setminus X_0$ and the action of $\mathbb{G}:=\T\times \mathbb{S}^{1}$ on $X\times\left(\mathbb{P}^{1}\setminus \{0\}\right)$. 
\end{itemize}
\end{defn}
In what follows we shall tacitly identify $\hat{\T}$ with $\T$ and $\hat{\mathbb{G}}$ with $\mathbb{G}$.

\begin{defn}\label{Test-triv-prod}
A {\it smooth $\T$-compatible K\"ahler test configuration} $(\mathcal{X},\mathcal{A},\rho,\T)$ for $(X,\alpha,\T)$ is called 
\begin{itemize}
\item {\it trivial} if it is given by $(\mathcal{X}_0=X\times \mathbb{P}^{1},\mathcal{A}_0=\pi^{*}_X\alpha+\pi^{*}_{\mathbb{P}^{1}}[\FS],\T)$ and $\mathbb{C}^\star$-action $\rho_{0}(\tau)(x,z)=(x,\tau z)$ for any $\tau\in\mathbb{C}^{\star}$ and $(x,z)\in X\times\mathbb{P}^{1}$.
\item {\it product} if it is given by $(\mathcal{X}_{\rm prod},\mathcal{A}_{\rm prod},\rho_{\rm prod},\T)$ where $\mathcal{X}_{\rm prod}$ is the compactification (in the sense of \cite{odaka1, wang}, see also \cite[Example 2.8]{Bouk} and \cite[p. 12-13]{Mcduff-Tolman}) of $X\times\mathbb{C}$ with $\mathbb{C}^\star$-action $\rho_{\rm prod}(\tau)(x,z)=(\rho_X(\tau)x,\tau z)$ where $\rho_X$ is a $\mathbb{C}^\star$-action on $X$ and $\mathcal{A}_{\rm prod}$ is a K\"ahler class on $\mathcal{X}_{\rm prod}$ which restricts to $\alpha$ on $X_1\cong X$.  
\end{itemize}
\end{defn}

Let $(\mathcal{X},\mathcal{A},\T)$ be a smooth $\T$-compatible K\"ahler test configuration for $(X,\alpha,\T)$ and $\Omega\in\mathcal{A}$ a $\mathbb{G}$-invariant K\"ahler form. The action of $\T$ on $\mathcal{X}$ is Hamiltonian with $\Omega$-momentum map $m_{\Omega}:\mathcal{X}\rightarrow\mathfrak{t}^{*}$, normalized by $m_{\Omega}(X_1)=\PP$, where $\PP$ is a fixed momentum polytope for the induced $\T$-action on $X_1\cong X$. 

For any $\tau\in\mathbb{C}^{\star}$, we denote by 
\begin{equation}\label{not}
\Omega_{\tau}:=\Omega_{|X_\tau},\,\,\,\Omega_1=:\omega \text{ and }\,\omega_\tau:=\rho(\tau)^{*}\Omega_{\tau},
\end{equation}
where $\rho(\tau):X_1\overset{\sim}{\rightarrow} X_\tau$ is the restriction of $\rho(\tau)\in {\rm Aut}_{\rm red}(\mathcal{X})$ to $X_1$. The action of $\T$ on $X_\tau$ is Hamiltonian with $\Omega_{\tau}$-momentum map $(m_{\Omega})_{|X_\tau}$. Pulling the structure on $X_\tau$ back to $X_1$ via $\rho(\tau)$, we get a $\omega_{\tau}$-momentum map for the $\T$-action on $X_1$, given by 

\begin{equation}\label{m-tau}
m_\tau=m_{\Omega_\tau}\circ \rho(\tau).
\end{equation}

\begin{lem}\label{lm-mom}
For any $\tau\in\mathbb{C}^{\star}$, we have
\begin{equation*}
\int_{X_\tau}m_{\Omega_\tau}\Omega_{\tau}^{[n]}=\int_{X_1}m_{\tau}\omega_{\tau}^{[n]}=\int_{X_1}m_1\omega^{[n]}.
\end{equation*}
It follows that $\PP_\tau=\PP$ for any $\tau\in\mathbb{C}^{\star}$, where $\PP_\tau:=m_{\Omega}(X_\tau)=m_\tau(X_1)$ is the momentum polytope of the induced action of $\T$ on $X_\tau$ and $\Omega_{\tau}:=\Omega_{|X_\tau}$.
\end{lem}

\begin{proof}
Since $\Omega$ is $\mathbb{S}^{1}_{\rho}$-invariant, the following integral depends only on $t=-\log{|\tau|}$,
\begin{equation*}
\int_{X_\tau}m_{\Omega_\tau}\Omega_{\tau}^{[n]}=\int_{X_1}m_\tau\omega^{[n]}_\tau=\int_{X_1}m_\tau(\omega+dd^{c}\phi_\tau)^{[n]}.
\end{equation*}
Let $V_\rho$ be the generator of the $\mathbb{S}^{1}_\rho$-action. By \eqref{m-tau} we have
\begin{equation*}
\frac{d}{dt}m_\tau=\frac{d}{dt}(m_\Omega\circ\varphi^{t}_{\mathcal{J}V_\rho})
=(\varphi^{t}_{\mathcal{J}V_\rho})^{*}(dm_\Omega,\mathcal{J} V_\rho)_{\Omega}=-(\varphi^{t}_{\mathcal{J}V_\rho})^{*}(dm_\Omega,dh^{\rho})_{\Omega},
\end{equation*}
where $\mathcal{J}$ denotes the complex structure on $\mathcal{X}$, $\varphi^{t}_{\mathcal{J}V_\rho}=\rho(e^{-t})$ is the flow of $\mathcal{J}V_\rho$ and $h^{\rho}$ is a $\Omega$-Hamiltonian function for $V_\rho$. 
On the other hand, we have
\begin{equation*}
\frac{d}{dt}\rho(\tau)^{*}\Omega=(\varphi_{\mathcal{J}V_\rho}^{t})^{*}\mathcal{L}_{\mathcal{J}V_\rho}\Omega=-(\varphi_{\mathcal{J}V_\rho}^{t})^{*}dd^ch^\rho.
\end{equation*}
It follows that 
\begin{align*}
\frac{d}{dt}\int_{X_1}m_\tau\omega^{[n]}_\tau=&\frac{d}{dt}\int_{X_1}m_\tau((\rho(\tau)^{*}\Omega)_{|X_1})^{[n]}\\
=&-\int_{X_\tau}((dm_\Omega,dh^\rho)_{\Omega})_{|X_\tau}\Omega_{\tau}^{[n]}-\int_{X_\tau}m_{\Omega_\tau} dd^{c}h^{\rho}_{|X_\tau} \wedge\Omega_{\tau}^{[n-1]}\\
=&-\int_{X_\tau}((dm_\Omega,dh^\rho)_{\Omega})_{|X_\tau}\Omega_{\tau}^{[n]}+\int_{X_\tau}m_{\Omega_\tau} \Delta_{\Omega_{\tau}}\big(h_{|X_\tau}^\rho\big) \Omega_{\tau}^{[n]}\\
=&-\int_{X_\tau}((dm_\Omega,dh^\rho)_{\Omega})_{|X_\tau}\Omega_{\tau}^{[n]}+\int_{X_\tau}\big(dm_{\Omega_\tau},
dh_{|X_\tau}^\rho\big)_{\Omega_{\tau}} \Omega_{\tau}^{[n]}=0, 
\end{align*}
where we have used that $((dm_\Omega,dh^\rho)_{\Omega})_{|X_\tau}=\big(dm_{\Omega_\tau},
dh_{|X_\tau}^\rho\big)_{\Omega_{\tau}}$
since the symplectic gradient of $m_\Omega:\mathcal{X}\rightarrow\tor^{*}$ is given by the $\tor$-valued fundamental vector field for the $\T$-action, and thus is tangent to the fibers. It follows that
\begin{equation*}
\int_{X_1}m_\tau\omega^{[n]}_\tau=\int_{X_1}m_1\omega^{[n]}.
\end{equation*}
\end{proof}

Since $m_\Omega:\mathcal{X}\rightarrow\tor^*$ is continuous it follows from  \Cref{lm-mom} that $m_{\Omega}(\mathcal{X})=\PP$.

\begin{defn}\label{Def-Fut-TC}
Let $(\mathcal{X},\mathcal{A},\T)$ be a smooth $\T$-compatible K\"ahler test configuration for the compact K\"ahler manifold $(X,\alpha)$ and $\uu\in C^{\infty}(\PP,\mathbb{R}_{>0})$, $\vv\in C^{\infty}(\PP,\mathbb{R})$. The $(\uu,\vv)$-{\it Futaki invariant} of $(\mathcal{X},\mathcal{A},\T)$ is defined to be the real number
\begin{align}
\begin{split}\label{Fut-TC}
\mathcal{F}_{\uu,\vv}(\mathcal{X},\mathcal{A})=&-\int_{\mathcal{X}}\big(\Scal_{\uu}(\Omega)-c_{(\uu,\vv)}(\alpha)\vv(m_{\Omega})\big)\Omega^{[n+1]}\\&+2\int_{\mathcal{X}}\uu(m_{\Omega})\pi^{\star}\FS\wedge \Omega^{[n]}
\end{split}
\end{align}
where $\Omega\in\mathcal{A}$ is a $\T$-invariant representative of $\mathcal{A}$, $\FS$ is the Fubini-Study metric on $\mathbb{P}^1$ with $\Ric(\FS)=\FS$, and $c_{(\uu,\vv)}(\alpha)$ is the $(\uu,\vv)$-slope of $(X,\alpha)$ given by \eqref{Top-const}.
\end{defn}
\begin{rem}\label{rem-Fut-sym}
\begin{enumerate}
\item\label{rem-Fut-sym-i} By \Cref{top-c}, \eqref{Fut-TC} is independent from the choice of a $\T$-invariant K\"ahler form $\Omega\in\mathcal{A}$. For $\uu=\vv\equiv 1$ we also recover the Futaki invariant of a smooth test configuration introduced in \cite{Dervan-Ross, Dyr1, Dyr2}.
\item\label{rem-Fut-sym-ii}  It is easy to show that 
\begin{comment}
Let $\Omega\in\mathcal{A}$ and $\omega\in\alpha$ be $\T$-invariant representatives. We denote by $\accentset{\circ}{\Omega}$ the two form on $\mathcal{X}\setminus X_0$ given given by $\Omega_{|T_{\hat x}X_{\pi(\hat{x})}}$ on $T_{\hat x}X_{\pi(\hat{x})}$, and by $0$ on $\mathcal{H}_{\hat{x}}:=(T_{\hat x}X_{\pi(\hat{x})})^{\perp_{\Omega_{\hat{x}}}}\cong T_{\pi(\hat{x})}\mathbb{P}^{1}$, for any $\hat{x}\in\mathcal{X}\setminus X_0$. Then there exist a $\mathbb{T}$-invariant function on $\mathcal{X}\setminus X_0$ such that $\Omega=\accentset{\circ}{\Omega}+Z\pi^{*}\FS$.
\end{comment}
\begin{align*}
2\int_{\mathcal{X}}\uu(m_{\Omega})\pi^{*}\FS\wedge \Omega^{[n]}=& 2\int_{\mathcal{X}\setminus X_0}\uu(m_{\Omega})\pi^{*}\FS\wedge \Omega^{[n]} \\
=&2\int_{\mathbb{P}^{1}\setminus\{0\}}\left(\int_{X_\tau}\uu(m_{\Omega_\tau})\Omega^{[n]}_\tau\right)\FS\\
=&2\vol(\mathbb{P}^{1})\Big(\int_{X_1}\uu(m_{\omega})\omega^{[n]}\Big)\\
=&(8\pi)\int_X\uu(m_{\omega})\omega^{[n]}, 
\end{align*}
where for passing from the second line to the third line we used that $\rho(\tau)^*\Omega_\tau$ and $\omega$ are in the same K\"ahler class $\mathcal{A}_{|X_1}$ on $X_1$, see \Cref{top-c}. Thus, we obtain the following equivalent expression for the $(\uu,\vv)$-Futaki invariant
\begin{align}\label{Fut-int+intX}
\begin{split}
\mathcal{F}_{\uu,\vv}(\mathcal{X},\mathcal{A})=&-\int_{\mathcal{X}}\big(\Scal_{\uu}(\Omega)-c_{(\uu,\vv)}(\alpha)\vv(m_{\Omega})\big)\Omega^{[n+1]}\\&+(8\pi)\int_X\uu(m_{\omega})\omega^{[n]}.
\end{split}
\end{align} 
\item It is easy to compute the $(\uu,\vv)$-Futaki invariant of the trivial test configuration $(\mathcal{X}_0,\mathcal{A}_0)$ (see \Cref{Test-triv-prod}), using that for a product K\"ahler form $\Omega_0:=\pi^{*}_X\omega+\pi_{\mathbb{P}^1}^{*}\FS$ we have $\Scal_{\uu}(\Omega_0)=\Scal_{\uu}(\omega)+2\uu(m_{\omega})$, then \eqref{Fut-TC} reduces to
\begin{equation*}
\mathcal{F}_{\uu,\vv}(\mathcal{X}_0,\mathcal{A}_0)=-4\pi\int_X\big(\Scal_{\uu}(\omega)-c_{(\uu,\vv)}(\alpha)\vv(m_{\omega})\big)\omega^{[n]}.
\end{equation*} 
\end{enumerate}
\end{rem}
\begin{defn}\cite{Dervan, Dyr2}
We say that $(X,\alpha,\T)$  is
\begin{enumerate}
\item $(\uu,\vv)$-{\it K-semistable} on smooth K\"ahler test configurations if $\mathcal{F}_{\uu,\vv}(\mathcal{X},\mathcal{A})\geq 0$ for any $\T$-compatible test configuration $(\mathcal{X},\mathcal{A},\T)$ of $(X,\alpha,\T)$ and $\mathcal{F}_{\uu,\vv}(\mathcal{X}_0,\mathcal{A}_0)=0$ for the trivial test configuration $(\mathcal{X}_0,\mathcal{A}_0)$.
\item $(\uu,\vv)$-{\it K-stable} on smooth  K\"ahler test configuations if it is $(\uu,\vv)$-{\it K-semistable} and $\mathcal{F}_{\uu,\vv}(\mathcal{X},\mathcal{A})=0$ if and only if $(\mathcal{X},\mathcal{A})=(\mathcal{X}_{\rm prod},\mathcal{A}_{\rm prod})$ is a product in the sense of \Cref{Test-triv-prod}.
\end{enumerate} 
\end{defn}

Following \cite{Dervan-Ross, Dyr1}, there is a family of $\T$-invariant K\"ahler potentials $\phi_\tau\in\mathcal{K}^{\T}_\omega(X_1)$, $\tau\in\mathbb{C}^*\subset\mathbb{P}^1$ given by the following Lemma.

\begin{lem}\label{lem-hat-Omega}
Let $\Omega\in\mathcal{A}$ be a $\mathbb{G}$-invariant K\"ahler form on $\mathcal{X}$. 
\begin{enumerate}
\item\label{lem-hat-Omega-i} On $\mathcal{X}^{\star}:=\mathcal{X}\setminus X_0$ we have 
\begin{equation}\label{Omega-Phi}
\Omega=\hat\omega+dd^{c}\Phi,
\end{equation}
where $\hat{\omega}:=(\pi_X\circ\lambda)^{\star}\omega$ with $\lambda$ the map given by \eqref{lambda} and $\pi_X$ is the projection on the first factor of $X\times (\mathbb{P}^1\setminus \{0\})$, and $\Phi$ is a smooth $\mathbb{G}$-invariant function on $\mathcal{X}^{\star}$, such that for any $\tau\in\mathbb{C}^*$, 
\begin{equation}\label{rho=ddc}
\phi_\tau:=\rho(\tau)^*(\Phi_{\mid X_\tau})\in \mathcal{K}^{\T}_\omega(X_1),
\end{equation}
satisfies
\begin{equation*}
\omega_\tau-\omega=dd^{c}\phi_\tau,
\end{equation*}
where we recall that $\omega_\tau$ is defined in \eqref{not}.
\item\label{lem-hat-Omega-iii} $m_{\hat{\omega}}^{\xi}:=m_{\Omega}^{\xi}-(d^c\Phi)(\xi)$, $\xi\in\tor$ is a moment map of $\hat{\omega}$ restrected to a fiber $X_\tau$ for the $\T$-action on $\mathcal{X}^\star$, satisfying $m_{\hat{\omega}}(\mathcal{X}^{\star})=\PP$. 
\end{enumerate}
\end{lem} 
\begin{proof}
(i) Using \cite[Proposition 3.10]{Dyr1} we can find a smooth function $\Phi$ on $\mathcal{X}^{\star}$ such that $\Omega=\hat{\omega}+dd^{c} \Phi$ on $\mathcal{X}^\star$. Taking the restriction of the latter equality to $X_\tau$ ($\tau\neq 0$) we have $\Omega_{\tau}=\rho(\tau^{-1})^{*}\omega+dd^{c}(\Phi_{|X_\tau})$, pulling back by $\rho(\tau)$ yields $\omega_\tau-\omega=dd^{c}\phi_\tau$.\\

(ii) By the relation \eqref{Omega-Phi} and the fact that the action of $\T$ preserves the fibers we obtain that $m_{\hat\omega}^{\xi}:=m_{\Omega}^{\xi}-(d^{c}\Phi)(\xi)$ is a momentum map of $(X_\tau,\hat{\omega}_{|X_\tau})$. It thus follows from Lemmas \ref{lem-mom-norm} and \ref{lm-mom} that $m_{\hat{\omega}}(X_\tau)=\PP$.
\end{proof}

The main result of this section is the following theorem which extends the results from \cite{Dervan-Ross, Dyr1} to arbitrary values of $\uu,\vv$:
\begin{thm}\label{Mabuchi-Slope}
Let $(\mathcal{X},\mathcal{A},\T)$ be a smooth $\T$-compatible K\"ahler test configuration, for a compact K\"ahler manifold $(X,\alpha,\T)$ and $\uu\in C^{\infty}(\PP,\mathbb{R}_{>0})$, $\vv\in C^{\infty}(\PP,\mathbb{R})$ are weight functions. If the central fiber $X_0$ is reduced, then 
\begin{equation*}
\underset{t\rightarrow +\infty}{\lim}\frac{\mathcal{M}_{\uu,\vv}(\phi_t)}{t}=\mathcal{F}_{\uu,\vv}(\mathcal{X},\mathcal{A}).
\end{equation*}
where $\phi_t:=\phi_{\tau}$ with $\tau=e^{-t+is}$ is given by \eqref{rho=ddc}. In particular if $\mathcal{M}_{\uu,\vv}$ is bounded from bellow, then $$\mathcal{F}_{\uu,\vv}(\mathcal{X},\mathcal{A})\geq 0.$$
\end{thm}
Before we give the proof we need a couple of technical lemmas.
\begin{lem}\label{lim-E}
Under the hypotheses of \Cref{Mabuchi-Slope} we have,
\begin{equation}\label{eq-lim-E}
\underset{t\rightarrow +\infty}{\lim}\dfrac{\mathcal{E}_{\vv}(\phi_t)}{t}=\int_{\mathcal{X}} \vv(m_{\Omega})\Omega^{[n+1]}.
\end{equation}
\end{lem}
\begin{proof}
We will start by showing as in \cite{Dervan-Ross, Dyr1, tian} that,
\begin{equation}\label{pi-Omega-n}
\pi_{\star}\big(\vv(m_{\Omega})\Omega^{[n+1]}\big)=dd^{c}\mathcal{E}_{\vv}(\phi_\tau),
\end{equation}
on $\mathbb{C}^{\star}\subset\mathbb{P}^1$, in the sens of currents. From the very definition of the functional $\mathcal{E}_{\vv}$ (see \eqref{E}) we have 
\begin{align*}
\mathcal{E}_{\vv}(\phi_\tau)=&\int^{1}_0 \left(\int_X\phi_\tau \vv( m_{\epsilon\phi_\tau})\omega_{\epsilon\phi_\tau}^{[n]}\right)d\epsilon\\
=&\int^{1}_0 \left(\int_X\phi_\tau  \vv(\epsilon m_{\tau}+(1-\epsilon)m_\omega)(\epsilon\omega_{\tau}+(1-\epsilon)\omega)^{[n]}\right)d\epsilon\\
=&\int^{1}_0 \left(\int_{X_\tau}(\Phi\vv(m_{\Omega_\epsilon})\Omega_{\epsilon}^{[n]})_{|X_\tau}\right)d\epsilon 
\end{align*}
where $\Omega_\epsilon:=\epsilon\Omega+(1-\epsilon)\hat\omega$, $m_{\Omega_\epsilon}:=\epsilon m_{\Omega}+(1-\epsilon)m_{\hat\omega}$, and $\hat{\omega}$, $\Phi$ are given in \Cref{lem-hat-Omega}. It thus follows that $\mathcal{E}_{\vv}(\phi_\tau)$ extends to a smooth function on $\mathbb{P}^{1}\setminus\{0\}$. Let $f(\tau)$ be a smooth function with compact support in $\mathbb{C}^\star\subset\mathbb{P}^{1}$. Letting $\hat{f}:=\pi^{\star}f$ we have
\begin{align}
\begin{split}\label{eq-ddc-a}
\langle dd^{c}\mathcal{E}_{\vv}(\phi_\tau),f\rangle=&\int^{1}_0\left(\int_{\mathbb{C}^\star}dd^{c}f(\tau)\int_{X_\tau}(\Phi \vv(m_{\Omega_\epsilon})\Omega_{\epsilon}^{[n]})_{|X_\tau}\right)d\epsilon\\
=&\int^{1}_0\left(\int_{\mathcal{X}^{\star}}\Phi \vv(m_{\Omega_\epsilon})dd^{c}\hat{f}\wedge\Omega_{\epsilon}^{[n]}\right)d\epsilon\\
=& -\int^{1}_0\left(\int_{\mathcal{X}^{\star}}\Phi d\hat{f}\wedge d^{c}\vv(m_{\Omega_\epsilon})\wedge\Omega_{\epsilon}^{[n]}\right)d\epsilon\\
&-\int^{1}_0\left(\int_{\mathcal{X}^{\star}}\vv(m_{\Omega_\epsilon}) d\hat{f}\wedge d^{c}\Phi\wedge\Omega_{\epsilon}^{[n]}\right)d\epsilon\\
=&-\int^{1}_0\left(\int_{\mathcal{X}^{\star}}\Phi d\hat{f}\wedge d^{c}\vv(m_{\Omega_\epsilon})\wedge\Omega_{\epsilon}^{[n]}\right)d\epsilon\\
&+\int^{1}_0\left(\int_{\mathcal{X}^{\star}}\hat{f} \vv(m_{\Omega_\epsilon}) dd^{c}\Phi\wedge\Omega_{\epsilon}^{[n]}\right)d\epsilon\\
&+\int^{1}_0\left(\int_{\mathcal{X}^{\star}}\hat{f}d\vv(m_{\Omega_\epsilon}) \wedge d^{c}\Phi\wedge\Omega_{\epsilon}^{[n]}\right)d\epsilon,
\end{split}
\end{align}
The first integral in the last equality vanishes. Indeed, for a basis $(\xi_i)_{i=1,\cdots,\ell}$ of $\tor$ we have
\begin{equation*}
d\hat{f}\wedge d^{c}\vv(m_{\hat\Omega_\epsilon})\wedge\Omega_{\epsilon}^{[n]}=\sum_{i=1}^{\ell}\vv_{,i}(m_{\Omega_\epsilon})(df)(\pi_*\xi_i)\Omega_{\epsilon}^{[n+1]}=0,
\end{equation*}
since the action of $\T$ preserves the fibers of $\mathcal{X}\to \mathbb{P}^{1}$. 
For the remaining integrals in the last equality in \eqref{eq-ddc-a}, integration by parts in the variable $\epsilon$ gives
\begin{align}
\begin{split}\label{eq-ddc-b}
&\int^{1}_0\left(\int_{\mathcal{X}^{\star}}\hat{f} \vv(m_{\Omega_\epsilon}) dd^{c}\Phi \wedge\Omega_{\epsilon}^{[n]}\right)d\epsilon\\
=&\int^{1}_0\left(\int_{\mathcal{X}^{\star}}\hat{f} \vv(m_{\Omega_\epsilon}) \frac{d}{d\epsilon}\Omega_{\epsilon}^{[n+1]}\right)d\epsilon\quad \text{(since }\Omega_\epsilon:=\hat\omega+\epsilon dd^{c}\Phi\text{)}\\
=&\int_{\mathcal{X}^{\star}}\hat{f} \vv(m_{\Omega})\Omega^{[n+1]}-\int^{1}_0\left(\int_{\mathcal{X}^{\star}}\hat{f} \left(\frac{d}{d\epsilon}\vv(m_{\Omega_\epsilon})\right) \Omega_{\epsilon}^{[n+1]}\right)d\epsilon\\
=&\int_{\mathcal{X}^{\star}}\hat{f} \vv(m_{\Omega})\Omega^{[n+1]}-\int^{1}_0\left(\int_{\mathcal{X}^{\star}}\hat{f} d\vv(m_{\Omega_\epsilon})\wedge d^{c}\Phi\wedge \Omega_{\epsilon}^{[n]}\right)d\epsilon,
\end{split}
\end{align}
where for passing from the third line to the last line we used the following 
\begin{align*}
\left(\frac{d}{d\epsilon}\vv(m_{\Omega_\epsilon})\right) \Omega_{\epsilon}^{[n+1]}=&\sum_{i=1}^{\ell}\vv_{,i}(m_{\Omega_\epsilon})d^{c}\Phi(\xi_i)\Omega_{\epsilon}^{[n+1]}\\
=&\sum_{i=1}^{\ell}\vv_{,i}(m_{\Omega_\epsilon})dm^{\xi_i}_{\Omega_{\epsilon}}\wedge d^{c}\Phi\wedge\Omega_{\epsilon}^{[n]}\\
=&\vv(m_{\Omega_\epsilon})\wedge d^{c}\Phi\wedge \Omega_{\epsilon}^{[n]}.
\end{align*}
 By substituting \eqref{eq-ddc-b} in \eqref{eq-ddc-a} we get \eqref{pi-Omega-n}.

Now we establish \eqref{eq-lim-E} using \eqref{pi-Omega-n}, following the proof \cite[Theorem 4.9]{Dyr1}. Let $\mathbb{D}_\epsilon\subset\mathbb{C}$ be the disc of center $0$ and radius $\epsilon>0$. Using the change of coordinates $(t,s)$ given by $\tau=e^{-t+is}\in\mathbb{C}$ and the $\mathbb{S}^{1}$-invariance of $\mathcal{E}_{\vv}(\phi_\tau)$ we calculate
\begin{align*}
\int_{\mathcal{X}}\vv(m_{\Omega})\Omega^{[n+1]}=&\underset{\epsilon\to 0}{\lim}\int_{\mathcal{X}\setminus\pi^{-1}(\mathbb{D}_\epsilon)}\vv(m_{\Omega})\Omega^{[n+1]}\\
=&\underset{\epsilon\to 0}{\lim}\int_{\mathbb{P}^1\setminus\mathbb{D}_\epsilon}\pi_\star\big( \vv(m_{\Omega})\Omega^{[n+1]}\big)\\
=&\underset{\epsilon\to 0}{\lim}\int_{\mathbb{P}^1\setminus\mathbb{D}_\epsilon}dd^{c}\mathcal{E}_{\vv}(\phi_\tau)\quad\text{by \eqref{pi-Omega-n}}\\
=&\underset{\epsilon\to 0}{\lim}\left(\left.\frac{d}{dt}\right|_{t=-\log\epsilon}\mathcal{E}_{\vv}(\phi_t)\right)\quad\text{by the Green-Riesz formula} \\
=&\underset{t\rightarrow+\infty}{\lim}\frac{d}{dt}\mathcal{E}_{\vv}(\phi_t)=\underset{t\rightarrow+\infty}{\lim}\frac{\mathcal{E}_{\vv}(\phi_t)}{t}.
\end{align*}
\end{proof}

Let $\Omega\in\mathcal{A}$ be $\mathbb{G}$-invariant K\"ahler form. We consider the K\"ahler metric on $\mathcal{X}^{\star}$ given by
$\hat{\omega}+\pi^*\FS=\lambda^*(\pi^{*}_X\omega+\pi_{\mathbb{P}^{1}}^{*}\FS)$ (by the equivariance of $\lambda$), where $\hat{\omega}:=(\pi_X\circ\lambda)^{*}\omega$ with $\lambda$ the  map given by \eqref{lambda} and $\pi_X,\pi_{\mathbb{P}^{1}}$ denote the projections on the factors of $X\times(\mathbb{P}^{1}\setminus\{0\})$. Then we have on $\mathcal{X}^{\star}$
\begin{equation}\label{eq-ric-Om}
\Ric(\Omega)-\pi^{*}\FS-\widehat{\Ric(\omega)}=\frac{1}{2}dd^{c}\Psi,
\end{equation}
where $\Psi=\log\left(\frac{\Omega^{n+1}}{\hat{\omega}^{n}\wedge\pi^{*}\FS}\right)$ and $\widehat{\Ric(\omega)}:=(\pi_X\circ\lambda)^{*}\Ric(\omega)$. Using \eqref{eq-ric-Om} and \Cref{lm-m} \ref{lm-m-ii}, we obtain on $\mathcal{X}^\star$
\begin{equation}\label{m-Ric-psi}
m_{\widehat{\Ric(\omega)}}^{\xi}=m_{\Ric(\Omega)}^{\xi}+\frac{1}{2}(d^{c}\Psi)(\xi),
\end{equation}
for any $\xi\in \tor$, where $m_{\widehat\Ric(\Omega)}:=(\pi_X\circ\lambda)^*m_{\Ric(\omega)}$.

\begin{lem}\label{lem-weak-ddcE0}
Under the hypotheses of \Cref{Mabuchi-Slope}, we have 
\begin{equation}\label{weak-ddcE0}
dd^{c}\mathcal{E}^{\Ric(\omega)}_{\uu}(\phi_\tau)=\pi_*\left(\uu(m_{\Omega})\widehat{\Ric(\omega)}\wedge \Omega^{[n]}+\langle (d\uu)(m_{\Omega}),m_{\widehat{\Ric(\omega)}}\rangle\Omega^{[n+1]}\right).
\end{equation}
\end{lem}
\begin{proof}
From the very definition of $\mathcal{E}^{\Ric(\omega)}_{\uu}$ (see \eqref{E-theta}) we have
\begin{equation*}
\mathcal{E}^{\Ric(\omega)}_{\uu}(\phi_\tau)=\int^{1}_0\Big(\int_{X_\tau}\big[\Phi \big(\uu(m_{\Omega_\epsilon})\widehat{\Ric(\omega)}\wedge \Omega^{[n-1]}_\epsilon+\langle (d\uu)(m_{\Omega_\epsilon}),m_{\widehat{\Ric(\omega)}}\rangle\Omega^{[n]}_\epsilon\big)\big]_{|X_\tau}\Big)d\epsilon,
\end{equation*}
where $\Omega_\epsilon:=\epsilon\Omega+(1-\epsilon)\hat\omega$, $m_{\Omega_\epsilon}:=\epsilon m_{\Omega}+(1-\epsilon)m_{\hat\omega}$, and $\hat{\omega}$, $\Phi$ are given in \Cref{lem-hat-Omega}. As in the proof of \Cref{lim-E}, we see that $\mathcal{E}_{\uu}^{\Ric(\omega)}(\phi_\tau)$ extends to a smooth function on $\mathbb{P}^{1}\setminus\{0\}$. Furthermore, for any smooth function $f(\tau)$ with compact support in $\mathbb{C}^\star\subset\mathbb{P}^{1}$, we have
\begin{align*}
&\langle dd^{c}\mathcal{E}^{\Ric(\omega)}_{\uu}(\phi_\tau),f\rangle=\int_{\mathbb{C}^\star}\mathcal{E}^{\Ric(\omega)}_{\uu}(\phi_\tau) dd^{c}f=\\
=&\int^{1}_0\Big(\int_{\mathbb{C}^\star} dd^{c}f\int_{X_\tau}\big[\Phi \big(\uu(m_{\Omega_\epsilon})\widehat{\Ric(\omega)}\wedge \Omega^{[n-1]}_\epsilon+\langle (d\uu)(m_{\Omega_\epsilon}),m_{\widehat{\Ric(\omega)}}\rangle\Omega^{[n]}_\epsilon\big)\big]_{|X_\tau}\Big)d\epsilon\\
=&\int^{1}_0\Big(\int_{\mathcal{X}^{\star}} \hat{f}dd^{c}\big[\Phi \big(\uu(m_{\Omega_\epsilon})\widehat{\Ric(\omega)}\wedge \Omega^{[n-1]}_\epsilon+\langle (d\uu)(m_{\Omega_\epsilon}),m_{\widehat{\Ric(\omega)}}\rangle\Omega^{[n]}_\epsilon\big)\big]\Big)d\epsilon\\
=&-\int^{1}_0\Big(\int_{\mathcal{X}^{\star}} \Phi \big[d\hat{f}\wedge
 d^c(\uu(m_{\Omega_\epsilon}))\wedge \widehat{\Ric(\omega)}\wedge \Omega^{[n-1]}_\epsilon+d\hat{f}\wedge
 d^c(\langle (d\uu)(m_{\Omega_\epsilon}),m_{\widehat{\Ric(\omega)}}\rangle)\wedge \Omega^{[n]}_\epsilon\big]\Big)d\epsilon\\
&+\int^{1}_0\Big(\int_{\mathcal{X}^{\star}} \hat{f}\big[d(\langle (d\uu)(m_{\Omega_\epsilon}),m_{\widehat{\Ric(\omega)}}\rangle)\wedge
 d^c\Phi\wedge\Omega^{[n]}_\epsilon+d(\uu(m_{\Omega_\epsilon}))\wedge
 d^c\Phi\wedge \widehat{\Ric(\omega)}\wedge \Omega^{[n-1]}_\epsilon\big]\Big)d\epsilon\\
&+\int^{1}_0\Big(\int_{\mathcal{X}^{\star}} \hat{f}\big[\uu(m_{\Omega_\epsilon})\widehat{\Ric(\omega)}\wedge(dd^{c}\Phi)\wedge \Omega^{[n-1]}_\epsilon+\langle (d\uu)(m_{\Omega_\epsilon}),m_{\widehat{\Ric(\omega)}}\rangle(dd^{c}\Phi)\wedge\Omega^{[n]}_\epsilon\big)\big]\Big)d\epsilon\\
=&I_1+I_2+I_3,
\end{align*}
where $I_1,I_2,I_3$ respectively denote the integrals on the first, second and third lines of the last equality. Now we compute each integral individually. We have 
\begin{align*}
&d\hat{f}\wedge
 d^c(\langle (d\uu)(m_{\Omega_\epsilon}),m_{\widehat{\Ric(\omega)}}\rangle)\wedge \Omega^{[n]}_\epsilon+d\hat{f}\wedge
 d^c(\uu(m_{\Omega_\epsilon}))\wedge \widehat{\Ric(\omega)}\wedge \Omega^{[n-1]}_\epsilon\\
=&\sum_{i,j}\uu_{,ij}(m_{\Omega_\epsilon})(d\hat{f})(\xi_j)m^{\xi_i}_{\widehat{\Ric(\omega)}} \Omega^{[n+1]}_\epsilon +\sum_{i}\uu_{,i}(m_{\Omega_\epsilon})d\hat{f}\wedge
 d^c m^{\xi_i}_{\widehat{\Ric(\omega)}}\wedge \Omega^{[n]}_\epsilon\\
&+\sum_i \uu_{,i}(m_{\Omega_\epsilon})(d\hat f)(\xi_i)(\Lambda_{\Omega_\epsilon}\widehat{\Ric(\omega)})\Omega^{[n+1]}_\epsilon-(d\hat{f}\wedge
 d^c(\uu(m_{\Omega_\epsilon})),\widehat{\Ric(\omega)})_{\Omega_\epsilon}\Omega^{[n+1]}_\epsilon\\
=&\sum_{i}\uu_{,i}(m_{\Omega_\epsilon})(d\hat{f}\wedge
 d^c m_{\Omega_\epsilon}^{\xi_i},\widehat{\Ric(\omega)})\Omega^{[n+1]}_\epsilon-(d\hat{f}\wedge
 d^c(\uu(m_{\Omega_\epsilon})),\widehat{\Ric(\omega)})\Omega^{[n+1]}_\epsilon=0,
\end{align*}
where $\xxi=(\xi_i)_{i=1,\cdots,\ell}$ is a basis of $\tor$. It follows that $I_1=0$. For the integral $I_2$, a similar calculation gives
\begin{equation*}
I_2=\int^{1}_0\Big(\int_{\mathcal{X}^{\star}} \hat{f}\big[\sum_i\uu_{,i}(m_{\Omega_\epsilon})(d^c\Phi)(\xi_i) \widehat{\Ric(\omega)}\wedge \Omega^{[n]}_\epsilon+\sum_{i,j}\uu_{,ij}(m_{\Omega_\epsilon})m^{\xi_i}_{\widehat{\Ric(\omega)}}
 (d^c\Phi)(\xi_j)\Omega^{[n]}_\epsilon\big]\Big)d\epsilon,
\end{equation*}
Now we consider the integral $I_3$. Using the fact that $\Omega_\epsilon=\hat{\omega}+\epsilon dd^{c}\Phi$, an integration by parts with respect to $\epsilon$ gives
\begin{align*}
I_3=&\int^{1}_0\Big(\int_{\mathcal{X}^{\star}} \hat{f}\big[\uu(m_{\Omega_\epsilon})\widehat{\Ric(\omega)}\wedge\big(\frac{d}{d\epsilon}\Omega^{[n]}_\epsilon\big)+\langle (d\uu)(m_{\Omega_\epsilon}),m_{\widehat{\Ric(\omega)}}\rangle\big(\frac{d}{d\epsilon}\Omega^{[n+1]}_\epsilon\big)\big]\Big)d\epsilon\\
=&\int_{\mathcal{X}^{\star}} \hat{f}\big[\uu(m_{\Omega})\widehat{\Ric(\omega)}\wedge\Omega^{[n]}+\langle (d\uu)(m_{\Omega}),m_{\widehat{\Ric(\omega)}}\rangle\Omega^{[n+1]}\big]\\
&-\int^{1}_0\Big(\int_{\mathcal{X}^{\star}}\hat{f}\big[\big(\frac{d}{d\epsilon}\uu(m_{\Omega_\epsilon})\big)\widehat{\Ric(\omega)}\wedge\Omega^{[n]}_\epsilon+\big(\frac{d}{d\epsilon}\langle (d\uu)(m_{\Omega_\epsilon}),m_{\widehat{\Ric(\omega)}}\rangle\big)\Omega^{[n+1]}_\epsilon\big]\Big)d\epsilon,
\end{align*}
By \Cref{lem-hat-Omega} \ref{lem-hat-Omega-iii} the second integral on the second equality is given by
\begin{align*}
&\int^{1}_0\Big(\int_{\mathcal{X}^{\star}}\hat{f}\big[\big(\frac{d}{d\epsilon}\uu(m_{\Omega_\epsilon})\big)\widehat{\Ric(\omega)}\wedge\Omega^{[n]}_\epsilon+\big(\frac{d}{d\epsilon}\langle (d\uu)(m_{\Omega_\epsilon}),m_{\widehat{\Ric(\omega)}}\rangle\big)\Omega^{[n+1]}_\epsilon\big]\Big)d\epsilon\\
=&\int^{1}_0\Big(\int_{\mathcal{X}^{\star}} \hat{f}\big[\sum_i\uu_{,i}(m_{\Omega_\epsilon})
 (d^c\Phi)(\xi_i) \widehat{\Ric(\omega)}\wedge \Omega^{[n]}_\epsilon+\sum_{i,j}\uu_{,ij}(m_{\Omega_\epsilon})m^{\xi_i}_{\widehat{\Ric(\omega)}}
 (d^c\Phi)(\xi_j)\Omega^{[n]}_\epsilon\big]\Big)d\epsilon\\
 =&I_2.
\end{align*}
It follows that 
\begin{equation*}
I_1+I_2+I_3=\int_{\mathcal{X}^{\star}} \hat{f}\big[\uu(m_{\Omega})\widehat{\Ric(\omega)}\wedge\Omega^{[n]}+\langle (d\uu)(m_{\Omega}),m_{\widehat{\Ric(\omega)}}\rangle\Omega^{[n+1]}\big].
\end{equation*}
This completes the proof. 
 \end{proof}

\begin{lem}\label{lim-(H+E)}
Under the hypotheses of \Cref{Mabuchi-Slope},
\begin{align}
\begin{split}\label{eq-lim-(H+E)}
&\underset{t\rightarrow +\infty}{\lim}\dfrac{1}{t}\left(\int_{X_1} \psi_t \uu(m_{\phi_t})\omega_{\phi_t}^{[n]}-2\mathcal{E}^{\Ric(\omega)}_{\uu}(\phi_t)\right)\\=&-2\int_{\mathcal{X}} \uu(m_{\Omega})(\Ric(\Omega)-\pi^{\star}\FS)\wedge \Omega^{[n]}+\langle (d\uu)(m_{\Omega}),m_{\Ric(\Omega)}\rangle\Omega^{[n+1]}
\end{split}
\end{align}
where $\phi_t$ is given by \eqref{rho=ddc} and $\psi_t=\psi_\tau$ with $\tau=e^{-t+is}$ is given by
\begin{equation}\label{h-psi}
\psi_{\tau}:=\rho(\tau)^{*}\big(\Psi_{|X_\tau}\big)\in C^{\infty}(X_1,\mathbb{\R})^{\T}.
\end{equation}
\end{lem}

\begin{proof}
We define on $\mathbb{C}^\star$ the function $\mathcal{H}(\tau):=\int_X \psi_\tau \uu(m_{\tau})\omega_{\tau}^{[n]}$. Let $f(\tau)$ be a test function with support in $\mathbb{C}^\star\subset\mathbb{P}^{1}$ and $\hat{f}:=\pi^{\star}f$. We have
\begin{align*}
\langle dd^{c}\mathcal{H},f\rangle=&\int_{\mathbb{C}^\star}dd^{c}f\int_{X_\tau} (\Psi\uu(m_{\Omega})\Omega^{[n]})_{|X_\tau}\\
=&\int_{\mathcal{X}^{\star}} \Psi\uu(m_{\Omega})dd^{c}\hat{f}\wedge\Omega^{[n]}\\
=&\int_{\mathcal{X}^{\star}}\Psi d(\uu(m_{\Omega}))\wedge d^{c}\hat{f}\wedge\Omega^{[n]}-\int_{\mathcal{X}^{\star}} \uu(m_{\Omega})d\Psi\wedge d^{c}\hat{f}\wedge\Omega^{[n]}
\end{align*}
Notice that $d(\uu(m_{\Omega}))\wedge d^{c}\hat{f}\wedge\Omega^{n}=0$ since the $1$-form $d^{c}\hat{f}$ is zero on the fundamental vector fields of the $\T$-action. Integration by parts gives
\begin{align*}
\langle dd^{c}\mathcal{H},f\rangle=&\int_{\mathcal{X}^{\star}}\hat{f}d\Psi\wedge d^c \uu(m_{\Omega})\wedge \Omega^{[n]}+\int_{\mathcal{X}^{\star}}\hat{f}\uu(m_{\Omega})dd^{c}\Psi\wedge\Omega^{[n]}.
\end{align*}
Using the equations \eqref{eq-ric-Om} and \eqref{m-Ric-psi} we obtain
\begin{align}
\begin{split}\label{weak-ddc-Ent}
\langle dd^{c}\mathcal{H},f\rangle =&-2\int_{\mathcal{X}^{\star}}\hat{f} \langle (d\uu)(m_{\Omega}),m_{\Ric(\Omega)}-m_{\widehat{\Ric(\omega)}}\rangle\Omega^{[n+1]}\\&-\int_{\mathcal{X}^{\star}}\hat{f}\uu(m_{\Omega})(\Ric(\Omega)-2\pi^*\FS-\widehat{\Ric(\omega)})\wedge\Omega^{[n]}.
\end{split}
\end{align}
Combining \eqref{weak-ddcE0} and \eqref{weak-ddc-Ent} gives
\begin{align*}
dd^{c}\big(\mathcal{H}(\tau)&-2\mathcal{E}^{\Ric(\omega)}_{\uu}(\phi_\tau)\big)\\
=&-2\pi_\star\left(\uu(m_{\Omega})(\Ric(\Omega)-\pi^{\star}\FS)\wedge \Omega^{[n]}+\langle (d\uu)(m_{\Omega}),m_{\Ric(\Omega)}\rangle\Omega^{[n+1]}\right).
\end{align*}
We conclude in the same way as in the proof of \Cref{lim-E}.
\end{proof}

We consider the following function on $\mathbb{C}^\star$:
\begin{equation}\label{Mabuchi-hat}
\mathcal{M}^{\Psi}_{\uu,\vv}(\phi_\tau):=\int_X \psi_\tau \uu(m_{\phi_\tau})\omega_{\phi_\tau}^{[n]}-2\mathcal{E}^{\Ric(\omega)}_{\uu}(\phi_\tau)+c_{(\uu,\vv)}(\alpha)\mathcal{E}_{\uu}(\phi_\tau),
\end{equation}
where $\phi_\tau$ and $\psi_\tau$ are given by \eqref{rho=ddc} and \eqref{h-psi} respectively. From the definition of $\mathcal{M}^{\Psi}_{\uu,\vv}(\phi_\tau)$ and Lemmas \ref{lim-E} and \ref{lim-(H+E)} we see that
\begin{equation}\label{slop-Mabuchi-hat}
\underset{t\rightarrow +\infty}{\lim}\frac{\mathcal{M}^{\Psi}_{\uu,\vv}(\phi_t)}{t}=\mathcal{F}_{\uu,\vv}(\mathcal{X},\mathcal{A}).
\end{equation}

\begin{lem}\label{lem-int-bounded}
If the central fiber $X_0$ is reduced, then the integral
\begin{equation*}
\Upsilon(\tau):=\int_{X_\tau} \log\left(\frac{\Omega^{n}\wedge\pi^*\FS}{\Omega^{n+1}}\right)\uu(m_{\Omega})\Omega_{\tau}^{[n]},
\end{equation*}
is bounded on $\mathbb{C}^{\star}$.
\end{lem}

\begin{proof}
The integral $\Upsilon(\tau)$ is bounded from above since $Z(\hat x)=\frac{\Omega^{n}\wedge\pi^*\FS}{\Omega^{n+1}}$ is a non-negative smooth function on $\mathcal{X}$ and the integral $\int_{X_\tau} \uu(m_{\Omega})\Omega_{\tau}^{[n]}$ is independent from $\tau$ (see \Cref{lm-mom}). Notice that $\Upsilon(\tau)$ is bounded if and only if $\int_{X_\tau} |\log(Z)|\uu(m_{\Omega})\Omega_{\tau}^{[n]}$ is bounded. Indeed, if $\Upsilon(\tau)=O(1)$ then 
\begin{equation*}
\int_{X_\tau}|\log(Z)|\uu(m_{\Omega})\Omega_{\tau}^{[n]}=\int_{X_\tau} (\log(Z)+|\log(Z)|)\uu(m_{\Omega})\Omega_{\tau}^{[n]}-\Upsilon(\tau)=\mathcal{O}(1).
\end{equation*}
It follows that $\int_{X_\tau} |\log(Z)|\uu(m_{\Omega})\Omega_{\tau}^{[n]}=\mathcal{O}(1)$. The converse follows from 
\begin{equation*}
|\Upsilon(\tau)|\leq \int_{X_\tau} |\log(Z)|\uu(m_{\Omega})\Omega_{\tau}^{[n]}.
\end{equation*}
Using that $\uu(m_\Omega)$ is a smooth function on $\mathcal{X}$ we see that $\int_{X_\tau} |\log(Z)|\uu(m_{\Omega})\Omega_{\tau}^{[n]}=\mathcal{O}(1)$ if and only if $\int_{X_\tau} |\log(Z)|\Omega_{\tau}^{[n]}=\mathcal{O}(1)$, which is also equivalent to $\int_{X_\tau} \log(Z)\Omega_{\tau}^{[n]}=\mathcal{O}(1)$. By \cite[Remark 4.12]{Dervan-Ross}, if the central fiber $X_0$ is reduced then $\int_{X_\tau} \log(Z)\Omega_{\tau}^{[n]}=\mathcal{O}(1)$ which implies that $\Upsilon(\tau)=\mathcal{O}(1)$.
\end{proof}

Now we are in position to give a proof for \Cref{Mabuchi-Slope}.

\begin{proof}[\bf Proof of \Cref{Mabuchi-Slope}]
From the modified Chen-Tian formula in \Cref{Chen-Tian}, \eqref{Mabuchi-hat} and by \Cref{lem-int-bounded} we get
\begin{align*}
\mathcal{M}_{\uu,\vv}(\phi_\tau)-\mathcal{M}^{\Psi}_{\uu,\vv}(\phi_\tau)=&\int_X \left(\log\left(\frac{\omega^{n}_\tau}{\omega^{n}}\right)- \psi_\tau \right)\uu(m_\tau)\omega^{[n]}_\tau\\
=&\int_{X_\tau} \left(\log\left(\frac{\Omega^{n}\wedge\pi^*\FS}{\hat\omega^{n}\wedge\pi^*\FS}\right)- \Psi \right)\rho(\tau^{-1})^{*}(\uu(m_{\tau})\omega^{[n]}_\tau)\\
=&\int_{X_\tau} \log\left(\frac{\Omega^{n}\wedge\pi^*\FS}{\Omega^{n+1}}\right)\uu(m_{\Omega})\Omega_{\tau}^{[n]}=\mathcal{O}(1).
\end{align*}
Dividing by $t$ (where we recall $\tau=e^{-t+is}$) and passing to the limit when $t$ goes to infinity concludes the proof. 
\end{proof}

\begin{proof}[\bf Proof of Theorems \ref{thm:main} and \ref{thm:BB}]
These are direct corollaries of Theorems \ref{thm:mabuchi-bounded} and \ref{Mab-bound-1-w} respectively, together with \Cref{Mabuchi-Slope} and \Cref{fut-obs}.
\end{proof}

\begin{prop}\label{submersion-case}
If $(\mathcal{X},\mathcal{A},\T)$ is a K\"ahler test configuration of $(X,\alpha,\T)$ such that $\pi:\mathcal{X}\rightarrow \mathbb{P}^{1}$ is a smooth submersion then
\begin{equation*}
\begin{split}
\mathcal{F}_{\uu,\vv}(\mathcal{X},\mathcal{A})=\mathcal{F}_{\uu,\vv}^\alpha(V_\rho)-\frac{{\rm Vol}(\mathcal{X},\mathcal{A})}{{\rm Vol}(X,\alpha)}\int_{X}\big(\Scal_{\uu}(\omega)-c_{\uu,\vv}(\alpha)\vv(m_{\omega})\big)\omega^{[n]},
\end{split}
\end{equation*}
where $V_\rho$ is the generator of the $\mathbb{S}^{1}_\rho$-action on $X_0$, and $\mathcal{F}_{\uu,\vv}^{\alpha}(V_\rho)$ is the $(\uu,\vv)$-Futaki invariant of the smooth central fibre $(X_0,\alpha)$  introduced in \Cref{Def-Fut}. In particular if $(X,\alpha,\T)$ is $(\uu,\vv)$-semistable on smooth test configurations, then 
$$\int_X\Scal_{\uu}(\omega)\omega^{[n]}=c_{(\uu,\vv)}(\alpha)\int_X\vv(m_{\omega})\omega^{[n]}\text{ and }\mathcal{F}^{\alpha}_{\uu,\vv}\equiv 0.$$
\end{prop}

\begin{proof}
We just adapt the arguments from \cite{ding-tian} to our weighted setting. From \Cref{Defn-Mabuchi} we have 
\begin{align}
\begin{split}\label{M1}
\frac{d}{dt}\mathcal{M}_{\uu,\vv}(\phi_\tau)=&-\int_{X_1}\dot{\phi}_\tau\big(\Scal_{\uu}(\omega_\tau)-c_{\uu,\vv}(\alpha)\vv(m_\tau)\big)\omega^{[n]}_{\tau},\\
=&-\int_{X_\tau}\rho(\tau^{-1})^*\dot{\phi}_\tau\big(\Scal_{\uu}(\Omega_{\tau})-c_{\uu,\vv}(\alpha)\vv(m_{\Omega_{\tau}})\big)\Omega^{[n]},
\end{split}
\end{align} 
where $t=-\log|\tau|$, $\dot{\phi}_\tau=\frac{d\phi_\tau}{dt}$ and $\omega_\tau$, $\phi_\tau$, $m_\tau$ are given by \eqref{rho=ddc} and \eqref{m-tau}. Note that the flow of the vector field $\mathcal{J}V_\rho$ is $\varphi^{t}_{\mathcal{J}V_\rho}=\rho(e^{-t})$ where $\mathcal{J}$ denotes the complex structure of $\mathcal{X}$. Let  $h^{\rho}$ be a Hamiltonian function of $V_\rho$ with respect to $\Omega$. We have $\frac{d}{dt}\rho(\tau)^{*}\Omega=-dd^{c}(\rho(\tau)^*h^{\rho})$. On the other hand, using \eqref{rho=ddc} we get
$\frac{d}{dt}(\rho(\tau)^{*}\Omega)_{|X_1}=dd^{c}\dot{\phi}_\tau$. By taking the restriction on $X_1$ of the first equality and comparing to the secon, we get
\begin{equation}\label{M2}
h^{\rho}_{|X_\tau}=-\rho(\tau^{-1})^*\dot{\phi}_\tau+a(\tau),
\end{equation}
where $a(\tau)\in\R$ is a constant depending on $\tau\in\mathbb{C}^*$. By \eqref{M2} and \Cref{E-lem-def}, we have
\begin{equation*}
a(\tau)=\frac{1}{{\rm Vol}(X,\alpha)}\Big(\int_{X_\tau}h^\rho\Omega^{[n]}+\frac{d\mathcal{E}_{1}(\phi_\tau)}{dt}\Big).
\end{equation*}
Using that $\pi:\mathcal{X}\to\mathbb{P}^{1}$ is a smooth submersion and \Cref{lim-E}, we get
\begin{equation}\label{eq-lim-a}
\underset{t\rightarrow\infty}{\lim}a(\tau)=\frac{1}{{\rm Vol}(X,\alpha)}\Big(\int_{X_0}h^\rho\Omega^{[n]}+{\rm Vol}(\mathcal{X},\mathcal{A})\Big).
\end{equation}
Substituting \eqref{M2} in \eqref{M1}, we obtain
\begin{align}
\begin{split}\label{eq-M-tau}
\frac{d}{dt}\mathcal{M}_{\uu,\vv}(\phi_\tau)=&\int_{X_\tau}\big(\Scal_{\uu}(\Omega_{\tau})-c_{\uu,\vv}(\alpha)\vv(m_{\Omega_{\tau}})\big)h^\rho\Omega^{[n]}\\
&-a(\tau)\int_{X_\tau}\big(\Scal_{\uu}(\Omega_{\tau})-c_{\uu,\vv}([\Omega_\tau])\vv(m_{\Omega_{\tau}})\big)\Omega^{[n]}.
\end{split}
\end{align}
Passing to the limit when $t\to\infty$ in \eqref{eq-M-tau} and using \Cref{Mabuchi-Slope}, we obtain
\begin{align*}
&\mathcal{F}_{\uu,\vv}(\mathcal{X},\mathcal{A})=\underset{t\rightarrow\infty}{\lim}\frac{d}{dt}\mathcal{M}_{\uu,\vv}(\phi_\tau)\\
&=\int_{X_0}\big(\Scal_{\uu}(\Omega_0)-c_{\uu,\vv}(\alpha)\vv(m_{\Omega_0})\big)h^{\rho}\Omega^{[n]}\\
&-\frac{1}{{\rm Vol}(X,\alpha)}\Big(\int_{X_0}h^\rho\Omega^{[n]}+{\rm Vol}(\mathcal{X},\mathcal{A})\Big)\int_{X_0}\big(\Scal_{\uu}(\Omega_{0})-c_{\uu,\vv}([\Omega_0])\vv(m_{\Omega_{0}})\big)\Omega^{[n]}\\
&=\int_{X_0}\big(\Scal_{\uu}(\Omega_0)-c_{\uu,\vv}(\alpha)\vv(m_{\Omega_0})\big)\Big(h^{\rho}-\frac{1}{{\rm Vol}(X,\alpha)}\int_{X_0}h^\rho\Omega^{[n]}\Big)\Omega^{[n]}\\
&-\frac{{\rm Vol}(\mathcal{X},\mathcal{A})}{{\rm Vol}(X,\alpha)}\int_{X_0}\big(\Scal_{\uu}(\Omega_{0})-c_{\uu,\vv}([\Omega_0])\vv(m_{\Omega_{0}})\big)\Omega^{[n]}\\
&=\mathcal{F}_{\uu,\vv}^\alpha(V_\rho)-\frac{{\rm Vol}(\mathcal{X},\mathcal{A})}{{\rm Vol}(X,\alpha)}\int_{X}\big(\Scal_{\uu}(\omega)-c_{\uu,\vv}(\alpha)\vv(m_{\omega})\big)\omega^{[n]}.
\end{align*}
where $\Omega_0=\Omega_{|X_0}\in\mathcal{A}_{|X_0}$, and we have used in the last equality that for any $\tau\in\mathbb{C}^{\star}$ we have
\begin{align*}
&\int_{X_\tau}\Scal_{\uu}(\Omega_\tau)\Omega^{[n]}=\int_{X_1}\Scal_{\uu}(\omega_\tau)\omega^{[n]}_\tau=\int_{X}\Scal_{\uu}(\omega)\omega^{[n]},\\
&\int_{X_\tau}\vv(m_{\Omega_\tau})\Omega^{[n]}=\int_{X_1}\vv(m_{\omega_\tau})\omega^{[n]}_\tau=\int_{X}\vv(m_{\omega})\omega^{[n]},
\end{align*}
see \Cref{top-c}. 

For the second statement, as $\int_{X}\big(\Scal_{\uu}(\omega)-c_{\uu,\vv}(\alpha)\vv(m_{\omega})\big)\omega^{[n]}=0$ by the definition of semi-stability, we consider the product test configurations associated to $V$ and $-V$ for any $V\in\mathfrak{h}_{\rm red}$, we obtain $\mathcal{F}_{\uu,\vv}^\alpha(V)=-\mathcal{F}_{\uu,\vv}^\alpha(V)\geq 0$ i.e. $\mathcal{F}_{\uu,\vv}^\alpha\equiv0.$
\end{proof}

\begin{rem}\label{Rem-Dervan}
In \cite{Dervan}, Dervan defines a $\T$-relative Donaldson--Futaki invariant $\DF_{\T}(\mathcal{X},\mathcal{A})$ for a smooth $\T$-compatible K\"ahler test configuration $\mathcal{X}$ as follows
\begin{equation*}
\DF_{\T}(\mathcal{X},\mathcal{A}):=\mathcal{F}_{1,1}(\mathcal{X},\mathcal{A})-\sum_{i=1}^{\ell}\frac{\langle h_{\rho},h_{i}\rangle_{X_0}}{\langle h_{i},h_{i}\rangle_{X_0}}\mathcal{F}_{1,1}^{\alpha}(\xi_i),
\end{equation*}
where $\xxi:=(\xi_i)_{i=1,\cdots,\ell}$ is a  basis of $\tor$ with corresponding Killing potentials $h_i=f_i(m_\Omega)=\langle m_\Omega,\xi_i\rangle+\lambda_i$, such that $\langle h_{i},h_{j}\rangle_{X_0}=\int_{X_0}h_ih_j\Omega^{n}=0$ for $i\neq j$ and $\int_{X_0}h_i\Omega^{n}=0$, where the integration on $X_0$ is defined by $\int_{X_0}:=\sum_i m_i\int_{(X^{(i)}_0)_{\rm reg}}$ with $[X_0]=\sum_i m_i X^{(i)}_0$ being the analytic cycle associated to $X_0$  and $(X^{(i)}_0)_{\rm reg}$ standing for the regular part of the irreducible component $X^{(i)}_0$ of $X_0$. Using \Cref{lm-mom}, we have
\begin{align}
\begin{split}\label{eq-h}
&\int_{X}\vv_{\rm ext}(m_\omega)\omega^{[n]}=\int_{X_1}\vv_{\rm ext}(m_\Omega)\Omega^{[n]}=\int_{X_\tau}\vv_{\rm ext}(m_\Omega)\Omega^{[n]},\\
&\mathcal{F}_{1,1}(\xi_i)=\langle \vv_{\rm ext}(m_\omega),h_{i}\rangle_{X}=\langle \vv_{\rm ext}(m_\Omega),f_{i}(m_\Omega)\rangle_{X_1}=\langle \vv_{\rm ext}(m_\Omega),h_i\rangle_{X_\tau},
\end{split}
\end{align}
for any $\tau\in\mathbb{C}^{\star}\subset\mathbb{P}^{1}$. As the family $\pi:\mathcal{X}\to\mathbb{P}^{1}$ is proper and flat, the current of integration along the fibers $X_\tau$ is continuous and converges to the integration over the analytic cycle of the central fiber $[X_0]$ (see \cite{Barlet}). Passing to the limit when $\tau\to 0$ in \eqref{eq-h}, we thus obtain $\int_{X}\vv_{\rm ext}(m_\omega)\omega^{[n]}=\int_{X_0}\vv_{\rm ext}(m_\Omega)\Omega^{[n]}$ and $\mathcal{F}_{1,1}(\xi_i)=\langle \vv_{\rm ext}(m_\Omega),h_{(\rho_i,\Omega)}\rangle_{X_0}$.
Thus,
\begin{equation}\label{F-Dervan}
\DF_{\T}(\mathcal{X},\mathcal{A})=\mathcal{F}_{1,1}(\mathcal{X},\mathcal{A})-\langle \vv_{\rm ext}(m_\Omega),h_{\rho}\rangle_{X_0}.
\end{equation}
On the other hand, the $(1,\vv_{\rm ext})$-Futaki invariant of $(\mathcal{X},\mathcal{A})$ is given by
\begin{equation}\label{F-ext}
\mathcal{F}_{1,\vv_{\rm ext}}(\mathcal{X},\mathcal{A})=-\int_{\mathcal{X}}\Scal(\Omega)\Omega^{[n+1]}+2\int_{\mathcal{X}}\pi^{\star}\FS\wedge \Omega^{[n]}+\int_{\mathcal{X}}\vv_{\rm ext}(m_\Omega)\Omega^{[n+1]}.
\end{equation}
(Recall that $c_{(1,\vv_{\rm ext})}(\alpha)=1$, see Section~\ref{s:(u,v)-extremal}). From \eqref{F-Dervan} and \eqref{F-ext}, we infer
\begin{align*}
\begin{split}
\mathcal{F}_{1,\vv_{\rm ext}}(\mathcal{X},\mathcal{A})-\DF_{\T}(\mathcal{X},\mathcal{A})=&\langle \vv_{\rm ext}(m_\Omega),h_{\rho}\rangle_{X_0}+\int_{\mathcal{X}}(\vv_{\rm ext}(m_\Omega)-c_{1,1}(\alpha))\Omega^{[n+1]}\\
=&\langle \vv_{\rm ext}(m_\Omega),h_{\rho}\rangle_{X_0}+\underset{t\to\infty}{\lim}\frac{d\mathcal{E}_{\accentset{\circ}{\vv}_{\rm ext}}(\phi_\tau)}{dt}\\
=&\langle \vv_{\rm ext}(m_\Omega),h_{\rho}\rangle_{X_0}+\underset{t\to\infty}{\lim}\Big(\int_{X_1}\dot{\phi}_\tau\accentset{\circ}{\vv}_{\rm ext}(m_\tau)\omega^{[n]}_\tau\Big)\\
=&\langle \vv_{\rm ext}(m_\Omega),h_{\rho}\rangle_{X_0}-\underset{t\to\infty}{\lim}\Big(\int_{X_\tau}h_{\rho}\accentset{\circ}{\vv}_{\rm ext}(m_\Omega)\Omega^{[n]}\Big)\\
=&\langle \vv_{\rm ext}(m_\Omega),h_{\rho}\rangle_{X_0}-\int_{X_0}h_{\rho}\accentset{\circ}{\vv}_{\rm ext}(m_\Omega)\Omega^{[n]}=0,
\end{split}
\end{align*}
where in the second equality we used \Cref{lim-E} for 
$$\accentset{\circ}{\vv}_{\rm ext}=\vv_{\rm ext}-c_{1,1}(\alpha)=\vv_{\rm ext}-\frac{1}{{\rm Vol}(X,\alpha)}\int_{X_\tau}\vv_{\rm ext}(m_\Omega)\Omega^{[n]},$$
for any $\tau\in\mathbb{C}^\star$ and in the fourth equality we used \eqref{M2}. It follows that 
$$\mathcal{F}_{1,\vv_{\rm ext}}(\mathcal{X},\mathcal{A})=\DF_{\T}(\mathcal{X},\mathcal{A}).$$
\end{rem}

\section{Algebraic definition of a $(\uu,\vv)$-Donaldson-Futaki invariant}\label{sec-8}
  
\subsection{The $(\uu,\vv)$-Donaldson-Futaki invariant of a smooth polarized variety}\label{sec 8.1}

Let $(X,L)$ be a smooth compact polarized projective manifold, where $L$ is an ample holomorphic line bundle on $X$ and $\T\subset {\rm Aut}(X,L)$ is an $\ell$-dimensional real torus on the total space of $L$, which covers a torus action (still denoted by $\T$) in $\Autred\cong{\rm Aut}(X,L)/\mathbb{C}^{\star}$. Let $\xxi=(\xi_1,\cdots,\xi_\ell)\in \tor$ be a basis of $\mathbb{S}^{1}$-generators of $\T$ and $\bA_{\xxi}^{(k)}:=(A^{(k)}_{\xi_1},\ldots,A^{(k)}_{\xi_\ell})$ the induced infinitesimal actions of $\xi_i$ on the finite dimensional space $\mathcal{H}_k:=H^{0}(X,L^{k})$ of global holomorphic sections of $L^{k}$ for $k\gg 1$. For a $\T$-invariant Hermitian metric $h$ on $L$ with curvature two form $\omega\in 2\pi c_1(L)$ we have (see e.g. \cite[Proposition 8.8.2]{Gauduchon})
\begin{equation}\label{A-inf}
A^{(k)}_{\xi_i}+\sqrt{-1}\nabla_{\xi_i}= km^{\xi_i}_{\omega}{\rm Id}_{\mathcal{H}_k},
\end{equation}
where $\nabla$ is the Chern connection of $h^k:=h^{\otimes k}$ and $m^{\xi_i}_{\omega}$ is a $\omega$-Hamiltonian function of $\xi_i$. Using the basis $\xxi$ we identify $\tor\cong\mathbb{R}^{\ell}$ and we get a momentum map $m_\omega:=(m_{\omega}^{\xi_1},\cdots,m_{\omega}^{\xi_\ell}):X\to\R^\ell$ for the action of $\T$ on $X$ with momentum image $\PP:=m_\omega(X)$. Notice that if $h_\phi:=e^{-2\phi}h$ is another $\T$-invariant Hermitian metric on $L$ with positive curvature $\omega_\phi>0$, the corresponding momentum map satisfies $m^{\xi_i}_\phi=m_{\omega}^{\xi_i}+(d^c\phi)(\xi_i)$, thus showing, by virtue of \Cref{lem-mom-norm} \ref{m3}, that the image $m_{\phi}(X)=\PP$ is independent of the metric $h_\phi$. We thus have a polytope $\PP\subset \tor^*$ associated to the polarization $(X,L,\T)$.

The spectrum of $k^{-1}A_{\xi_j}^{(k)}$ is given by $\{\lambda^{(k)}_i(\xi_j),\,\lambda^{(k)}_i\in W_k\}$ where $W_k:=\{\lambda^{(k)}_i,\,i=1,\cdots, N_k\}\subset \Lambda^{*}$ is the finite set of weights of the complexified action of $\T$ on $\mathcal{H}_k$ and $\Lambda^{*}$ is the dual of the lattice $\Lambda\subset\tor$ of circle subgroups of $\T$ (see e.g. \cite{AM, B-N}).

\begin{lem}\label{W-P}
The set of weights $W_k$ is contained in the momentum polytope $\PP$ of the action of $\T$ on $(X,L)$.
\end{lem}
\begin{proof}
This Lemma is well known (see e.g. \cite[Section 5]{AM}, we give the proof for the sake of clarity. Let $\lambda^{(k)}_i\in W_k$, $\xi_j\in\xxi$ an $\mathbb{S}^1$-generator for the $\T$-action on $X$, and $s^{(k)}_{j,i}\in\mathcal{H}_k$ an eigensection associated to the eigenvalue $\lambda^{(k)}_i(\xi_j)$ of $k^{-1}A_{\xi_j}^{(k)}$. Using \eqref{A-inf}, we have 
\begin{align*}
(\lambda^{(k)}_i(\xi_j)-m^{\xi_j}_\omega)|s^{(k)}_{j,i}|^{2}_{h^k}=&(k^{-1}A_{\xi_j}^{(k)}s_{j,i},s_{j,i})_{h^{k}}-m^{\xi_j}_\omega|s^{(k)}_{j,i}|^{2}_{h^k}\\
=&-\frac{\sqrt{-1}}{2}(d|s^{(k)}_{j,i}|^{2}_{h^k})(\xi_j)-\frac{1}{2}(d|s^{(k)}_{j,i}|^{2}_{h^k})(J\xi_j).
\end{align*} 
At a point of global maximum $x_0$ of the smooth function $|s^{(k)}_{j,i}|^{2}_{h^k}$ on $X$, we obtain 
$$\lambda^{(k)}_i(\xi_j)=m^{\xi_j}_\omega(x_0)\in \PP.$$ 
It follows that $W_k\subset\PP$.
\end{proof}
Using the weight decomposition of $\mathcal{H}_k$
\begin{equation*}
\mathcal{H}_k=\bigoplus_{\lambda^{(k)}_i\in W_k}\mathcal{H}(\lambda^{(k)}_i),
\end{equation*}
and \Cref{W-P}, for any smooth function $\uu\in C^{\infty}(\PP,\mathbb{R})$ we can define the operator $\uu(k^{-1}\bA_{\xxi}^{(k)}):\mathcal{H}_{k}\rightarrow\mathcal{H}_{k}$ by 
\begin{equation}\label{w(A)-eq}
\uu(k^{-1}\bA_{\xxi}^{(k)})_{\mid \mathcal{H}(\lambda^{(k)}_i)}:=\uu(k^{-1}\lambda^{(k)}_i){\rm Id}_{\mathcal{H}(\lambda^{(k)}_i)}.
\end{equation}
 
\begin{defn}
We define the {\it $\uu$-weight} of the action of $\T$ on $(X,L)$ by
\begin{equation}\label{W-u}
W_{\uu}(L^{k}):=\Tr(\uu(k^{-1}\bA_{\xxi}^{(k)})).
\end{equation}
\end{defn}

\begin{lem}\label{lem-W-u}
The {\it $\uu$-weight} of the action of $\T$ on $(X,L)$ admits the following asymptotic expansion
\begin{equation}\label{w-u-exp}
(2\pi)^{n}W_{\uu}(L^{k})=k^{n}\int_X \uu(m_\omega)\omega^{[n]}+\frac{k^{n-1}}{4}\int_X\Scal_{\uu}(\omega)\omega^{[n]}+\mathcal{O}(k^{n-2}).
\end{equation}
for any smooth function $\uu$ with compact support containing $\PP$. 
\end{lem}

\begin{proof}
This is a direct consequence of \Cref{TYZ} below, by letting $\vv=\uu$ in \eqref{Bergman}, and integrating in both sides over $X$.
\end{proof}
The following result is a straightforward consequence of \Cref{lem-W-u}.
\begin{cor}
Let $(X,L)$ be a smooth polarized projective
variety endowed with a torus action $\T\subset {\rm Aut}(X,L)$ and $\uu,\vv\in C^{\infty}(\PP,\mathbb{R})$ smooth functions on the corresponding polytope $\PP\subset\tor^*$. For any $\mathbb{C}^{*}$-action $\rho$ commuting with $\T$ and a family $\xxi$ of $\mathbb{S}^1$-generators of $\T$, we consider the weight
\begin{equation*}
W_{\uu}^{(k)}(\xxi,\rho):=\Tr\left(\uu\big(k^{-1}\bA^{(k)}_{\xxi}\big)\cdot 
k^{-1}A^{(k)}_{\rho}\right),
\end{equation*}
where $A^{(k)}_{\rho}$ is the induced infinitisimal action of $\rho$ on $\mathcal{H}_k$. Then, $W_{\vv}(\xxi,\rho)$ admits an asymptotic expansion
\begin{equation*}
W_{\uu}^{(k)}(\xxi,\rho)=a_{\uu}^{(0)}(\xxi,\rho)k^{n}+a_{\uu}^{(1)}(\xxi,\rho)k^{n-1}+\mathcal{O}(k^{n-2}),
\end{equation*}
and the $(\uu,\vv)$-Futaki invariant introduced in \Cref{Def-Fut} with respect to the K\"ahler class $\alpha:=2\pi c_1(L)$ satisfies
\begin{equation*}
\frac{1}{4(2\pi)^n}\mathcal{F}^{\alpha}_{\uu,\vv}(V_\rho)=a_{\uu}^{(1)}(\xxi,\rho)-\frac{c_{\uu,\vv}(L)}{4}a_{\vv}^{(0)}(\xxi,\rho),
\end{equation*}
where $V_\rho$ is the generator of the $\mathbb{S}_{\rho}^1$-action on $X$, and $c_{\uu,\vv}(L)$ is the $(\uu,\vv)$-slope of $(X,2\pi c_1(L))$ defined in \eqref{Top-const}.
\end{cor}
\subsection{The $(\uu,\vv)$-Donaldson-Futaki invariant of a polarized test configuration}
Following \cite{Do-02}, we consider a (possibly singular) polarized test configuration of exponent $r\in\mathbb{N}$, compatible with $(X,L,\T)$, defined as follows:
\begin{defn}
A $\T$-{\it compatible polarized test configuration} $(\mathcal{X},\mathcal{L})$ of exponent $r\in\mathbb{N}$ associated to the smooth polarized variety $(X,L)$ is a normal polarized variety $(\mathcal{X},\mathcal{L},\hat{\T})$ endowed with a torus $\hat\T\subset{\rm Aut}(\mathcal{X},\mathcal{L})$ and
\begin{itemize}
\item a flat morphism $\pi:\mathcal{X}\rightarrow \mathbb{P}^{1}$ such that the torus action $\hat\T$ on $\mathcal{X}$ preserves each fiber $X_\tau:=\pi^{-1}(\tau)$, and $(X_1,\mathcal{L}_{|X_1},\hat\T)$ is  equivariantly isomorphic to $(X,L^r,\T)$;
\item a $\mathbb{C}^{\star}$-action $\rho$ on $\mathcal{X}$ commuting with $\hat\T$ and covering the usual $\mathbb{C}^{\star}$-action on $\mathbb{P}^{1}$;
\item an isomorphism
\begin{equation}\label{lambda-1}
\lambda: (X\times(\mathbb{P}^{1}\setminus \{0\}),L^{r}\otimes \mathcal{O}_{\mathbb{P}^1}(r))\cong(\mathcal{X}\setminus X_0,\mathcal{L}),
\end{equation}
which is equivariant with respect to the actions of $\mathbb{G}:=\hat\T\times \mathbb{S}^{1}_{\rho}$ on $\mathcal{X}\setminus X_0$ and the action of $\T\times\mathbb{S}^{1}$ on $X\times\left(\mathbb{P}^{1}\setminus \{0\}\right)$. 
\end{itemize}
\end{defn}

To simplify the discussion, we shall assume in the sequel that $r=1$ and that $L$ is a very ample polarization on $X$. 

By the consideration in \Cref{sec 8.1}, for each $\tau\neq 0$,  $(X_\tau, \mathcal{L}_{|X_\tau}, \hat \T)$ gives rise to a momentum polytope $\PP_\tau \subset \tor^*$. Using the biholomorphism \eqref{lambda-1}, we know that  $(X_\tau, \mathcal{L}_{|X_\tau},\hat  \T)$  and $(X_1, \mathcal{L}_{|X_1}, \hat \T)$ are equivariantly isomorphic polarized varieties, and thus $\PP_\tau =\PP_1=\PP$ for all $\tau \neq 0$.

For any $\tau\in \mathbb{P}^1$, following \Cref{sec 8.1}, we let $\bA^{(k)}_{\xxi}(\tau):=(A^{(k)}_{\xi_1}(\tau),\ldots,A^{(k)}_{\xi_\ell}(\tau))$ be infinitisimal generators of the $\mathbb{S}^{1}$-actions on $\mathcal{H}_k(\tau):=H^0(X_\tau,\mathcal{L}^{k}_{|X_\tau})$, induced by the $\mathbb{S}^{1}$-generators $\xxi=(\xi_1,\cdots,\xi_\ell)$ for the $\hat{\T}$-action, on the fiber $(X_\tau,\mathcal{L}_{|X_\tau})$. We claim that the spectrum of the operators $A^{(k)}_{\xi_j}(\tau)$ is independent of $\tau\in \mathbb{P}^1$, and is contained in $\PP$.  To see this, we can use the observation from \cite[Sect. 2.3]{Do-02} which associates to any $\T$-compatible polarized test configuration $(\mathcal{X}, \mathcal{L}, \hat{\T})$ a continuous family  $\mathcal{V}_k(\tau) \subset {\rm Sym}^k(\mathbb{C}^{N+1})$ of $m$-planes in the Grassmanian ${\rm Gr}_m({\rm Sym}^k(\mathbb{C}^{N+1}))$, where ${\rm Sym}^k$ denotes the vector space of symmetric homogeneous polynomials in  $N +1$ complex variables. In this picture, $(X_\tau, \mathcal{L}_{|X_\tau})$ is seen as a polarized subvariety of $({\mathbb P}^{N},\mathcal{O}(1))$, and the space of sections $\mathcal{H}_k(\tau) := H^0(X_\tau, (\mathcal{L}_{|X_\tau})^k)$ is identified to ${\rm Sym}^k(\mathbb{C}^{N+1}) /\mathcal{V}_k(\tau)$. We can further assume that the action of $\hat{\T}$ on $(X_\tau,\mathcal{L}_{|X_\tau})$ comes from the restriction to $X_\tau$ of a subtorus of $\hat{\T}\subset {\rm SL}(N+1, \mathbb{C})$, and thus $\hat{\T}$ also acts on ${\rm Sym}^k(\mathbb{C}^{N+1})$; furthermore, writing $\widehat{{\bA}}^{(k)}_{\xxi}:=(\hat{A}^{(k)}_{\xi_1},\cdots\hat{A}^{(k)}_{\xi_\ell})$, where $\hat{A}^{(k)}_{\xi_j}$ is the infinitisimal generator of the circle action $\mathbb{S}_{\xi_j}^1$ on ${\rm Sym}^k(\mathbb{C}^{N+1})$, the operators
$$\hat{A}^{(k)}_{\xi_j}: {\rm Sym}^k(\mathbb{C}^{N+1}) \to {\rm Sym}^k(\mathbb{C}^{N+1}),$$ 
must preserve the $m$-planes $\mathcal{V}_k(\tau)$ (as the action preserves each $X_\tau$ viewed as the subspace of common zeroes of elements in $\mathcal{V}_k(\tau)$), and thus 
$$A_{\xi_j}^{(k)}(\tau) : {\rm Sym}^k(\mathbb{C}^{N+1}) /\mathcal{V}_k(\tau) \to {\rm Sym}^k(\mathbb{C}^{N+1}) /\mathcal{V}_k(\tau)$$
are the linear maps induced by $\widehat{{\bA}}^{(k)}_{\xxi}$ on the quotient spaces $\mathcal{H}_k(\tau)$. Introducing a $\hat{\T}$-invariant Hermitian product on ${\rm Sym}^k(\mathbb{C}^{N+1})$, we thus obtain a continuous $\hat{A}^{(k)}_{\xi_j}$-invariant decomposition 
$${\rm Sym}^k(\mathbb{C}^{N+1}) = \mathcal{V}_k(\tau) \oplus \mathcal{V}^{\perp}_k(\tau),$$
and the spectrum of $A_{\xi_i}^{(k)}(\tau)$ is nothing but the spectrum of $\hat{A}_{\xi_i}^{(k)}$ restricted to $\mathcal{V}^{\perp}_k(\tau)$. Using that $\mathcal{V}_{k}^{\perp}(\tau)$ vary continuously in the Gramsannian, we conclude that the spectrum of $\hat{A}^{(k)}_{\xi_j}$ restricted to $\mathcal{V}_{k}^{\perp}(\tau)$ is constant. It is contained in $\PP$ by \Cref{W-P}.

It follows that for any $\uu\in C^{\infty}(\PP,\mathbb{R})$, we can define $\uu(k^{-1}\bA^{(k)}_{\xxi}(0))$, where $\bA^{(k)}_{\xxi}(0)= (A_{\xi_1}^{(k)}(0),\cdots, A_{\xi_{\ell}}^{(k)}(0))$ denote the the generators of circle actions corresponding to the central fibre $(X_0, \mathcal{L}_{|X_0}, \hat{\T})$. Thus, for $\uu\in C^{\infty}(\PP,\mathbb{R})$ we can consider the following $\uu$-weight
\begin{equation}\label{W-0}
W_{\uu}^{(k)}(\xxi,\rho):=\Tr\left(\uu\big(k^{-1}\bA^{(k)}_{\xxi}(0)\big)\cdot 
k^{-1}A^{(k)}_\rho\right).
\end{equation}

\begin{defn}
Let $\uu\in C^{\infty}(\PP,\mathbb{R}_{>0})$ and $\vv\in C^{\infty}(\PP,\mathbb{R})$, and suppose that we have the following asymptotic expansions on the central fiber $(X_0,L_0)$
\begin{align}
\begin{split}\label{star}
W_{\vv}^{(k)}(\xxi,\rho)=& a_{\vv}^{(0)}(\xxi,\rho)k^{n}+\mathcal{O}(k^{n-1}),\\
W_{\uu}^{(k)}(\xxi,\rho)=&a_{\uu}^{(0)}(\xxi,\rho)k^{n}+a_{\uu}^{(1)}(\xxi,\rho)k^{n-1}+\mathcal{O}(k^{n-2}).
\end{split}
\end{align} 
Then we define the $(\uu,\vv)$-{\it Donaldson-Futaki} invariant of the normal $\T$-compatible polarized test configuration $(\mathcal{X},\mathcal{L})$ to be the number
\begin{equation}\label{Don-Fut-Alg}
\DF_{\uu,\vv}(\mathcal{X},\mathcal{L}):=a_{\uu}^{(1)}(\xxi,\rho)-\frac{c_{\uu,\vv}(L)}{4}a_{\vv}^{(0)}(\xxi,\rho),
\end{equation}
where $c_{\uu,\vv}(L)$ is the $(\uu,\vv)$-slope of $(X,2\pi c_1(L))$ given by \eqref{Top-const}.
\end{defn}

Using \Cref{lem-W-u}, we have the following

\begin{cor}\label{DF-smooth}
If $(\mathcal{X},\mathcal{L})$ is a $\T$-compatible polarized test configuration with smooth central fiber, then the expansions \eqref{star} hold, and
\begin{align*}
(2\pi)^{n}W_{\vv}^{(k)}(\xxi,\rho)=& k^{n}\int_{X_0}h_{\rho}\vv(m_{\Omega_0})\Omega^{[n]}_0+\mathcal{O}(k^{n-1}),\\
(2\pi)^{n}W_{\uu}^{(k)}(\xxi,\rho)=&k^{n}\int_{X_0}h_{\rho}\uu(m_{\Omega_0})\Omega^{[n]}_0+\frac{k^{n-1}}{4}\int_{X_0}h_{\rho}\Scal_{\uu}(\Omega_0)\Omega^{[n]}_0+\mathcal{O}(k^{n-2}),
\end{align*} 
where $h_\rho$ is the $\Omega$-Hamiltonian of the generator $V_\rho$ of the action $\mathbb{S}^{1}_\rho$ on $X_0$ with respect to a $\mathbb{G}$ invariant K\"ahler metric $\Omega\in 2\pi c_1(\mathcal{L})$ and $\Omega_0:=\Omega_{|X_0}$. In particular, the $(\uu,\vv)$-Donaldson-Futaki invariant \eqref{Don-Fut-Alg} of $(\mathcal{X},\mathcal{L})$ is given by
\begin{equation*}
\DF_{\uu,\vv}(\mathcal{X},\mathcal{L})=\frac{1}{4(2\pi)^{n}}\mathcal{F}^{\alpha}_{\uu,\vv}(V_\rho),
\end{equation*}
where $\mathcal{F}^{\alpha}_{\uu,\vv}(V_\rho)$ the Futaki invariant of the class $\alpha:=2\pi c_1(L)$, introduced in \Cref{Def-Fut}.
\end{cor}
We deduce from \Cref{DF-smooth} and \Cref{submersion-case}
\begin{cor}\label{DF=F}
If $(\mathcal{X},\mathcal{L})$ is a smooth $\T$-compatible polarized test configuartion such that $\pi:\mathcal{X}\to\mathbb{P}^1$ is a smooth submersion, then 
\begin{equation*}
\DF_{\uu,\vv}(\mathcal{X},\mathcal{L})=\frac{1}{4(2\pi)^{n}}\mathcal{F}_{\uu,\vv}(\mathcal{X},2\pi c_1(\mathcal{L})),
\end{equation*}
where $\mathcal{F}_{\uu,\vv}(\mathcal{X},2\pi c_1(\mathcal{L}))$ is the $(\uu,\vv)$-Futaki invariant of the $\T$-compatible K\"ahler test configuration $(\mathcal{X},2\pi c_1(\mathcal{L}))$ introduced in \Cref{Def-Fut-TC}.
\end{cor}

\section{The $(\uu,\vv)$-Futaki invariant of a toric test configurations}\label{sec-9}

In this section we consider the special case when $X$ is a smooth toric variety i.e. $\T\subset\Autred$ with $\dim_\mathbb{R}\T=\dim_\mathbb{C}X=n$. Let $\omega\in\alpha$ be a fixed $\T$-invariant K\"ahler form, $m_\omega:X\rightarrow\tor^*$ a corresponding momentum map, and $\PP=m_\omega(X)$ the corresponding momentum polytope. By Delzant Theorem \cite{Delzant}, $(X,\alpha)$ can be recovered from the {\it labelled integral Delzant polytope} $(\PP,\Lab)$ where $\Lab=(L_j)_{j=1,d}$ is the collection of non-negative defining affine-linear functions for $\PP$, with $dL_j$ being primitive elements of the lattice $\Lambda$ of circle subgroups of $\T$. We denote by $\PP^{0}$ the interior of $\PP$ and by  $X^{0}:=m^{-1}_\omega(\PP^{0})$ the dense open set of $X$ of points with principle $\T$ orbits. Let us consider the momentum/angle coordinates $(p,t)\in\PP^{0}\times\mathbb{T}$ with respect to the K\"ahler metric $(g,J,\omega)$. By a result of Guillemin (see \cite{guillemin}) 
\begin{align}
\begin{split}\label{g-J}
&g=\langle dp,\bG^{u},dp\rangle+\langle dt,\bH^{u},dt\rangle,\\
&Jdt=-\langle \bG^{u},dp\rangle,\\
&\omega=\langle dp\wedge dt\rangle,
\end{split}
\end{align}
on $X^{0}$, where $u$ is a smooth, strictly convex function called the {\it symplectic potantial} of $(\omega,J)$, $\bG^{u}:\PP^{0}\rightarrow S^{2}\tor$ is the Hessian of $u$, $\bH^{u}:\PP^{0}\rightarrow S^{2}\tor^{*}$ is its point-wise inverse and $\langle\cdot,\cdot,\cdot\rangle$ denote the contraction $\tor^*\times S^{2}\tor\times\tor^{*}\rightarrow\R$ or the dual one. Conversely if $u$ is a strictly convex smooth function on $\PP^{0}$, \eqref{g-J} defines a K\"ahler structure on $X^{0}$ which extends to a global $\T$-invariant K\"ahler structure on $X$ iff $u$ satisfies the boundary conditions of Abreu (see \cite{abreu}). We denote by $\mathcal{S}(\PP,\Lab)$ the set of smooth strictly convex functions on $\PP^{0}$ satisfying these boundary conditions. For $u\in\mathcal{S}(\PP,\Lab)$, we have the following expression for the scalar curvature of $(g,J)$ (see \cite{abreu0}),
\begin{equation*}
\Scal(g)=-\sum_{i,j=1}^{n}\bH_{ij,ij}^{u},
\end{equation*} 
where $\bH^{u}=(\bH_{ij}^{u})$ in a basis of $\tor$. Let $\uu\in C^{\infty}(\PP,\mathbb{R}_{>0})$. By the calculations in \cite[Section 3]{AM}, the following expression for the $\uu$-scalar curvature of $(g,J)$ is straightforward
\begin{equation}\label{Scal-u-tor}
\Scal_{\uu}(g)=-\sum_{i,j=1}^{n}\left(\uu\bH_{ij}^{u}\right)_{,ij}.
\end{equation} 
We recall that by the maximality of $\T$, any $\T$-invariant Killing potential of \eqref{g-J} is the pull-back by $m_\omega$ of an affine-linear function on $\PP$.  
\begin{lem}\label{lem-Fut-tor}
Let $\uu\in C^{\infty}(\PP,\mathbb{R}_{>0})$ and $\vv\in C^{\infty}(\PP,\mathbb{R})$. For any affine-lirear function $f$ on $\PP$, the $(\uu,\vv)$-Futaki invariant corresponding to the $\T$-invariant Hamiltonian Killing vector field $\xi:=df$ is given by
\begin{equation}\label{Fut-tor}
(2\pi)^{-n}\mathcal{F}_{\uu,\vv}^{\alpha}(\xi)=2\int_{\partial\PP}f\uu d\sigma-c_{(\uu,\vv)}(\alpha)\int_{\PP}f\vv dp,
\end{equation}
where $dp$ is a Lebesgue measure on $\tor^*$, $d\sigma$ is the induced measure on each face $F_i\subset\partial\PP$ by letting $dL_i\wedge d\sigma=-dp
$ and the constant $c_{(\uu,\vv)}(\alpha)$ is given by
\begin{equation}\label{Top-const-tor}
c_{(\uu,\vv)}(\alpha)=2\left(\frac{\int_{\partial\PP}\uu d\sigma}{\int_{\PP}\vv dp}\right).
\end{equation}
\end{lem}
\begin{proof}
Let $u\in\mathcal{S}(\PP,\Lab)$ and $(g,J)$ be the corresponding $\omega$-compatible K\"ahler structure $X$ given by \eqref{g-J}. The $(\uu,\vv)$-Futaki invariant of the K\"ahler class $\alpha=[\omega]$ is given by
\begin{equation*}
\mathcal{F}_{\uu,\vv}^{\alpha}(\xi)=\int_X\Scal_{\uu}(g)f(m_\omega)\omega^{[n]}-c_{(\uu,\vv)}(\alpha)\int_{X}f(m_\omega)\vv(m_\omega)\omega^{[n]},
\end{equation*}
where $f$ is an affine linear function on $\tor^*$ with $\xi=df\in \tor$. In the momentum-action coordinates $(p,t)\in\PP^{0}\times\T$ we have $\omega^{[n]}=\langle dp\wedge dt\rangle^{[n]}=dp_1\wedge dt_1\wedge\cdots\wedge dp_n\wedge dt_n$. Then, using \eqref{Scal-u-tor} and \cite[Lemma 2]{AM}, we get
\begin{align*}
(2\pi)^{-n}\mathcal{F}_{\uu,\vv}^{\alpha}(\xi)=&-\int_{\PP}\bigg(\sum_{i,j=1}^{n}\big(\uu\bH_{ij}^{u}\big)_{,ij}\bigg) f dp-c_{(\uu,\vv)}(\alpha)\int_{\PP}f\vv dp\\
=&2\int_{\partial\PP}f\uu d\sigma-c_{(\uu,\vv)}(\alpha)\int_{\PP}f\vv dp.
\end{align*} 
Similarly we deduce \eqref{Top-const-tor}.
\end{proof}
For any $f\in C^{0}(\PP,\mathbb{R})$ we define 
\begin{equation}\label{Eq-Fut-P}
\mathcal{F}^{\PP}_{\uu,\vv}(f):=2\int_{\partial\PP}f\uu d\sigma-c_{(\uu,\vv)}(\alpha)\int_{\PP}f\vv dp.
\end{equation}
Using again \cite[Lemma 2]{AM} we obtain 
\begin{equation}\label{eq-f-x}
(2\pi)^{-n}\int_X(\Scal_{\uu}(g_u)-c_{\uu,\vv}(\alpha)\vv(m_\omega))f\omega^{[n]}=\mathcal{F}^{\PP}_{\uu,\vv}(f)-\int_{\PP}\Bigg(\sum_{i,j=1}^{n}\bH_{ij}f_{,ij}\Bigg)\uu dp,
\end{equation}
for any $u\in\mathcal{S}(\PP,\Lab)$ and $f\in C^{\infty}(\PP,\mathbb{R})$. It follows that
\begin{lem}\cite{AM,Do-02}
If there exist $u\in \mathcal{S}(\PP,\Lab)$ such that the corresponding $\omega$-compatible K\"ahler structure $(g,J)$ solves $\Scal_{\uu}(g)=c_{(\uu,\vv)}(\alpha)\vv(m_\omega)$ then $\mathcal{F}^{\PP}_{\uu,\vv}(f)\geq 0$ for any smooth convex function $f$ on $\PP$.
\end{lem}

\subsection{Toric test configuration}\label{sec-9.1}
We start by recalling the construction of toric test configurations introduced by Donaldson in \cite[Section 4]{Do-02}. Let $(X,L)$ be a smooth polarized toric manifold with integral momentum polytope $\PP\subset\tor^{*}\cong\mathbb{R}^{n}$ with respect to the lattice $\mathbb{Z}^n\subset\mathbb{R}^n$ and 
\begin{equation}\label{PL-func}
f:=\max(f_1,\cdots,f_r),
\end{equation}
a convex piece-wise affine-linear function with integer coefficients, i.e. we assume that each $f_j$ in \eqref{PL-func} is an affine-linear function $f_j(p):=\langle v_j,p\rangle+\lambda_j$ with $v_j\in \mathbb{Z}^n$ and $\lambda_j\in\mathbb{Z}$. We also assume that the polytope $\mathrm{Q}$ defined by
\begin{equation}\label{Q-testconf}
\mathrm{Q}=\{(p,\pprim)\in\PP\times \mathbb{R}\,:\,0\leq \pprim\leq R-f(p)\},
\end{equation}
has integral vertices in $\mathbb{Z}^{n+1}$, where $R$ is an integer such that $f\leq R$ on $\PP$. By \cite[Proposition 4.1.1]{Do-02} there exist an $(n+1)$-dimensional projective toric variety $(\mathcal{X}_{\mathrm{Q}},\mathbb{G})$ and a polarization $\mathcal{L}_{\mathrm{Q}}\rightarrow \mathcal{X}_{\mathrm{Q}}$ corresponding to the labelled integral Delzant polytope $\mathrm{Q}\subset\mathbb{R}^{n+1}$ and the lattice $\mathbb{Z}^{n+1}\subset\mathbb{R}^{n+1}$. In general, $\mathcal{X}_{\mathrm{Q}}$ is a compact toric orbifold (see \cite{LT}), but $\mathcal{X}_{\mathrm{Q}}$ can be smooth for a suitable choice of $f$. There is an embedding $\iota:X\hookrightarrow \mathcal{X}_{\mathrm{Q}}$ such that $\iota(X)$ is the pre-image of the face $\PP=\mathrm{Q}\cap\big(\mathbb{R}^{n}\times\{0\}\big)$ of $\mathrm{Q}$, and the restriction of $\mathcal{L}_{\mathrm{Q}}$ to $\iota(X)$ is isomorphic to $L$. Notice that by the Delzant Theorem \cite{Delzant, LT} the stabilizer of $\iota(X)\subset\mathcal{X}_{\mathrm{Q}}$ in $\mathbb{G}$ is $\mathbb{S}^{1}_\rho=\mathbb{S}^{1}_{(n+1)}$, where $\mathbb{S}^{1}_{(n+1)}$ is the $(n+1)$-th factor of $\mathbb{G}=\mathbb{R}^{n+1}/2\pi\mathbb{Z}^{n+1}$ so that $\mathbb{G}/\mathbb{S}^{1}_\rho$ is identified with the torus action $\mathbb{T}=\mathbb{R}^{n}/2\pi\mathbb{Z}^{n}$ on $X$. Furthermore, Donaldson shows in \cite{Do-02} that there exist a $\mathbb{C}^{\star}$-equivariant map $\pi:\mathcal{X}_{\mathrm{Q}}\rightarrow \mathbb{P}^{1}$ such that $(\mathcal{X}_{\mathrm{Q}},\mathbb{S}_{\rho}^1,\mathcal{L}_{\mathrm{Q}})$ is a $\T$-compatible polarized test configuration. We consider the Futaki-invariant $\mathcal{F}_{\uu,\vv}(\mathcal{X}_{\mathrm{Q}},2\pi c_1(\mathcal{L}_{\mathrm{Q}}))$ given by \eqref{Fut-TC} corresponding to $(\mathcal{X}_{\mathrm{Q}},2\pi c_1(\mathcal{L}_{\mathrm{Q}}))$, and notice that it makes sense even when $\mathcal{X}_{\mathrm{Q}}$ is an orbifold. 
\begin{prop}\label{DF=F-toric}
Let $f=\max(f_1,\cdots,f_r)$ be a convex piece-wise linear function on $\PP$, with integer coefficients and $\mathcal{X}_{\mathrm{Q}}$ the toric test configuration constructed as above. Then the $(\uu,\vv)$-Futaki invariant \eqref{Fut-TC} of $(\mathcal{X}_{\mathrm{Q}},2\pi c_1(\mathcal{L}_{\mathrm{Q}}))$ is given by
\begin{equation}\label{eq-fut-111}
\mathcal{F}_{\uu,\vv}(\mathcal{X}_{\mathrm{Q}},2\pi c_1(\mathcal{L}_{\mathrm{Q}}))=(2\pi)^{n+1}\mathcal{F}^{\PP}_{\uu,\vv}(f),
\end{equation}
where $\mathcal{F}^{\PP}_{\uu,\vv}(f)$ is the integral defined in \eqref{Eq-Fut-P}. Furthermore, the $(\uu,\vv)$-Donaldson-Futaki invariant \eqref{Don-Fut-Alg} corresponding to $(\mathcal{X}_{\mathrm{Q}},\mathcal{L}_{\mathrm{Q}})$ is well-defined, and is given by
\begin{equation}\label{eq-fut-222}
\DF_{\uu,\vv}(\mathcal{X}_{\mathrm{Q}},\mathcal{L}_{\mathrm{Q}})=4\mathcal{F}^{\PP}_{\uu,\vv}(f).
\end{equation}
\end{prop}

\begin{proof}
We start by proving the first claim \eqref{eq-fut-111}. Let $\Omega\in 2\pi c_1(\mathcal{L}_{\mathrm{Q}})$ be a $\mathbb{G}$-invariant K\"ahler form on $\mathcal{X}_{\q}$ and $\omega\in 2\pi c_1(L)$ be the induced $\mathbb{T}$-invariant K\"ahler form on $\iota(X)\subset\mathcal{X}_{\q}$. We have by \Cref{rem-Fut-sym} \ref{rem-Fut-sym-ii}
\begin{align}
\begin{split}\label{eq-fut-scal}
\mathcal{F}_{\uu,\vv}(\mathcal{X},2\pi c_1(\mathcal{L}_{\mathrm{Q}}))=&-\int_{\mathcal{X}}\big(\Scal_{\uu}(\Omega)-c_{(\uu,\vv)}(2\pi c_1(L))\vv(m_{\Omega})\big)\Omega^{[n+1]}\\&+(8\pi)\int_X\uu(m_{\omega})\omega^{n}.
\end{split}
\end{align}
Let $(p,\pprim,t,\tprim)\in \mathrm{Q}\times\T\times\mathbb{S}^{1}_\rho$ be the momentum/angular coordinates on $\mathcal{X}_{\mathrm{Q}}^0$ such that $(p,t)\in \mathrm{P}\times\T$ are the momentum/angular coordinates on $X^0$. Then,
\begin{equation}\label{z1}
(8\pi)\int_X\uu(m_{\omega})\omega^{n}=4(2\pi)^{n+1}\int_{\PP}\uu(p)dp.
\end{equation}
and 
\begin{equation}\label{z2}
\int_{\mathcal{X}_{\mathrm{Q}}} \vv(m_{\Omega})\Omega^{[n+1]}=(2\pi)^{n+1}\int_{\mathrm{Q}} \vv(p)dp \wedge d\pprim=(2\pi)^{n+1}\int_{\PP} \vv(p)(R-f(p))dp.
\end{equation}
For the remaining term in \eqref{eq-fut-scal}, using \eqref{eq-f-x} we have
\begin{align}
\begin{split}\label{z3}
(2\pi)^{-(n+1)}&\int_{\mathcal{X}_{\mathrm{Q}}}\Scal_{\uu}(\Omega)\Omega^{[n+1]}=2\int_{\partial\mathrm{Q}}\uu d\sigma_{\mathrm{Q}}\\
=&2\int_{\mathrm{P}}\uu dp+2\int_{(R-f)(\PP)}\uu d\mu_{(R-f)(\PP)}+2\int_{\partial \PP}(R-f)\uu d\sigma_{\PP}\\
=&4\int_{\PP}\uu dp+2\int_{\partial \PP}(R-f)\uu d\sigma_{\PP},
\end{split}
\end{align}
where the measure $d\mu_{(R-f)(\PP)}$ is defined by $df\wedge d\mu_{(R-f)(\PP)}=dp\wedge d\pprim$. Substituting \eqref{z1}--\eqref{z3} into \eqref{eq-fut-scal} yields
\begin{align*}
(2\pi)^{-(n+1)}\mathcal{F}_{\uu,\vv}(\mathcal{X}_{\mathrm{Q}},2\pi c_1(\mathcal{L}_{\mathrm{Q}}))=& -2\int_{\partial \PP}(R-f)\uu d\sigma_{\PP}+c_{\uu,\vv}(\alpha)\int_{\mathrm{P}}(R-f)\vv dp\\=&\mathcal{F}^{\PP}_{\uu,\vv}(f).
\end{align*}

Now we give the proof of the second claim \eqref{eq-fut-222}. The central fiber $X_0$ is the reduced divisor on $\mathcal{X}_{\mathrm{Q}}$ associated to the preimage of the union of facets of $\mathrm{Q}$ corresponding to the graph of $R-f$. By a well-known fact in toric geometry (see e.g. \cite{Do-02}) the set of weights for the complexified torus $\mathbb{G}^{c}$ on $H^{0}(\mathcal{X},\mathcal{L}^{k}_{\mathrm{Q}})$ is $k\mathrm{Q}\cap\mathbb{Z}^{n+1}$. It thus follows that the weights for the $\mathbb{C}^{\star}_\rho$-action on $H^{0}(X_0,L_{0}^{k})$ are $k(R-f)(k\PP)\cap \mathbb{Z}$. We conclude that 
\begin{equation*}
W_{\uu}^{(k)}(\xxi,\rho)=\sum_{\lambda\in k\PP\cap\mathbb{Z}^{n}}(R-f)\Big(\frac{\lambda}{k}\Big)\uu\Big(\frac{\lambda}{k}\Big),
\end{equation*}
where $W_{\uu}^{(k)}(\xxi,\rho)$ is the $\uu$-weight defined by \eqref{W-0}. By \cite{Guil-Stern-1,Zel}, for any smooth function $\Phi$ on $\tor^*$ and $k$ large enough we have
\begin{equation*}
\sum_{\lambda\in k\PP\cap\mathbb{Z}^{n}}\Phi\Big(\frac{\lambda}{k}\Big)=k^{n}\int_{\PP}\Phi dp+\frac{k^{n-2}}{2}\int_{\partial\PP}\Phi d\sigma_{\PP}+\mathcal{O}(k^{n-2}).
\end{equation*}
Taking $\Phi:=(R-f)\uu$ and using the above formula for any affine-linear piece of $\Phi$, we get
\begin{equation*}
W_{\uu}^{(k)}(\xxi,\rho)=k^{n}\int_{\PP}(R-f)\uu dp+\frac{k^{n-2}}{2}\int_{\partial\PP}(R-f)\uu d\sigma_{\PP}+\mathcal{O}(k^{n-2}).
\end{equation*}
Analogously, for $W_{\vv}^{(k)}(\xxi,\rho)$ we obtain
\begin{equation*}
W_{\vv}^{(k)}(\xxi,\rho)=k^{n}\int_{\PP}(R-f)\vv dp+\mathcal{O}(k^{n-1}).
\end{equation*}
Using \eqref{Don-Fut-Alg}, it follows that
\begin{equation*}
\DF_{\uu,\vv}(\mathcal{X}_{\mathrm{Q}},\mathcal{L}_{\mathrm{Q}})=4\mathcal{F}^{\PP}_{\uu,\vv}(f).
\end{equation*}
\end{proof}

\begin{rem}
Instead of a convex piece-wise affine-linear function $f$ with integer coefficients we can take a convex piece-wise affine-linear functions with rational differentials, i.e. assuming that each $f_j$ in \eqref{PL-func} is of the form with $f_j(p)=\langle v_j,p\rangle+\lambda_j$ with $v_j\in \mathbb{Q}^n$. The polytope $\mathrm{Q}$ such a function defines is not longer with rational vertices, but still defines a toric K\"ahler orbifold $(\mathcal{X}_{\mathrm{Q}},\mathcal{A}_{\mathrm{Q}})$, see \cite{LT}. This gives rise to a toric K\"ahler test configuration compatible with $\T$ and the formula \eqref{eq-fut-111} in \Cref{DF=F-toric} computes the corresponding $(\uu,\vv)$-Futaki invariant of $(\mathcal{X}_{\mathrm{Q}},\mathcal{A}_{\mathrm{Q}})$.
\end{rem}

\section{The $(\uu, \vv)$-Futaki invariant of rigid semisimple toric fibrations}\label{sec-10} This is the case relevant to  the example (iv) from the Introduction.  Following \cite{ACGT}, we consider $X= V\times_{\T} K \xrightarrow{\pi} B$ to be the total space of a fibre-bundle associated to a principle $\T$-bundle  $K \to B$ over the product  $B = \prod_{j=1}^N (B_j, \omega_j, g_j)$ of compact cscK manifolds $(B_j, \omega_j, g_j)$ of complex dimension $d_j$,  satisfying the Hodge condition $[\omega_j/2\pi] \in H^2(B_j, \Z)$,  and a compact  $2\ell$-dimensional toric K\"ahler manifold $(V, \omega_V,  g_V, J_V, \T)$ corresponding to a labelled Delzant polytope $(\PP, {\bf L})$  in $\tor^*$.  We assume that $K$ is endowed with a connection $1$-form ${\boldsymbol \theta} \in \Omega^1(K, \tor)$ satisfying
\begin{equation*}
d{\boldsymbol \theta} = \sum_{j=1}^N {\xi}_j \otimes \omega_j,  \  {\xi}_j \in  \tor, \  j=1, \cdots,  N. 
\end{equation*}
and that the toric K\"ahler metric $(g_V, \omega_V, J_V)$ on $V$ is given by \eqref{g-J} for a symplectic potential $u \in {\mathcal S}(\PP, {\bf L})$.   As shown in \cite{ACGT}, $X$ admits a bundle-adapted K\"ahler  metric  $(g, \omega)$ which, on the open dense subset $X^0= K \times \PP^0 \subset X$,  takes  the form
\begin{equation}\label{generalized-calabi}
\begin{split}
g &= \sum_{j=1}^N \Big(\langle \xi_j, p \rangle + c_j\Big)\pi^* g_j  + \langle dp,  {\boldsymbol G}^{u}, dp \rangle + \langle \boldsymbol{\theta},  {\boldsymbol H}^u,  \boldsymbol{\theta} \rangle, \\
\omega &= \sum_{j=1}^N \Big(\langle \xi_j, p \rangle + c_j\Big)\pi^* \omega_j  + \langle dp \wedge \boldsymbol{\theta} \rangle,
\end{split}
\end{equation}
where $p \in \PP^0$ and $c_j$ are real constants such that $(\langle \xi_j, p \rangle + c_j)>0$ on $\PP$. Such K\"ahler metrics, parametrized by $u\in {\mathcal S}(\PP, \bf L)$ and the real constants $c_j$,  are referred to in \cite{ACGT} as given by {\it the generalized Calabi ansatz} in reference to the well-known construction of Calabi~\cite{calabi} of extremal K\"ahler metrics on ${\mathbb P}^1$-bundles.

We notice that the K\"ahler manifold $(X, \omega, g)$ is invariant under the $\T$-action with momentum map identified with $p \in \PP$. Furthermore,  it is shown in \cite[(7)]{ACGT} that the scalar curvature of \eqref{generalized-calabi} is given by
\begin{equation*}
\begin{split}
\Scal(g)  &= \sum_{j=1}^N \frac{\Scal_j}{\langle \xi_j, p\rangle + c_j} - \frac{1}{\w(p)} \sum_{r, s=1}^{\ell} \frac{\partial^2}{\partial p_r \partial p_s} \Big(\w(p) {\boldsymbol H}^u_{rs}\Big) \\
              &=  \sum_{j=1}^N \frac{\Scal_j}{\langle \xi_j, p\rangle + c_j}  + \frac{1}{\w(p)} \Scal_{\w}(g_V),
              \end{split}
\end{equation*}
where we have put $\w(p) := \prod_{j=1}^N(\langle \xi_j , p \rangle + c_j)^{d_j}$ and we have used \eqref{Scal-u-tor} for passing from the first line to the second. Similarly, by \cite[equation (12)]{ACGT}, the $g$-Laplacian of (the pull-back to $X$) of a smooth function $f(p)$ on $\PP$ is given by
              \begin{equation*}
              \Delta_g f = - \frac{1}{\w(p)} \sum_{r,s=1}^{\ell}\frac{\partial}{\partial p_r}\Big(\w(p) \frac{\partial f}{\partial p_s} {\boldsymbol H}^u_{rs}\Big).
              \end{equation*}
Using the above formulae, we check by a direct computation that for any positive smooth function $\uu$ on $\PP$  we have
\begin{equation}\label{u-scalar}
\Scal_{\uu}(g)  = \uu(p)\Big(\sum_{j=1}^N \frac{\Scal_j}{\langle \xi_j, p\rangle + c_j}\Big)  + \frac{1}{\w(p)} \Scal_{\w \uu}(g_V)
\end{equation}
Using that the volume form of \eqref{generalized-calabi} is 
\begin{equation*} 
\omega^{[n]} = \w(p) \Big(\bigwedge_{j=1}^{N} \omega_j^{[d_j]}\Big) \wedge \langle dp \wedge \boldsymbol{\theta}\rangle^{[\ell]},  
\end{equation*}
and the integration by parts formula  \cite[Lemma 2]{AM}, we compute that the $(\uu,\vv)$-Futaki invariant on $X$ acts on a vector field $\xi \in \tor$ by
\begin{equation}\label{(u,v)-Futaki-bundle}
\begin{split}
 \frac{{\mathcal F}^{[\omega]}_{\uu, \vv}(\xi)}{(2\pi)^{\ell} \Big(\prod_{j=1}^N {\rm Vol}(B_j, [\omega_j])\Big) } =&  2 \int_{\partial \PP} f  \uu \w d\sigma  +  \int_{\PP}\Big(\sum_{j=1}^N \frac{\Scal_j}{\langle \xi_j, p\rangle + c_j}\Big) f \uu\w dp\\
  & - c_{\uu, \vv}([\omega]) \int_{\PP} f \vv \w dp,
 \end{split}
\end{equation}
where $f= \langle \xi, p \rangle + \lambda$ is a Killing potential of $\xi$.

\smallskip
As in \Cref{sec-9.1}, we can construct a $\T$-compatible smooth K\"ahler test configuration  associated to $X$ defined by a convex piece-wise linear function $f= {\rm max}(f_1, \cdots, f_k)$ on $\tor^*$ such that the polytope $\q \subset \R^{\ell+1}$ given by \eqref{Q-testconf} is Delzant with respect to the the lattice $\Z^{\ell+1}$. Denote by $(\mathcal V_{\q}, \mathcal A_{\q})$ the corresponding smooth toric variety, and by $\mathcal K= K \times {\mathbb S}^1_{(\ell+1)} \to B$ the principal $\T^{\ell+1}$-bundle over $B$ with trivial $(\ell+1)$-factor, and let $\mathcal X= \mathcal V \times _{\T^{\ell+1}} \mathcal K \to B$ be the resulting $\mathcal V$-bundle over $B$.  We can now consider a K\"ahler form $\Omega$ on $\mathcal X$ obtained by the generalized Calabi ansatz~\eqref{generalized-calabi}; as the connection $1$-form on $\mathcal K$ has a curvature $\sum_{j=1}^N \xi_j \otimes \omega_j$ with  $\xi_j  \in \tor= {\rm Lie}(\T^{\ell}) \subset {\rm Lie}(\T^{\ell+1})$,  $\Omega$  induces on the pre-image $X \subset \mathcal X$ of the facet $\PP\subset \q$ a K\"ahler form $\omega$ given by \eqref{generalized-calabi} with the same affine linear functions $(\langle \xi_j, p\rangle + c_j)$. A similar computation to \eqref{(u,v)-Futaki-bundle}, performed on the total space $(\mathcal X_{\q}, \Omega)$ by using Definition 11 (see also the proof of \Cref{lem-Fut-tor} above) leads to the expression \eqref{(u,v)-Futaki-bundle} for the $(\uu,\vv)$-Futaki invariant associated to $(\mathcal X_{\q}, \mathcal A_{\q})$ with $f$ being the piece-wise linear convex function defining $\q$.

\smallskip  Let us now  suppose that  $X = {\mathbb P}(\cO \oplus \cL) \xrightarrow{\pi} B$  with $B$ as above,  where $\cO$ stands for the trivial holomorphic line bundle over $B$ and $\cL$ is a holomorphic line bundle of the form  $\cL = \bigotimes_{j=1}^N \cL_j$ for  $\cL_j$  being the pull-back to $B$ of a holomorphic line bundle over $B_j$ with $c_1(\cL_j) = \xi_j [\omega_j/2\pi]$, $\xi_j \in \Z$.  This is the so-called {\it admissible setting} (without blow-downs) of \cite{ACGT3}, pioneered in \cite{calabi} and studied in many works. In our setting above, such an $X$  is a $\mathbb P^1$-bundle obtained from the principle ${\mathbb S}^1$-bundle over $B$ associated to $\cL^{-1}$. 
We can take $\PP=[-1,1] \subset \R$, and suppose that $\uu(z)>0$ and $\vv(z)$ are smooth functions  defined over $[-1,1]$. A K\"ahler metric  $(\omega, g)$ on $X$ of the form \eqref{generalized-calabi} can be equivalently written as
\begin{equation}\label{calabi}
\begin{split}
g &= \sum_{j=1}^N (\xi_jz+ c_j)\pi^* g_j  +  \frac{dz^2}{\Theta(z)}+ \Theta(z)\theta^2 \\
\omega &= \sum_{j=1}^N (\xi_j z + c_j)\pi^* \omega_j  + dz\wedge \theta, \ d\theta = \sum_{j=1}^N \xi_j \pi^*\omega_j,
\end{split}
\end{equation}
for positive affine-linear functions $\xi_j z + c_j$  on $[-1,1]$. This is the more familiar Calabi ansatz, written in terms of the {\it  profile function} $\Theta(z)$ (see e.g. \cite{Hwang-Singer}) which must be smooth on $[-1,1]$ and satisfy
\begin{equation}\label{boundary}
\Theta(\pm 1)=0, \ \ \Theta'(\pm 1) = \mp 2,
\end{equation}
and 
\begin{equation}\label{positive}
\Theta(z)>0 \ \ \textrm{on} \ \ (-1,1),
\end{equation}
for \eqref{calabi} to define a smooth K\"ahler metric on $X$. We let  $\w(z)= \prod_{j=1}^N (\xi_j z + c_j)^{d_j}$ be the corresponding polynomial in $z$.

We now take  $\q$  be the chopped rectangle with base $\PP$, corresponding to the convex piece-wise affine linear function $f_{z_0}(z) = {\rm max}(z+1-z_0, 1)$ where $z_0\in (-1,1)$ is a given point. We can construct as above an ${\mathbb S}^1$-compatible K\"ahler test configuration $(\mathcal X_{\q}, \mathcal A_\q)$ associated to $(X, [\omega], {\mathbb S}^1)$. It is not difficult to see that the complex manifold $\mathcal X_{\q}$ is the degenaration to the normal cone with respect to the infinity section $S_{\infty}\subset X$, see \cite{Ross-Thomas, ACGT3} but the K\"ahler class ${\mathcal A}_{\q}$ on $\mathcal X_{\q}$ defines a polarization only  for rational values of $z_0$. Formula \eqref{(u,v)-Futaki-bundle} shows that the  $(\uu, \vv)$-Futaki invariant of $(\mathcal X_{\q}, \mathcal A_\q)$ is a positive multiple of the quantity
\begin{equation}\label{futaki-z}
\begin{split}
F(z_0) &:= 2\Big(f_{z_0}(1)\uu(1)\w(1) - f_{z_0}(-1)\uu(-1)\w(-1) \Big) \\
            & + \int_{-1}^{1} f_{z_0}(z)\Big(\uu(z)\w(z)\Big(\sum_{j=1}^N \frac{\Scal_j}{\xi_jz+ c_j}\Big) - c_{\uu,\vv}([\omega])\vv(z)\w(z)\Big) dz .
\end{split}
\end{equation}
Let us now assume that there exists a smooth function $\Theta(z)$ on $[-1,1]$, which satisfies  \eqref{boundary} and 
\begin{equation}\label{formal-solution}
\big(\uu\w \Theta\big)''(z) =   \uu(z)\w(z) \Big(\sum_{j=1}^N \frac{\Scal_j}{\xi_jz+ c_j}\Big)  - c_{\uu,\vv}([\omega]) \vv(z) \w(z).  
\end{equation}
Substituting in the RHS of \eqref{futaki-z} and integrating by parts over the intervals $[-1, z_0]$ and $[z_0, 1]$  gives
\begin{equation}\label{F-z-0}
F(z_0)= \uu(z_0) \w(z_0) \Theta(z_0).
\end{equation}
As $\uu(z)$ and $\w(z)$ are positive functions on $[-1,1]$, we conclude that if $(X, [\omega], {\mathbb S}^1)$ is $(\uu,\vv)$-K-stable on smooth $\mathbb{S}^1$-compatible  K\"ahler test configurations with reduced central fibre, then $\Theta(z)$ must also satisfy \eqref{positive}. By the formula \eqref{u-scalar}, the corresponding K\"ahler metric \eqref{calabi} will be then $(\uu,\vv)$-cscK. 

The existence of a solution of \eqref{formal-solution}  satisfying \eqref{boundary} is in general overdetermined.  Following \cite{AMT},  in the case when $\vv(z)>0$ on $[-1,1]$ one can resolve the over-determinacy by introducing an affine-linear function $\vv_{\rm ext}(z)= A_1z + A_2$, such that  
\begin{equation}\label{formal-solution-ext}
\big(\uu\w \Theta\big)''(z) =   \uu(z)\w(z) \Big(\sum_{j=1}^N \frac{\Scal_j}{\xi_jz+ c_j}\Big)  -   \vv(z) \vv_{\rm ext}(z) \w(z) 
\end{equation}
admits a unique solution $\Theta_{\rm ext}^{\uu,\vv}(z)$ satisfying \eqref{boundary}:  the coefficients $A_1$ and $A_2$, as well as two constants of integration in \eqref{formal-solution}, are then uniquely determined from the four boundary conditions in \eqref{boundary}. Furthermore, a straightforward generalization of \cite[Lemma 2.4]{AMT} shows that $\vv_{\rm ext}(z)$ corresponds to the affine-linear  function  introduced in Section 3.2, i.e. solutions of \eqref{formal-solution-ext} introduce $(\uu,\vv\vv_{\rm ext})$-cscK metrics of Calabi type, which are equivalently $(\uu,\vv)$-extremal. The methods of this article allow us to obtain the following generalization of \cite[Theorem~3]{AMT}.
\begin{thm}\label{p:calabi-type} Let $X= {\mathbb P}(\cO  \oplus \cL) \to B$ be a projective ${\mathbb P}^1$-bundle as above, endowed with the ${\mathbb S}^1$-action by multiplication on $\cO$,  and $\alpha=[\omega/2\pi]$ be the K\"ahler class of a K\"ahler metric in the form  \eqref{calabi}. We let $\PP=[-1,1]$ be the momentum polytope of $(X, \alpha, {\mathbb S}^1)$,  $\uu, \vv$ be smooth positive functions on $[-1,1]$ and $\Theta_{\rm ext}^{\uu,\vv}(z)$ the unique solution  of \eqref{formal-solution-ext} satisfying \eqref{boundary}. Then, 
\begin{enumerate}
\item[$\bullet$] If $(X, \alpha, \Sph^1)$ is $(\uu, \vv\vv_{\rm ext})$-K-stable on $\Sph^1$-compatible smooth K\"ahler test configurations with reduced central fibre, then  $\Theta_{\rm ext}^{\uu,\vv}(z)>0$ on $(-1,1)$ and $\alpha$ admits a $(\uu,\vv)$-extremal K\"ahler metric of the form \eqref{calabi} with $\Theta= \Theta^{\uu,\vv}_{\rm ext}$.
\item[$\bullet$] If $(X, \alpha, \Sph^1)$ admits a $(\uu,\vv)$-extremal K\"ahler metric, then $(X, \alpha, \Sph^1)$ is $(\uu, \vv\vv_{\rm ext})$-K-semistable on $\Sph^1$-compatible smooth K\"ahler test configurations with reduced central fibre, and $\Theta_{\rm ext}^{\uu,\vv}(z) \ge 0$.
\end{enumerate}
\end{thm}
\begin{proof}
The first part follows from the identity \eqref{F-z-0} which shows that $\Theta^{\uu, \vv}_{\rm ext}$ must satisfy both \eqref{boundary} and \eqref{positive}.  The second part follows from formula \eqref{F-z-0} and \Cref{thm:main}, if the constants $(c_1, \ldots, c_N)$ in \eqref{calabi} are rational as in this case the corresponding K\"ahler class $\alpha$ is rational. To treat the case when $(c_1,\ldots, c_N)$ are not necessarily rational, we can use \Cref{thm:LeBrun-Simanca} below (with fixed $\uu, \vv$ and varying the constants $c_j$). Accordingly, for any rational constants $(\tilde c_1, \ldots, \tilde c_N)$ sufficiently close to $(c_1, \ldots, c_N)$ the corresponding K\"ahler class $\tilde \alpha$ will admit a $(\uu,\vv)$ extremal K\"ahler metric, and hence the corresponding function $\widetilde \Theta_{\rm ext}^{\uu,\vv}(z)$ will be non-negative on $(-1,1)$ by virtue of \Cref{thm:main}. As $\Theta_{\rm ext}^{\uu,\vv}(z)$ depends smoothly on $(c_1, \ldots, c_N)$, it follows that $\Theta_{\rm ext}^{\uu,\vv} (z) \ge 0$ too.
\end{proof}

\begin{rem} (i) As already mentioned in the Introduction, we expect  that Theorem 2 can be improved by showing that the existence of $(\uu, \vv)$-cscK metric in $\alpha$ implies $(\uu,\vv)$-K-stability, not only $(\uu, \vv)$-K-semi-stability.  Accordingly, we expect Theorem~\ref{p:calabi-type} to be improved to a complete Yau--Tian--Donaldson type correspondence between $(\uu, \vv\vv_{\rm ext})$-K-stable and $(\uu,\vv)$-extremal K\"ahler classes on $X$ of the form \eqref{calabi}, in which either notion corresponds to the positivity condition \eqref{positive} for $\Theta^{\uu,\vv}_{\rm ext}(z)$.

(ii) In \cite{AMT},  the analogous statement of Theorem~\ref{p:calabi-type} is achieved by considering polarized test configuration $(\mathcal X_\q, \mathcal L_\q)$ as above (corresponding to rational values of $z_0$), and computing the relative version of the algebraic $(\uu, \vv)$-Donaldson--Futaki invariant ${\rm DF}_{\uu,\vv}(\mathcal X_{\q}, \mathcal L_{\q})$. This provides a yet another instance where the differential-geometric definition coincides with the algebraic definition of  the $(\uu,\vv)$-Futaki invariant.

\end{rem}

\appendix

\section{The $(\uu,\vv)$-equivariant Bergman kernels and boundedness of the $(\uu,\vv)$-Mabuchi energy}
Let $(X,L)$ be a polarized manifold, $\alpha=2\pi c_1(L)$ the corresponding K\"ahler class, and $\T\subset {\rm Aut}(X,L)$ a real $\ell$-dimensional torus with momentum polytope $\PP$ as in Section 8.1. Let $h$ be a $\T$-invariant Hermitian metric on $L$ with curvature $2$-form $\omega\in\alpha$. We identify the space of $\T$-invariant Hermitian metrics $h_\phi:=e^{-2\phi}h$ with positive curvature forms $\omega_\phi$ with the space $\mathcal{K}^{\T}_\omega$ of $\T$-invariant K\"ahler potentials $\phi$ on $X$.\\

Let $\xxi:=(\xi_1,\cdots,\xi_\ell)$ be a chosen basis of $\mathbb{S}^{1}$-generators of $\T$ and $\bA_{\xxi}^{(k)}:=(A^{(k)}_{\xi_1},\ldots,A^{(k)}_{\xi_\ell})$ the induced infinitesimal actions of $\xi_i$ on the space $\mathcal{H}_k$ given by \eqref{A-inf}. For $\uu\in C^{\infty}(\PP,\mathbb{R}_{>0})$ we consider the following weighted $L^2$-inner product on $C^{\infty}(X,L^{k})$ 
\begin{equation}\label{L2-inner}
\langle s, s^{\prime}\rangle_{\uu,k\phi}:=k^{n}\int_X (s,s^\prime)_{k\phi} \uu(m_{\phi})\omega^{[n]}_{\phi}.
\end{equation}
where $(s,s^\prime)_{k\phi}:=h^{k}_\phi(s,s^\prime)$. The operators $(A_{\xi_j}^{(k)})_{j=1,\cdots,\ell}$ are Hermitian with respect to $\langle\cdot,\cdot\rangle_{\uu,k\phi}$, with spectrum contained in the momentum polytope  $\PP$ (see \cref{{W-P}}).

\begin{defn}\cite{B-N,Sz-book, ZZ}
Let  $\phi\in \mathcal{K}^{\T}_\omega$, $\{s_i\mid i=0,\cdots,N_k\}$ be a $\langle \cdot,\cdot\rangle_{\uu,k\phi}$-orthonormal basis of $\mathcal{H}$ and $\vv\in C^{\infty}(\PP,\mathbb{R})$. Then the $(\uu,\vv)$-equivariant Bergman kernel of the Hermitian metric $h^{k}_\phi$ on $L^{k}$, is the function defined on $X$ by,
\begin{equation}\label{Bergman}
B_{\vv}(\uu,k\phi):=\uu(m_{\phi})\sum_{i=0}^{N_k}\Big(\vv\big(k^{-1}\bA_{\xxi}^{(k)}\big)(s_i),s_i\Big)_{k\phi}.
\end{equation}
where $\vv\big(k^{-1}\bA_{\xxi}^{(k)}\big)$ is given by \eqref{w(A)-eq}.
\end{defn}

Equivalently, $B_{\vv}(\uu,k\phi)$ is the restriction to the diagonal $\{x=x^{\prime}\}\subset X\times X$ of the Schwartz kernel of the operator $\vv\big(k^{-1}\bA_{\xxi}^{(k)}\big)\Pi_{\uu}^{k\phi}$, where $\Pi_{\uu}^{k\phi}:L^{2}(X,L^{k})\rightarrow \mathcal{H}_k$ denote the orthogonal projection with respect to the inner product $\langle\cdot,\cdot\rangle_{\uu,k\phi}$ (see \cite{ZZ}).\\

\smallskip
Asymptotic expansions of \eqref{Bergman} in $k\gg 1$ are known to exist in many special cases, see e.g. \cite{B-N, ma-ma, Sz-book, ZZ}. We need the following general result, which follows essentially from \cite{charles}, with a ramification from \cite{ma-ma}.

\begin{prop}\label{prop-Charl}
Let $(T^{(k)}_j)_{j=1,\cdots,\ell}$ be a family of $\langle\cdot,\cdot\rangle_{\uu,k\phi}$-self adjoint commuting Toeplitz operators, such that the set of joint eigenvalues of $(T^{(k)}_j)_{j=1,\ell}$ is contained in $\PP$. Suppose that the symbol of $T^{(k)}_j$, $j=1,\cdots,\ell$ is given by
\begin{equation*}
\sigma(T^{(k)}_j):=\sum_{i\geq0} \hbar^{i} f^{(j)}_i\in C^{\infty}(X)[[\hbar]].
\end{equation*}
Then for any smooth function $\vv$ with compact support containing $\PP$, the operator $\vv(T^{(k)}_1,\cdots,T^{(k)}_\ell)$ is a Toeplitz operator with symbol
\begin{equation*}\label{sigma-u-TT}
\sigma(\vv(T^{(k)}_1,\cdots,T^{(k)}_\ell))=s_0(\uu,\vv)+s_1(\uu,\vv)\hbar+\mathcal{O}(\hbar^{2}),
\end{equation*}
where $s_0(\uu,\vv),s_1(\uu,\vv)$ are given by
\begin{align*}
s_0(\uu,\vv)=&\vv(f^{(1)}_0,\cdots,f^{(\ell)}_0),\\
s_1(\uu,\vv)=&\vv(f^{(1)}_0,\cdots,f^{(\ell)}_0)S_{\uu}(\phi)+\sum_{j=1}^{\ell} \vv_{,j}(f^{(1)}_0,\cdots,f^{(\ell)}_0)(f^{(j)}_1-f^{(j)}_0S_{\uu}(\phi))\\&+\frac{1}{4}\sum_{i,j=1}^{\ell}\vv_{,ij}(f^{(1)}_0,\cdots,f^{(\ell)}_0)(df^{(i)}_0,df^{(j)}_0)_\phi,
\end{align*}
with $S_{\uu}(\phi):=\frac{1}{4}\big(\Scal_\phi+2\Delta_\phi(\log(\uu(m_\phi)))\big)$.
\end{prop}

\begin{proof}
In the case of one $\langle\cdot,\cdot\rangle_{\uu,k\phi}$-self adjoint Toeplitz operator $T^{(k)}$ and a smooth function of one variable $\vv$, the fact that $\vv(T^{(k)})$ is again a Toeplitz operator is established in \cite[Proposition 12]{charles}. The proof given in \cite{charles} relies on the Helffer-Sjostrand formula, see e.g. \cite[Theorem 8.1]{DSJ}. Using its multivariable generalization \cite[Equation 8.18]{DSJ} the proof in \cite{charles} readily generalizes to show that $\vv(T^{(k)}_1,\cdots,T^{(k)}_\ell)$ is a Toeplitz operator for any smooth function on $\PP$, and family of $\langle\cdot,\cdot\rangle_{\uu,k\phi}$-self adjoint commuting Toeplitz operators $(T^{(k)}_j)_{j=1,\cdots,\ell}$ such that the set of joint eigenvalues of $(T^{(k)}_j)_{j=1,\ell}$ is contained in $\PP$. 

Following \cite{charles}, one can define a map $\sigma:\mathcal{T} \rightarrow C^{\infty}(X)[[\hbar]]$ called {\it the full symbol map} from the set $\mathcal{T}$ of Toeplitz operators, to the set of formal series $C^{\infty}(X)[[\hbar]]$, which we introduce as follows. For a general Toeplitz operator $T^{(k)}\in\mathcal{T}$, the restriction to the diagonal $T^{(k)}(x,x)$ of its Shwartz kernel admits an asymptotic expansion in $C^{\infty}$ of the form,
\begin{equation*}
(2\pi)^n T^{(k)}(x,x)=\sum_ik^{-i}a_i(x)+\mathcal{O}(k^{-\infty}),
\end{equation*}
and the symbol of $T^{(k)}$ is given  by the formal series 
\begin{equation*}
\sigma(T^{(k)}):=\sum_i\hbar^{i}a_i(x).
\end{equation*}
The full symbol map $\sigma$ is a morphism from the the algebra $(\mathcal{T},\circ)$ with respect to the composition $\circ$ of operators to the algebra $(C^\infty(X)[[\hbar]],\star_{\uu})$ endowed with an associative product $\star_\uu$ called the {\it star product}
\begin{equation*}
\big(\sum_{j\geq0}f_j\hbar^j\big)\star_{\uu}\big(\sum_{j\geq0}g_j\hbar^j\big):=\sum_{s\geq0}\big(\sum_{j=0}^sf_j\star_{\uu}g_{s-j}\big)\hbar^s,
\end{equation*}
where the product $f\star_\uu g$ for two smooth functions is also an element of $C^{\infty}(X)[[\hbar]]$ which we define as follows. 

We denote by $\Pi_{\uu}^{k\phi}$ the $\langle\cdot,\cdot\rangle_{\uu,k\phi}$-orthogonal projection on $\mathcal{H}_k$ with respect to the weighted $L^2$-inner product \eqref{L2-inner}. Using \cite[Theorem 0.2]{ma2012berezintoeplitz}, for any $f,g\in C^\infty(X,\mathbb{R})$, we have $\Pi_{\uu}^{k\phi}f\Pi_{\uu}^{k\phi}g\Pi_{\uu}^{k\phi}\in\mathcal{T}$, and the restriction to the diagonal $\{x=x^{\prime}\}\subset X\times X$ of its Schawrtz kernel admits a $C^\infty$-asymptotic expansion given by
\begin{equation*}
\big(\Pi_{\uu}^{k\phi}f\Pi_{\uu}^{k\phi}g\Pi_{\uu}^{k\phi}\big)(x,x)=fg+\big[\frac{1}{2}(df,dg)_\phi-S_{\uu}(\phi)fg\big]k^{-1}+\mathcal{O}(k^{-2}).
\end{equation*} 
We then let 
\begin{equation*}
f\star_{\uu}g:=\sigma\big(\Pi_{\uu}^{k\phi}f\Pi_{\uu}^{k\phi}g\Pi_{\uu}^{k\phi}\big)=fg+\hbar\left[\frac{1}{2}(df,dg)_\phi-S_{\uu}(\phi)fg\right]+\mathcal{O}(\hbar^2).
\end{equation*}
The star product considered in \cite{charles} is $\star_1$ but the theory extends without difficulty for arbitrary weight $\uu$, by our definition above. For instance, the unit $1_{\star_\uu}$ of $(C^{\infty}(X)[[\hbar]],\star_{\uu})$ is identified with the symbol $1_{\star_{\uu}}:=\sigma(\Pi^{k\phi}_{\uu})$. Let $\Pi_{\uu}^{k\phi}(x,x^{\prime})$ be the usual Bergman kernel associated to the projection $\Pi_{\uu}^{k\phi}$, i.e.
\begin{equation}\label{Berg}
\Pi_{\uu}^{k\phi}(x,x)=\uu(m_{\phi}(x))\sum_{i=0}^{N_k}\left|s_i\right|_{k\phi}^{2}(x).
\end{equation}
We claim that \eqref{Berg} admits the following $C^{\infty}$-asymptotic expansion in $k$ 
\begin{equation}\label{exp-ker}
(2\pi)^n\Pi_{\uu}^{k\phi}(x,x)=1+\frac{1}{k}S_{\uu}(\phi)+\mathcal{O}\left(\frac{1}{k^{2}}\right),
\end{equation}
so that $1_{\star_\uu}=1+\hbar S_{\uu}(\phi)+\mathcal{O}(\hbar^2)$. The expansion  \eqref{exp-ker} follows from \cite[Theorem 4.1.1]{ma-ma} by taking the twisting bundle $E:=X\times \mathbb{C}$ (in the
notation of \cite{ma-ma}) being the trivial line bundle over $X$, endowed with Hermitian metric $|\cdot|_E:=\uu(m_\phi)|\cdot|$, where $|\cdot|$ is the norm of $\mathbb{C}$. According to \cite[Theorem 4.1.1]{ma-ma}, \eqref{Berg} admits an asymptotic expansion \eqref{exp-ker} with coefficient before $k^{-1}$ being equal to 
\begin{equation*}
\frac{1}{4}(\Scal_\phi+2\Lambda_{\phi}F_{E})=\frac{1}{4}\big(\Scal_\phi+2\Delta_{\phi}(\log(\uu(m_\phi)))\big)=S_{\uu}(\phi), 
\end{equation*}
where $F_E=-dd^{c}\log(\uu(m_\phi))$ is the curvature form of the Chern connection of $(E,|\cdot|_{E})$.

We shall now compute the symbol of the Toeplitz operator $\vv(T^{(k)}_1,\cdots,T^{(k)}_\ell)$. Following \cite{charles}, the full calculus of the symbol of $\vv(T^{(k)}_1,\cdots,T^{(k)}_\ell)$ is given by the Taylor series expansion of $\vv$ at the point $\ba:=(f^{(1)}_0(x),\cdots,f^{(\ell)}_0(x))$, $\ba_j:=f^{(j)}_0(x)$ as follows: 
\begin{align}
\begin{split}\label{xx-1}
&\sigma(\vv(T^{(k)}_1,\cdots,T^{(k)}_\ell))=\vv(\ba)1_{\star_{\uu}}(x)+\sum_{j=1}^{\ell} \vv_{,j}(\ba)\left(\sum_{i\geq0} \hbar^{i} f^{(j)}_i(y)-\ba_j1_{\star_{\uu}}(y)\right)_{\mid y=x}\\
&+\frac{1}{2!}\sum_{p,q=1}^{\ell}\vv_{,pq}(\ba)\left(\sum_{i\geq0} \hbar^{i} f^{(p)}_i(y)-\ba_p1_{\star_{\uu}}(y)\right)\star_{\uu}\left(\sum_{i\geq0} \hbar^{i} f^{(q)}_i(y)-\ba_q1_{\star_{\uu}}(y)\right)_{\mid y=x}+\cdots
\end{split}
\end{align}
On the other hand, we compute 
\begin{align*}
&\Big(\sum_{i\geq0} \hbar^{i} f^{(p)}_i(y)-\ba_p1_{\star_{\uu}}(y)\Big)\star_{\uu}\Big(\sum_{i\geq0} \hbar^{i} f^{(q)}_i(y)-\ba_q1_{\star_{\uu}}(y)\Big)_{\mid y=x}\\
=&\left((f^{(p)}_0(y)-\ba_p)+\hbar(f^{(p)}_1(y)-S_{\uu}(y))\right)\star_{\uu}\left((f^{(q)}_0(y)-\ba_q)+\hbar(f^{(q)}_1(y)-S_{\uu}(y))\right)_{\mid y=x}+\mathcal{O}(\hbar^{2})\\
=&(f^{(p)}_0(y)-\ba_p)\star_{\uu}(f^{(q)}_0(y)-\ba_q)_{\mid y=x}+\hbar(f^{(p)}_0(x)-\ba_p)(f^{(q)}_1(x)-S_{\uu}(x))\\&+\hbar(f^{(q)}_0(x)-\ba_q)(f^{(p)}_1(x)-S_{\uu}(x))+\mathcal{O}(\hbar^{2})\\
=&\frac{\hbar}{2}(df^{(p)}_0,df^{(q)}_0)_\phi+\mathcal{O}(\hbar^{2}).
\end{align*}
Substituting back in \eqref{xx-1}, we obtain the claimed formula for the symbol $\sigma(\vv(T^{(k)}_1,\cdots,T^{(k)}_\ell))$ up to $\mathcal{O}(\hbar^2)$.
\end{proof}

\begin{thm}\label{TYZ}
Let $\vv\in C^\infty(\PP,\mathbb{R})$. The $(\uu,\vv)$-equivariant Bergman kernel of the $\T$-invariant Hermitian metric $h^{k}_\phi$ on $L^{k}$ admits an asymptotic expansion when $k\gg 1$, given by
\begin{equation*}
\begin{split}
(2\pi)^{n}B_{\vv}(\uu,k\phi)=\begin{cases}
\vv(m_{\phi})+\mathcal{O}(\frac{1}{k}),\\
 \uu(m_{\phi})+\frac{1}{4k}\Scal_{\uu}(\phi)+\mathcal{O}(\frac{1}{k^{2}}),
 &\mbox{if } {\vv=\uu.} 
 \end{cases}.
\end{split}
\end{equation*}
Moreover, the above expansions holds in $C^{\infty}$, i.e. for any integer $\ell\geq 0$ there exist a constant $C_\ell(\uu,\vv)>0$ such that,
\begin{align*}
\begin{split}
&\left\Vert(2\pi)^{n}B_{\vv}(\uu,k\phi)-\vv(m_{\phi})\right\Vert_{C^{\ell}}\leq \frac{C_{\ell}(\uu,\vv)}{k},\\
&\left\Vert(2\pi)^{n}B_{\uu}(\uu,k\phi)-\uu(m_{\phi})-\frac{1}{4k}\Scal_{\uu}(\phi)\right\Vert_{C^{\ell}}\leq \frac{C_{\ell}(\uu,\uu)}{k^{2}}.
\end{split}
\end{align*}
\end{thm}

\begin{proof}
Since the symbol map $\sigma$ is surjective with kernel given by the ideal of negligible Toeplitz operators $\mathcal{O}(k^{-\infty})\cap \mathcal{T}$ (see \cite[Proposition 3]{charles}), it suffices to calculate $\sigma\big(\vv\big(k^{-1}\bA_{\xxi}^{(k)}\big)\Pi_{\uu}^{k\phi}\big)$. We consider the special case of self-adjoint Toeplitz operators $T^{(k)}_j:=k^{-1}A^{(k)}_{\xi_j}\Pi^{k\phi}_{\uu}$.  We have  
\begin{equation*}
T^{(k)}_j(x,x)=\uu(m_\phi)\sum_{i=0}^{N_k}\big(k^{-1}A_{j}^{(k)}s_i,s_i\big)_{k\phi}.
\end{equation*}
By a straightforward calculation using \eqref{A-inf} the symbol of $T^{(k)}_j$ is given by
\begin{align*}
\sigma(T^{(k)}_j)=m_{\phi}^{\xi_j}+\Big[m_{\phi}^{\xi_j}S_{\uu}(\phi)-\frac{1}{2}\sum_{i=1}^{\ell}(\log\circ\uu)_{,i}(m_{\phi})(\xi_i,\xi_j)_\phi\Big]\hbar+\cdots
\end{align*}
Using \Cref{sigma-u-TT} we get
\begin{equation*}
\sigma(\vv(\bA_{\xxi}^{(k)})\Pi_{\uu}^{k\phi})=s_0(\uu,\vv)+s_1(\uu,\vv)\hbar+\cdots
\end{equation*}
where
\begin{align*}
s_0(\uu,\vv)=&\vv(m_\phi),\\
s_1(\uu,\uu)=&\uu(m_\phi)S_{\uu}(\phi)-\frac{1}{2}\sum_{i,j=1}^{\ell} \frac{\uu_{,i}(m_\phi)\uu_{,j}(m_\phi)}{\uu(m_\phi)}(\xi_i,\xi_j)_\phi+\frac{1}{4}\sum_{i,j=1}^{\ell}\uu_{,ij}(m_\phi)(\xi_i,\xi_j)_\phi.
\end{align*}
Replacing $S_{\uu}(\phi)$ by its expression in \Cref{prop-Charl} we obtain $s_1(\uu,\uu)=\Scal_{\uu}(\phi)$. 
\end{proof}

\subsection{The quantization maps}

Let $W_k$ denote the set of weights of the complexified action of $\T$, for $\lambda^{(k)}_i\in W_k$. We consider the following direct sum decomposition of the space $\mathcal{B}^{\T}(\mathcal{H}_k)$ of $\T$-invariant positive definite Hermitian forms on $\mathcal{H}_k$,
\begin{equation*}
\mathcal{B}^{\T}(\mathcal{H}_k):=\bigoplus_{\lambda^{(k)}_i\in W_k}\mathcal{B}^{\T}(\mathcal{H}(\lambda^{(k)}_i)),
\end{equation*}
where $\mathcal{B}^{\T}(\mathcal{H}(\lambda^{(k)}_i))$ is the space of $\T$-invariant positive definite Hermitian forms on $\mathcal{H}(\lambda^{(k)}_i)$
\begin{defn}\label{Quant-map}
Let $\uu\in C^{\infty}(\PP,\mathbb{R}_{>0})$, $\vv\in C^{\infty}(\PP,\mathbb{R})$. We introduce the following quantization maps:
\begin{enumerate}
\item The $(\uu,\vv)$-Hilbert map ${\rm Hilb}_{\uu,\vv}^{k}:\mathcal{K}^{\T}_\omega\rightarrow\mathcal{B}^{\T}(\mathcal{H}_k)$ which associates to every $\T$-invariant K{\"a}hler potential, the $\T$-invariant Hermitian inner product on $\mathcal{H}_k$, given by 
\begin{equation*}\label{Hilb}
\big({\rm Hilb}_{\uu,\vv}^{k}(\phi)\big)(\cdot,\cdot):=\sum_{\lambda^{(k)}_i\in W_k}\frac{\left(\langle \cdot,\cdot\rangle_{\uu,k\phi}\right)_{\mid \mathcal{H}_k(\lambda^{(k)}_i)}}{\uu(\lambda^{(k)}_i)-\frac{c_{(\uu,\vv)}(\alpha)}{4k}\vv(\lambda^{(k)}_i)},
\end{equation*}
where $c_{(\uu,\vv)}(\alpha)$ is given by \eqref{Top-const}

\item The $(\uu,\vv)$-Fubini--Study map ${\rm FS}^{k}_{\uu,\vv}:\mathcal{B}^{\T}(\mathcal{H}_k)\rightarrow\mathcal{K}^{\T}_\omega$ given by
\begin{equation*}
{\rm FS}^{k}_{\uu,\vv}(H):=\frac{1}{2k}\log{\left(\sum_{i=0}^{N_k}|s_i|_{h^{k}}^{2}\right)}-\frac{\log (c_k(\uu,\vv))}{2k},
\end{equation*}
where $\{s_i\}$ is an adapted $H$-orthonormal basis of $\mathcal{H}_k$ and $c_k(\uu,\vv)$ is a constant given by:
\begin{equation}\label{C_k}
c_k(\uu,\vv):=\frac{1}{k^n\int_X \uu(m_{\omega})\omega^{[n]}}\Big[W_{\uu}(L^{k})-\frac{c_{(\uu,\vv)}(\alpha)}{4k}W_{\vv}(L^{k})\Big],
\end{equation}
with $W_{\uu}(L^{k})$ the $\uu$-weight of the action of $\T$ on $L^{k}$ given by \eqref{W-u}.
\end{enumerate}
\end{defn}
\Cref{TYZ} yields
\begin{lem}\label{lem-dev-rho}
For $\phi\in\mathcal{K}^{\T}_\omega$, the Bergman kernel $\rho_{\uu,\vv}(k\phi)$ of $\Hilb_{\uu,\vv}(k\phi)$ satisfies
\begin{equation*}
\rho_{\uu,\vv}(k\phi)=B_{\uu}(\uu,k\phi)-\frac{c_{(\uu,\vv)}(\alpha)}{4k}B_{\vv}(\uu,k\phi),
\end{equation*}
and it has an asymptotic expansion,
\begin{equation*}
(2\pi)^{n}\rho_{\uu,\vv}(k\phi)=\uu(m_{\phi})+\frac{1}{4k}\left(\Scal_{\uu}(\phi)-c_{(\uu,\vv)}(\alpha)\vv(m_{\phi})\right)+\mathcal{O}\left(\frac{1}{k^{2}}\right).
\end{equation*}
The above asymptotic expansion holds in $C^{\infty}$, i.e. for any integer $\ell\geq 0$ we have, 
\begin{equation*}
\left\Vert(2\pi)^{n}\rho_{\uu,\vv}(k\phi)-\uu(m_{\phi})-\frac{1}{4k}\left(\Scal_{\uu}(\phi)-c_{(\uu,\vv)}(\alpha)\vv(m_{\phi})\right)\right\Vert_{C^{\ell}}\leq \frac{C_{\ell}(\uu,\vv)}{k^{2}}.
\end{equation*}
where $C_{\ell}(\uu,\vv)>0$.
\end{lem}

\begin{comment}
For $\phi\in \mathcal{K}^{\T}_\omega$ we have:
\begin{eqnarray*}
{\rm FS}^{k}_{\uu,\vv}\circ {\rm Hilb}_{\uu,\vv}(k\phi)&=&\phi+\frac{1}{2k}\log{\left(\frac{\rho_{\uu,\vv}(k\phi)}{c_k(\uu,\vv)\uu(m_{\phi})}\right)}\\
&=&\phi+\mathcal{O}(k^{-2}).
\end{eqnarray*}
\end{comment}
Following \cite{donaldson0, ST, Zhang}, we give the following definition
\begin{defn}
We say that a metric $\phi\in\mathcal{K}^{\T}_\omega$ is $(\uu,\vv)$-balanced of order $k$ if it satisfies:
\begin{equation*}
{\rm FS}^{k}_{\uu,\vv}\circ {\rm Hilb}_{\uu,\vv}^{k}(\phi)=\phi.
\end{equation*}
or equivalently 
\begin{equation*}
\rho_{\uu,\vv}(k\phi)=c_k (\uu,\vv)\uu(m_{\phi}),
\end{equation*}
where $c_k (\uu,\vv)$ is given by \eqref{C_k}. 
\end{defn}

Similarly to \cite{donaldson0} we have
\begin{prop}
Let $(\phi_j)_{j\geq0}$ be a sequence in $\mathcal{K}^{\T}_\omega$ such that every $\phi_j$ is a $(\uu,\vv)$-balanced metric of order $j$ and $\phi_j$ converge in $C^{\infty}$ to $\phi$. Then $\omega_\phi$ is $(\uu,\vv)$-cscK metric.
\end{prop}

\subsection{Proof of \Cref{thm:mabuchi-bounded}}
Here outline the proof of \Cref{thm:mabuchi-bounded} from the Introduction, which follows from a straightforward extension of the arguments of \cite{donaldson1, lahdili2, li, ST} that has been given in the special cases \ref{1i}, \ref{ii}, \ref{iii}. We start by defining the finite dimensional analogues of the $(\uu,\vv)$-Mabuchi energy defined on the spaces ${\rm FS}^{k}_{\uu,\vv}\left(\mathcal{B}^{\T}(\mathcal{H}_k)\right)$ and $\mathcal{B}^{\T}(\mathcal{H}_k)$ as follows,
\begin{align}
\begin{split}\label{L-Z}
\mathcal{L}^{k}_{\uu,\vv}&:=\mathcal{E}_{\uu,\vv}^{k }\circ \Hilb^{k}_{\uu,\vv}+2k^{n+1}c_k(\uu,\vv)\mathcal{E}_{\uu},\\
Z^{k}_{\uu,\vv}&:=2k^{n+1}c_k(\uu,\vv)\mathcal{E}_{\uu} \circ \fs^{k}_{\uu,\vv}+\mathcal{E}^{k}_{\uu,\vv},
\end{split}
\end{align}
where $\mathcal{E}_{\uu}$ is given by \eqref{E} and $\mathcal{E}_{\uu,\vv}^{k}:\mathcal{B}^{\T}(\mathcal{H}_k)\rightarrow \mathbb{R}$ is the function
\begin{equation*}\label{i}
\mathcal{E}_{\uu,\vv}^{k}(H)=\sum_{\lambda^{(k)}_i\in W_k}\Big(\uu(\lambda^{(k)}_i)-\frac{c_{(\uu,\vv)}(\alpha)}{4k}\vv(\lambda^{(k)}_i)\Big)\log{\left(\det H_{\lambda^{(k)}_i}\right)}.
\end{equation*}
Using \Cref{lem-dev-rho} one can show that
\begin{align}
\begin{split}\label{(a)}
&\underset{k\rightarrow\infty}{\lim}\Big[\frac{2}{k^{n}}\mathcal{L}^{k}_{\uu,\vv}+b_k\Big]= \mathcal{M}_{\uu,\vv},\\
&\underset{k\rightarrow\infty}{\lim}k^{-n}\Big[\mathcal{L}_{\uu,\vv}^{k}(\phi)-Z_{\uu,\vv}^{k}\circ \Hilb_{\uu,\vv}^{k}(\phi)\Big]=0,
\end{split}
\end{align} 
where the convergence holds in the $C^{\infty}$-sense. Suppose that $\mathcal{K}^{\T}_\omega$ contains a $(\uu,\vv)$-cscK metric $\phi^{*}\in \mathcal{K}^{\T}_\omega$. One can show as in \cite{li, ST, lahdili2} that the metrics $\Hilb_{\uu,\vv}^{k}(\phi^{*})$ are almost balanced in the sense that there exists a smooth function $\varepsilon_\phi(k)$, such that $\underset{k\rightarrow\infty}{\lim}\varepsilon_\phi(k)=0$ in $C^{\ell}(X,\mathbb{R})$ and, 
\begin{equation}\label{(b)}
k^{-n}Z_{\uu,\vv}^{k}\circ \Hilb_{\uu,\vv}^{k}(\phi)\geq k^{-n}Z_{\uu,\vv}^{k}\circ \Hilb_{\uu,\vv}^{k}(\phi^{*})+\varepsilon_\phi(k).
\end{equation}
for all $\phi\in\mathcal{K}^{\T}_\omega$. Using \eqref{(a)} and \eqref{(b)} the proof is identical to the one of \cite[Theorem 3.4.1]{ST}.
\subsection{Proof of \Cref{c:extremal}}
This is a direct consequence of \Cref{rem-Mab-rel} and \Cref{thm:mabuchi-bounded}.

\section{The structure of the automorphism group and stability of the $(\uu,\vv)$-cscK metrics under deformations}
The following result follows from straightforward calculations along the lines of \cite[Section 2.5]{Gauduchon} and is left to the reader.
\begin{lem}\cite{Gauduchon, lahdili1}\label{lem-str}
Let $(X,\alpha,\T)$ be a compact K\"ahler manifold with K\"ahler class $\alpha$ and $\T\subset\Autred$ a real torus with $\PP$ the $\T$-momentum image of $X$. Suppose that $\uu,\vv\in C^{\infty}(\PP,\mathbb{R}_{>0})$ are positive smooth functions on $\PP$. Then we have,
\begin{enumerate}
\item For any $\T$-invariant K\"ahler metric $\omega\in \alpha$ and any variation $\dot{\phi}\in T_\phi\mathcal{K}^{\T}_\omega$ we have
\begin{equation}\label{var-Scal-u-v}
{\boldsymbol \delta}\Big(\frac{\Scal_{\uu}(\omega)}{\vv(m_\omega)}\Big)(\dot{\phi})=-2\mathbb{L}^{\omega}_{\uu,\vv}(\dot{\phi})+d^{c}\dot{\phi}(\Xi_{\uu,\vv}),
\end{equation}
where $\Xi_{\uu,\vv}:=J\grad_g\left(\frac{\Scal_{\uu}(\omega)}{\vv(m_\omega)}\right)$ and $\mathbb{L}^{\omega}_{\uu,\vv}$ is the elliptic fourth order differential operator given by 
$$\mathbb{L}^{\omega}_{\uu,\vv}\dot{\phi}=\frac{\delta\delta\big(\uu(m_\omega)(D^{-}d)\dot{\phi}\big)}{\vv(m_\omega)},$$
$D$ is the Levi-Civita connection of $\omega$ and $D^-(d\dot{\phi})$ is the $J$-anti-invariant part of the tensor $D(d\dot{\phi})$.

\item For any $f\in C^{\infty}(X,\mathbb{R})^{\T}$ we have
\begin{equation*}
\mathcal{L}_{\Xi_{\uu,\vv}}f=\frac{-2\delta\delta\big(\uu(m_\omega)(D^{-}d^{c})f\big)}{\vv(m_\omega)}.
\end{equation*}

\item The operator $\mathbb{L}_{\uu,\vv}$ acting on $C^{\infty}(X,\mathbb{C})^{\T}$ admits a decomposition $\mathbb{L}_{\uu,\vv}=\mathbb{L}_{\uu,\vv}^++\mathbb{L}_{\uu,\vv}^-$ with $\mathbb{L}_{\uu,\vv}^\pm :=\mathbb{L}_{\uu,\vv}\pm\frac{\sqrt{-1}}{2}\mathcal{L}_{\Xi_{\uu,\vv}}$.

\item Let $V=\grad_g(h)+J\grad_{g}(f)$ with $h,f\in C^{\infty}(X,\mathbb{R})^{\T}$.   Then $V\in\mathfrak{h}_{\rm red}^{\T}$ if and only if $\mathbb{L}_{\uu,\vv}^+(h+if)=0$.
\end{enumerate}
\end{lem}

The next Theorem is established using \Cref{lem-str} and the same arguments as in the proof of \cite[Theorem 3.4.1]{Gauduchon} and \cite{FO1}.

\begin{thm}\cite{FO1,lahdili1}\label{thm-str}
If $X$ admits a $(\uu,\vv)$-extremal K\"ahler metric with $\uu,\vv\in C^{\infty}(\PP,\mathbb{R}_{>0})$. Then the complex Lie algebra of $\T$-equivariant automorphisms of $X$ admits the following decomposition
\begin{equation}\label{h-split}
\mathfrak{h}^{\T}= \big(\mathfrak{a}\oplus\mathfrak{k}_{\rm ham}^{\T}\oplus J\mathfrak{k}_{\rm ham}^{\T}\big)\oplus\left(\bigoplus_{\lambda>0}\mathfrak{h}^{\T}_{(\lambda)}\right),
\end{equation} 
where $\mathfrak{a}$ is the abelian Lie algebra of parallel vector fields, $\mathfrak{k}_{\rm ham}^{\T}$ is the real Lie algebra of $\T$-equivariant Hamiltonian isometries of $X$ and $\mathfrak{h}^{\T}_{(\lambda)}$, $\lambda>0$ denote the subspace of elements $V\in\mathfrak{h}^{\T}$ such that $\mathcal{L}_{\Xi_{\uu,\vv}}V=\lambda JV$. Moreover, the Lie algebra of $\T$-equivariant isometries of $X$ admits the following decomposition
\begin{equation}\label{k-split}
\mathfrak{k}^{\T}=\mathfrak{a}\oplus \mathfrak{k}_{\rm ham}^{\T}.
\end{equation}
\end{thm}

Using \Cref{thm-str} and the arguments in the proof of \cite[Theorem 3.5.1]{Gauduchon} we get the following generalization of the structure theorem for the group of holomorphic automorphisms of a $(\uu,\vv)$-extremal K\"ahler manifold.

\begin{cor}\cite{calabi,FO1,lahdili1}\label{thm:Calabi} 
Let $(X,\omega,g)$ be a compact $(\uu,\vv)$-extremal K\"ahler manifold $\omega$ with $\uu,\vv\in C^{\infty}(\PP,\mathbb{R}_{>0})$. Then the group ${\rm Isom}^{\T}_0(X,g)$ of $\T$-equivariant isometries of $X$ is a maximal compact connected subgroup of the identity component of the $\T$-equivariant automorphisms ${\rm Aut}^{\T}_0(X)$ of $X$. In particular, $(g,\omega)$ is invariant under the action of a maximal torus $\T_{\max}$ in ${\rm Aut}_{\rm red}(X)$. Furthermore, if $(g,\omega)$ is $(\uu,\vv)$-cscK, then ${\rm Aut}_{0}^{\T}(X)$ is a reductive complex Lie group.
\end{cor}

Now we consider the stability of the $(\uu,\vv)$-extremal metrics under deformations of the K\"ahler class $\alpha$ and the weight functions $\uu,\vv\in C^{\infty}(\PP,\mathbb{R}_{>0})$.\\
\smallskip

Let $X$ be a compact K\"ahler manifold, $\alpha$ a K\"ahler class, $\T_{\rm max}\subset\Autred$ a maximal torus and $\PP_{\alpha}\subset\tor^{*}$ a momentum polytope for $\alpha$ as in \Cref{lem-mom-norm}. Let $\beta\in H^{1,1}(X)$ and $U$ an open subset of $\tor^*$ with $\PP_\alpha\subset U$. Then there exist $a>0$ such that for any $|r|<a$ we can choose $\PP_{\alpha+r\beta}\subset U$ to be the momentum polytope of $\T_{\rm max}$ with respect to $\alpha+r\beta$. With these choices, we now suppose that $\uu,\vv$ are positive smooth functions on $U$ and $\tilde\uu,\tilde\vv$ are arbitrary smooth functions on $U$. We then have

\begin{thm}\label{thm:LeBrun-Simanca} 
Suppose that $\omega\in\alpha$ is a $\T_{\rm max}$-invariant $(\uu,\vv)$-extremal K\"ahler metric associated to $(\PP_\alpha,\uu,\vv)$. Then there exist $\varepsilon >0$, such that for any $|s|<\varepsilon, |t|<\varepsilon$, $|r|<\varepsilon$, there exists a $(\uu+t\tilde{\uu} ,\vv+s\tilde{\vv})$-extremal K\"ahler metric in the K\"ahler class $\alpha + r\beta$, associated to $(\uu+t\tilde{\uu} ,\vv+s\tilde{\vv})$ and $\PP_{\alpha+r\beta}\subset U$.
\end{thm}

The proof follows the lines of that for LeBrun--Simanca stability theorem in \cite{LS} (see also \cite{FS}).

\end{document}